\documentclass{amsart}
\usepackage{amsthm,amsmath,amssymb,epsfig,graphicx,mathtools}
\usepackage{hyperref}
\usepackage{dsfont}
\usepackage{color}
\usepackage{xspace}
\usepackage{url}
\usepackage[T1]{fontenc}
\usepackage[utf8]{inputenc}
\usepackage[ulem=normalem,draft]{changes}

\numberwithin{equation}{section}
\newtheorem{theorem}{Theorem}[section]
\newtheorem{lemma}[theorem]{Lemma}
\newtheorem{proposition}[theorem]{Proposition}

\theoremstyle{definition}
\newtheorem{example}[theorem]{Example} %[section]

\theoremstyle{definition}
\newtheorem{definition}[theorem]{Definition} %[section]

\theoremstyle{remark}
\newtheorem{remark}[theorem]{Remark}

\newcommand{\eqdef}{\ensuremath{\stackrel{\mbox{\upshape\tiny def.}}{=}}}
\newcommand{\norm}[1]{\left\lVert#1\right\rVert}
\newcommand{\inner}[1]{\left\langle#1\right\rangle}
\def\LM#1{\hbox{\vrule width.2pt \vbox to#1pt{\vfill \hrule width#1pt height.2pt
}}}
\def\LL{{\mathchoice {\>\LM7\>}{\>\LM7\>}{\,\LM5\,}{\,\LM{3.35}\,}}}
\newcommand{\mres}{\LL}
\def\A{\mathcal{A}}
\def\E{\mathcal{E}}
\def\G{\mathcal{G}}
\def\H{\mathcal{H}}
\def\L{\mathcal{L}}
\def\P{\mathcal{P}}

\DeclareMathOperator*{\argmin}{arg\,min}

\DeclareMathOperator{\diam}{diam}
\DeclareMathOperator{\dist}{dist}
\DeclareMathOperator{\supp}{supp}
\newcommand{\NN}{\mathbb{N}}
\newcommand{\RR}{\mathbb{R}}
\def\disp{\displaystyle}

\newcommand{\nuexc}{\nu_{\text{exc}}}
\newcommand{\nuH}{\nu_{\H^1}}

\newcommand{\rhoexc}{\rho_{\text{exc}}}
\newcommand{\rhoH}{\rho_{\H^1}}

\newcommand{\gammaexc}{\gamma_{\text{exc}}}
\newcommand{\gammaH}{\gamma_{\H^1}}

\newcommand{\oB}{\overline{B}}
\DeclareMathOperator{\len}{len}
\DeclareMathOperator{\card}{card}
\DeclareMathOperator{\inte}{int}

\DeclareMathOperator{\proj}{proj}
\DeclareMathOperator{\conv}{conv}
\newcommand{\cconv}{\overline{\conv}}

\newcommand{\enscond}[2]{ \left\{\, #1 \mid #2 \,\right\} }
\newcommand{\abs}[1]{\left\lvert#1\right\rvert}
\newcommand{\oi}[2]{ \left] #1, #2 \right[ }
\newcommand{\ci}[2]{ \left[ #1, #2 \right] }

\def\dd{\mathrm{d}}
\title[]{One-dimensional approximation of measures in Wasserstein distance} 

\author{Antonin Chambolle}
\email{chambolle@ceremade.dauphine.fr}
\author{Vincent Duval}
\email{vincent.duval@inria.fr}
\author{João Miguel Machado}
\email{joao-miguel.machado@ceremade.dauphine.fr}
\address{CEREMADE, CNRS, Universit\'e Paris-Dauphine, Universit\'e PSL, 75775 Paris, France}
\address{MOKAPLAN, Inria Paris, 46 rue Barrault, 75013 Paris, France}
\thanks{The authors are members of the Inria MOKAPLAN team. V.D. gratefully acknowledges support from the Agence Nationale de la Recherche (CIPRESSI, ANR-19-CE48-0017-01). }

\date{\today}

\begin{document}
\maketitle

\begin{abstract}
	 We propose a variational approach to approximate measures with measures uniformly distributed over a 1-dimensional set. The problem consists in minimizing a Wasserstein distance as a data term with a regularization given by the length of the support. As it is challenging to prove existence of solutions to this problem, we propose a relaxed formulation, which always admits a solution. In the sequel we show that, under some assumption on the original measure, a solution to the relaxed problem is solution to the original one. Finally we prove that, whenever the original measure has a density in $L^{\frac{d}{d-1}}(\RR^d)$, any optimal solution is supported by an Ahlfors regular set.
\end{abstract}
%\tableofcontents

\section{Introduction}\label{section.introduction}
In this paper we study the following 1-dimensional (1D) shape optimization
 problem: given a reference probability measure 
 $\rho_0 \in \mathcal{P}_p(\RR^d)$ (the set of
 Borel probability measures $\rho$ with $\int_{\RR^d}|x|^p \dd \rho<+\infty$, $p\ge 1$), we seek to approximate $\rho_0$ with measures uniformly supported on a one-dimensional connected subset of $\RR^d$. 
 This approximation is done by means of the following variational problem
\begin{equation}
	\label{problem.shape_optimization}
	\tag{$P_\Lambda$}
	\inf_{\Sigma \in \mathcal{A}} 
	W_p^p(\rho_0, \nu_{\Sigma}) + \Lambda \H^1(\Sigma),
\end{equation}
where the measure $\nu_{\Sigma}$ is defined as 
\begin{equation}\label{admissible_set}
	\nu_{\Sigma} \eqdef \frac{1}{\H^1(\Sigma)} \H^1\mres \Sigma, \text{ for } \Sigma \in 
    \mathcal{A}\eqdef 
	\left\{
		\Sigma \subset \RR^d :
		\begin{array}{l}
			0 < \mathcal{H}^1(\Sigma) < +\infty \\
			\text{compact, connected.}
		\end{array}
	\right\},
\end{equation}
and $\mathcal{H}^1$ denotes the 1-dimensional Hausdorff measure in $\RR^d$. The term $W_p$ denotes the usual Wasserstein distance on the space of probability measures (see~\cite{santambrogio2015optimal,villani2009optimal}
and Section~\ref{sec:wasserstein}).

One can trace the idea of approximating a probability measure by a 1D set back to the concept of \textit{principal curves} from the seminal paper \cite{hastie1989principal}, which extends linear regression to regression using general curves, and introduces a variational problem to define such curves.
In this variational sense, a principal curve minimizes the expectation of the distance to the curve, w.r.t.~a probability measure describing a data set (with some regularization to ensure existence). %This problem was introduced with the goal of performing regression with a general curve instead of a line or splines, and hence giving a more intrinsic description of the data. 
As proposed in \cite{kegl2000learning}, a length constraint is a simple and intrinsic way to ensure existence. The properties of such minimizers have been studied in detail in e.g.~\cite{lu2016average,delattre2020principal}. 

A further generalization consists in replacing the curve with a more general one-dimensional compact and connected set, yielding the \textit{average distance minimizer problem} introduced in~\cite{buttazzo2003optimal}, and its dual counterpart \textit{maximum distance minimizer problem}~\cite{paolini2004qualitative,lemenant2010presentation}. Such problems were conceived for applications in urban planning, where one seeks to minimize the average distance to a transportation network, giving rise to the need for a larger class of 1D sets allowing for bifurcations. 

While the above-mentioned problems only focus on some geometric approximation of the support of the measure, approximating a measure in the sense of weak convergence is sometimes more desirable. %Recently new variational models have been proposed to take into account the topology of the space of probability measures, where the approximation actually occurs. 
In \cite{lebrat2019optimal,chauffert2017projection}, the authors have proposed optimal transport based methods for the projection of probability measures onto classes of measures supported on simple curves, using the Wasserstein distance as a data term. Potential applications range from  3D printing to image compression and reconstruction. %In \cite{lebrat2019optimal} the reader will find a comprehensive catalog of applications of these models in many different domains ranging from 3D printing to image compression and reconstruction. 
In \cite{ehler2021curve}, the data fidelity term is chosen to be a discrepancy, see also \cite{neumayer2021optimal}. The advantage of using discrepancies is that approximation rates can be given independently from the dimension, being therefore a good alternative to overcome the curse of dimensionality. The problem we study is an attempt to generalize this class of problems 
 to the approximation with one-dimensional connected sets. 
%It is also important to point out that problem \eqref{problem.shape_optimization} falls in the category of 1D shape optimization problems.
%One could also relate \eqref{problem.shape_optimization} to  1D shape optimization problems such as the classical Steiner problem \cite{brazil2014history,paolini2013existence} or the optimal compliance problem studied in~\cite{bulanyi2022partial,buttazzo2007asymptotical,chambolle2017regularity}. For such problems the task of proving existence can be established with the direct method of the calculus of variations, \cite{dacorogna2008directmethod}, and well known topological tools (see \cite{sverak1993optimal}). In our case, the class of measures $\nu_\Sigma$ is not closed in the usual weak topologies considered for the space of probability measures. 
 
 One difficulty when studying \eqref{problem.shape_optimization} is that the class of measures $\nu_\Sigma$ is not closed in the usual weak topologies considered for the space of probability measures.
 While a sequence of sets $(\Sigma_{n})_{n\in \NN}$ in $\mathcal{A}$ with uniformly bounded length will have subsequences converging (in the Hausdorff sense) either to a point or a set in $\mathcal{A}$, the corresponding measures $\nu_{\Sigma_n}$ might converge to a measure which is not necessarily uniform on that set: longer parts of $\Sigma_n$ might concentrate in the limit on shorter parts of $\Sigma$, see~Figure~\ref{figure.concentration_sets_measure}.
 \begin{figure}[t]
    \centering
    %\tikzset{every picture/.style={width=30cm}} %set default line width to 0.75pt        

        \resizebox{0.9\linewidth}{!}{%
        \begin{tikzpicture}[x=0.75pt,y=0.75pt,yscale=-.6,xscale=0.6]
        %uncomment if require: \path (0,144); %set diagram left start at 0, and has height of 144
        
        %Straight Lines [id:da8802602754194353] 
        \draw [color={rgb, 255:red, 74; green, 144; blue, 226 }  ,draw opacity=1 ]   (60.41,81.47) -- (249.96,81.47) ;
        %Shape: Right Angle [id:dp14425218284451224] 
        \draw  [color={rgb, 255:red, 74; green, 144; blue, 226 }  ,draw opacity=1 ] (60.41,46.5) -- (249.96,46.5) -- (249.96,81.47) ;
        %Shape: Spiral [id:dp09119086341633076] 
        \draw  [color={rgb, 255:red, 208; green, 2; blue, 27 }  ,draw opacity=1 ] (60.41,46.5) .. controls (60.51,46.5) and (60.59,46.55) .. (60.64,46.65) .. controls (60.7,46.75) and (60.72,46.87) .. (60.68,47.01) .. controls (60.64,47.16) and (60.55,47.3) .. (60.41,47.39) .. controls (60.26,47.5) and (60.08,47.54) .. (59.88,47.53) .. controls (59.66,47.51) and (59.46,47.41) .. (59.27,47.24) .. controls (59.06,47.05) and (58.91,46.8) .. (58.83,46.5) .. controls (58.73,46.16) and (58.73,45.82) .. (58.81,45.46) .. controls (58.9,45.06) and (59.08,44.73) .. (59.36,44.44) .. controls (59.65,44.13) and (60,43.92) .. (60.41,43.83) .. controls (60.86,43.72) and (61.3,43.75) .. (61.74,43.93) .. controls (62.21,44.11) and (62.61,44.43) .. (62.94,44.87) .. controls (63.28,45.33) and (63.5,45.87) .. (63.59,46.5) .. controls (63.68,47.15) and (63.62,47.79) .. (63.39,48.43) .. controls (63.16,49.09) and (62.78,49.65) .. (62.27,50.1) .. controls (61.73,50.57) and (61.11,50.85) .. (60.41,50.95) .. controls (59.69,51.05) and (58.98,50.94) .. (58.3,50.61) .. controls (57.58,50.27) and (56.99,49.74) .. (56.52,49.02) .. controls (56.03,48.27) and (55.74,47.43) .. (55.65,46.5) .. controls (55.56,45.52) and (55.69,44.58) .. (56.06,43.68) .. controls (56.44,42.74) and (57.01,41.96) .. (57.77,41.36) .. controls (58.55,40.73) and (59.44,40.37) .. (60.41,40.26) .. controls (61.42,40.16) and (62.39,40.36) .. (63.33,40.84) .. controls (64.29,41.34) and (65.07,42.09) .. (65.69,43.08) .. controls (66.32,44.1) and (66.68,45.24) .. (66.77,46.5) .. controls (66.86,47.79) and (66.65,49.02) .. (66.15,50.21) .. controls (65.63,51.42) and (64.86,52.41) .. (63.86,53.18) .. controls (62.82,53.97) and (61.68,54.42) .. (60.41,54.51) .. controls (59.12,54.62) and (57.89,54.34) .. (56.71,53.7) .. controls (55.5,53.04) and (54.52,52.07) .. (53.77,50.8) .. controls (52.99,49.51) and (52.56,48.07) .. (52.47,46.5) .. controls (52.38,44.89) and (52.66,43.36) .. (53.31,41.9) .. controls (53.97,40.41) and (54.93,39.2) .. (56.18,38.27) .. controls (57.46,37.33) and (58.87,36.8) .. (60.41,36.7) .. controls (61.99,36.6) and (63.49,36.95) .. (64.91,37.76) .. controls (66.37,38.58) and (67.54,39.76) .. (68.44,41.3) .. controls (69.35,42.88) and (69.86,44.61) .. (69.94,46.5) .. controls (70.04,48.43) and (69.69,50.26) .. (68.9,51.99) .. controls (68.09,53.76) and (66.94,55.18) .. (65.44,56.27) .. controls (63.92,57.38) and (62.24,57.97) .. (60.41,58.08) .. controls (58.55,58.17) and (56.79,57.75) .. (55.12,56.78) .. controls (53.42,55.8) and (52.05,54.41) .. (51.01,52.58) .. controls (49.96,50.74) and (49.39,48.71) .. (49.3,46.5) .. controls (49.2,44.26) and (49.62,42.13) .. (50.56,40.12) .. controls (51.5,38.08) and (52.85,36.43) .. (54.59,35.19) .. controls (56.36,33.92) and (58.3,33.23) .. (60.41,33.14) .. controls (62.56,33.03) and (64.59,33.55) .. (66.5,34.67) .. controls (68.45,35.81) and (70.01,37.43) .. (71.19,39.52) .. controls (72.39,41.64) and (73.03,43.97) .. (73.12,46.5) .. controls (73.21,49.06) and (72.72,51.49) .. (71.65,53.77) .. controls (70.56,56.09) and (69.02,57.95) .. (67.03,59.35) .. controls (65.02,60.77) and (62.81,61.54) .. (60.41,61.64) .. controls (57.99,61.74) and (55.69,61.14) .. (53.53,59.87) .. controls (51.34,58.57) and (49.59,56.74) .. (48.26,54.36) .. controls (46.92,51.96) and (46.21,49.35) .. (46.12,46.5) .. controls (46.03,43.62) and (46.59,40.89) .. (47.8,38.33) .. controls (49.03,35.75) and (50.77,33.66) .. (53,32.1) .. controls (55.26,30.52) and (57.73,29.68) .. (60.41,29.58) .. controls (63.13,29.48) and (65.68,30.14) .. (68.09,31.59) .. controls (70.53,33.04) and (72.48,35.1) .. (73.94,37.74) .. controls (75.42,40.41) and (76.21,43.34) .. (76.3,46.5) .. controls (76.39,49.7) and (75.76,52.71) .. (74.4,55.55) .. controls (73.03,58.42) and (71.1,60.72) .. (68.62,62.44) .. controls (66.11,64.17) and (63.38,65.1) .. (60.41,65.2) .. controls (57.42,65.31) and (54.6,64.55) .. (51.94,62.95) .. controls (49.26,61.34) and (47.12,59.07) .. (45.51,56.15) .. controls (43.89,53.19) and (43.03,49.98) .. (42.94,46.5) .. controls (42.85,42.98) and (43.56,39.66) .. (45.05,36.55) .. controls (46.57,33.41) and (48.69,30.9) .. (51.41,29.02) .. controls (54.17,27.11) and (57.17,26.11) .. (60.41,26.02) .. controls (63.69,25.91) and (66.78,26.74) .. (69.68,28.5) .. controls (72.61,30.27) and (74.94,32.76) .. (76.69,35.96) .. controls (78.46,39.19) and (79.38,42.7) .. (79.47,46.5) .. controls (79.57,50.34) and (78.79,53.95) .. (77.15,57.33) .. controls (75.49,60.75) and (73.18,63.48) .. (70.21,65.52) .. controls (67.21,67.58) and (63.95,68.66) .. (60.41,68.76) .. controls (56.85,68.86) and (53.5,67.95) .. (50.35,66.04) .. controls (47.18,64.11) and (44.65,61.39) .. (42.76,57.93) .. controls (40.85,54.43) and (39.86,50.62) .. (39.77,46.5) .. controls (39.67,42.34) and (40.52,38.44) .. (42.3,34.77) .. controls (44.1,31.08) and (46.61,28.13) .. (49.82,25.93) .. controls (53.07,23.71) and (56.6,22.55) .. (60.41,22.45) .. controls (64.26,22.35) and (67.88,23.34) .. (71.27,25.42) .. controls (74.69,27.52) and (77.41,30.43) .. (79.44,34.18) .. controls (81.49,37.96) and (82.56,42.06) .. (82.65,46.5) .. controls (82.74,50.97) and (81.83,55.18) .. (79.9,59.11) .. controls (77.96,63.08) and (75.26,66.25) .. (71.8,68.61) .. controls (68.31,70.98) and (64.51,72.22) .. (60.41,72.32) .. controls (56.28,72.43) and (52.4,71.36) .. (48.77,69.12) .. controls (45.1,66.86) and (42.18,63.72) .. (40.01,59.71) .. controls (37.82,55.66) and (36.68,51.25) .. (36.59,46.5) .. controls (36.5,41.71) and (37.49,37.2) .. (39.55,32.99) .. controls (41.63,28.75) and (44.53,25.36) .. (48.24,22.85) .. controls (51.97,20.31) and (56.03,18.99) .. (60.41,18.89) .. controls (64.83,18.79) and (68.98,19.94) .. (72.86,22.33) .. controls (76.77,24.75) and (79.88,28.1) .. (82.19,32.4) .. controls (84.53,36.72) and (85.74,41.42) .. (85.83,46.5) ;
        %Straight Lines [id:da47741772247243164] 
        \draw    (155.19,50.38) -- (155.19,77.59) ;
        \draw [shift={(155.19,77.59)}, rotate = 270] [color={rgb, 255:red, 0; green, 0; blue, 0 }  ][line width=0.75]    (0,5.59) -- (0,-5.59)   ;
        \draw [shift={(155.19,50.38)}, rotate = 270] [color={rgb, 255:red, 0; green, 0; blue, 0 }  ][line width=0.75]    (0,5.59) -- (0,-5.59)   ;
        %Straight Lines [id:da26525690517665956] 
        \draw [color={rgb, 255:red, 74; green, 144; blue, 226 }  ,draw opacity=1 ][line width=1.5]    (456.45,63.34) -- (646,63.34) ;
        %Shape: Ellipse [id:dp028181625047899672] 
        \draw  [fill={rgb, 255:red, 208; green, 2; blue, 27 }  ,fill opacity=1 ] (451.16,63.34) .. controls (451.16,59.76) and (453.53,56.86) .. (456.45,56.86) .. controls (459.38,56.86) and (461.75,59.76) .. (461.75,63.34) .. controls (461.75,66.91) and (459.38,69.81) .. (456.45,69.81) .. controls (453.53,69.81) and (451.16,66.91) .. (451.16,63.34) -- cycle ;
        
        % Text Node
        \draw (307.25,28.47) node [anchor=north west][inner sep=0.75pt]  [font=\footnotesize]  {$\Sigma _{n}\xrightarrow[n\rightarrow\infty ]{d_{H}} \Sigma $};
        % Text Node
        \draw (247.12,94.89) node [anchor=north west][inner sep=0.75pt]  [font=\footnotesize]  {$\H^{1} \mres \Sigma_{n}\xrightharpoonup[n \rightarrow \infty ]{\star } 2\H^{1} \mres \Sigma \ +\delta _{x_{0}}$};
        % Text Node
        \draw (158.07,56.44) node [anchor=north west][inner sep=0.75pt]  [font=\footnotesize]  {$1/n$};

        \end{tikzpicture}
    }
    \caption{Concentration effects for the weak convergence of measures. Here $\Sigma_n$ is made of two lines getting closer and a spiral converging to a  point. In the Hausdorff limit we obtain a segment with shorter length, and a singular limiting measure.}\label{figure.concentration_sets_measure}
\end{figure}

Hence, minimizing sequences converge in general to a measure which is
not of the form $\nu_\Sigma$, and we need to determine a relaxation of our energy in a topology for which the Wasserstein distance is lower semi-continuous, such as the narrow convergence. The relaxed energy takes the form
\begin{equation}\label{problem.shape_optimization_relaxed}
	\tag{$\overline{P}_\Lambda$}
	\inf_{
		\substack{
			\nu \in \mathcal{P}_p(\RR^d)
		}
	  }
	W_p^p(\rho_0, \nu) + \Lambda \mathcal{L}(\nu),
\end{equation}
where  the length functional $\mathcal{L}$, defined in Section~\ref{subsec.deflength}, generalizes the notion of length of the support of a measure, see for instance Example~\ref{example.densities_rectifiable_set}.  We will show later on, in Proposition~\ref{proposition.length_functional_is_lsc},
that $\mathcal{L}$ is the lower semi-continuous relaxation, for the narrow topology, of the functional $\ell$ given by $\H^1(\Sigma)$ for measures
of the form $\nu_\Sigma$, and $+\infty$ else, see~\eqref{eq:defl}.
We also find that $\mathcal{L}(\nu) < \infty$ if and only if $\supp \nu \in \mathcal{A}$ or $\nu$ is a Dirac mass. The following Theorem gathers the various results proved throughout this paper.
\begin{theorem}\label{theorem.the_big_one}
  Let $\rho_0 \in \mathcal{P}_p(\RR^d)$, $\Lambda>0$. Then~\eqref{problem.shape_optimization_relaxed}  admits a solution $\nu$, and there exists $\Lambda_\star \ge 0$ such that if $\Lambda>\Lambda_\star$, $\nu$ is a Dirac mass. For $\Lambda<\Lambda_\star$, $\nu$ is supported by a set $\Sigma \in \mathcal{A}$ and the following properties hold.
  \begin{enumerate}
  %\item If $\rho_0$ is absolutely continuous w.r.t. $\H^1$ (resp. has a $L^\infty$ density w.r.t. $\H^1$), then  $\nu\ll \H^1\mres \Sigma$.
  \item If $\rho_0$ is absolutely continuous w.r.t.~$\H^1$, or has a $L^\infty$ density w.r.t.~$\H^1$, then so does $\nu$.
  \item  If $\rho_0$ does not give mass to $1D$-rectifiable sets, then $\nu=\nu_\Sigma$ and therefore is a solution to the original problem \eqref{problem.shape_optimization}.
  \item If $\rho_0 \in L^{\frac{d}{d-1}}(\RR^d)$, then $\Sigma$ is Ahlfors regular, \textit{i.e.} there is $r_0$ depending on $d,p,\rho_0$ and $\mathcal{L}(\nu)$ and $C$ depending only on $d,p$ such that for any $x \in \Sigma$ and $r \le r_0$ it holds that
    \[
      r \le \H^1(\Sigma \cap B_r(x))\le Cr. 			
    \]
  \end{enumerate}
\end{theorem}
%We believe that $\Lambda_*<+\infty$, although we are able to prove it
%only when $p=2$ or $\rho_0$ has compact support.

The paper is organized as follows: in Section~\ref{section.preliminary_properties} we recall a few tools from optimal transport and geometric measure theory. Next, in Section \ref{section.the_problem} we go through the definition of the length functional and its properties as well as the relaxed problem and the existence of a solution. In Section \ref{section.geometric_properties_transport_map} we discuss
the existence of $\Lambda_\star$.
In Section~\ref{section.solutions_are_absolutely_continuous}
(Theorem~\ref{theorem.solutions_are_absolutely_continuous}) we prove point~(1) from Theorem \ref{theorem.the_big_one}, while the existence is proved in Section~\ref{section.blowup} (Theorem~\ref{theorem.inR2_density_is_1}),
and the Ahlfors regularity is studied in Section~\ref{section.ahlfors_regularity}.

\section{Preliminaries}\label{section.preliminary_properties}
We start by introducing notions of convergence for sets and measures which will be useful to study problem \eqref{problem.shape_optimization} as well as the relaxed one \eqref{problem.shape_optimization_relaxed}. Next we describe some instrumental properties of the objects we shall use throughout the paper, namely the rectifiable sets and measures.

\subsection{Convergence of sets and measures}
% To study problem \eqref{problem.shape_optimization}, we need notions of convergence for the sets belonging to the admissible class $\mathcal{A}$ as well as to probability measures.

\subsubsection{Hausdorff and Kuratowski convergence}\label{subsection.Hausdorff_Kuratoski}
% We start by defining the notions of {\em Hausdorff} and {\em Kuratowski} convergence of sets, see for instance \cite{rockafellar2009variational}.
We recall some useful definitions of convergence for sets, see for instance~\cite[Chap.~4]{rockafellar2009variational}, \cite[Chap.~6]{ambrosio2000functions}.

The Hausdorff distance between two sets $A$, $B$ is defined
as:
\begin{equation}
	\label{hausdorff.distance}
	d_H(A,B) \eqdef \max\left\{
		\sup_{a \in A} \dist(a, B), \sup_{b \in B} \dist(b, A)
	\right\}, 
%	\text{ we write $A_n \xrightarrow[n \to \infty]{d_H} A$,}
\end{equation}
where $\dist(\cdot, A)$ denotes the distance function to the set $A$.
A sequence $\left(A_n\right)_{n \in \mathbb{N}}$ of closed subsets of $\RR^d$ converges in the Hausdorff sense to $A$ if $\lim_{n\to\infty} d_H(A_n, A)=0$, and we write $A_n \xrightarrow[n \to \infty]{d_H} A$.
One can prove that this notion of convergence is equivalent to uniform convergence of the distance functions. Since the latter are all 1-Lipschitz, as a consequence of Arzela-Ascoli's Theorem it follows that if the sequence is contained in a compact set, one can always extract a convergent subsequence. This compactness result is known as {\em Blaschke's Theorem}, see \cite[Theorem 6.1]{ambrosio2000functions}.
% for a simple proof.

A sequence of closed sets $C_n$ converges in the sense of Kuratowski to $C$, and we write $C_n \xrightarrow[n \to \infty]{K} C$, whenever the two properties hold:
\begin{enumerate}
\item Given a sequence $x_n \in C_n$, all its cluster points are contained in $C$.
\item For all points $x\in C$ there exists a sequence $x_n \in C_n$, converging to $x$.
\end{enumerate}
Again, one can show that $C_n\to C$ in the sense of Kuratowski
if and only if $\dist(x,C_n)\to\dist(x,C)$ (possibly infinite if $C=\emptyset$) locally uniformly (see~\cite[Cor.~4.7]{rockafellar2009variational}). In addition, the Bolzano-Weierstrass property holds for the Kuratowski convergence as well, \textit{i.e.} any sequence of closed sets has a subsequence which converges, possibly to the empty set.
% Kuratowski convergence can also be characterized in terms of convergence of the distance functions. Indeed using the notions of Kuratowski inner and outer limits, see \cite[Definition 4.1]{rockafellar2009variational}, one checks that $x \in C$ if and only if $\dist(x, C_n) \to 0$. Once again by Ascoli-Arzela's Theorem, we conclude that Kuratowski convergence is equivalent to locally uniform convergence of the distance functions.
% The following Lemma concerning the relation of Hausdorff and Kuratowski convergences, proven in Appendix \ref{appendix.kuratowski_convergence} will be useful. 

It is classical that Hausdorff and Kuratowski convergences coincide on sequences on uniformly bounded compact sets, the next lemma describes a more subtle relationship.
%The following Lemma describes a relation between Hausdorff and Kuratowski convergences.
We prove it in Appendix~\ref{appendix.kuratowski_convergence}.
\begin{lemma}\label{lemma.localized_Kuratowski}
	Let $(C_n)_{n \in \mathbb{N}}$ be a sequence of closed sets in $\RR^d$, converging to $C$ in the sense of Kuratowski.
    Then, for any $x \in \RR^d$,% and any radius $R> 0$ such that there exists a countable set $I\subset [0,+\infty)$ such
   % that 
    \[
      C_n\cap\overline{B_R(x)} 
      \xrightarrow[n \to \infty]{d_H} 
      C\cap\overline{B_R(x)}, 
     % \text{ for all $R\in [0,+\infty)\setminus I$.}
    \]
    for every radius $R>0$ such that $\overline{C\cap B_R(x)} = C\cap \overline{B}_R(x)$. Moreover, that condition holds for all $R>0$ except in a countable set.% of values of $R$.
\end{lemma}

\subsubsection{Convergence of measures} 
Given a Borel set $X\subset \RR^d$, we denote by $\mathcal{M}(X)$ (resp. $\mathcal{M}_+(X)$) the collection of the finite  (resp. finite positive) Radon measures on $X$.
The space of Borel probability measures on $X$ is denoted by $\mathcal{P}(X)$, and $\mathcal{P}_p(X)$ refers to its subset of probability measures with finite $p$-moment ($p\ge 1$).

Following~\cite{ambrosio2000functions}, we say that a sequence $\left(\mu_n\right)_{n \in \mathbb{N}}$ of Radon measures on $\RR^d$ locally weakly-$\star$ converges to some Radon measure $\mu$, if, for every continuous function with compact support $\phi \in C_c(\RR^d)$, 
	\begin{equation}\label{eqRadonweakstar}
	    \int_{\RR^d} \phi \dd \mu_n 
		\xrightarrow[n \to \infty]{}
		\int_{\RR^d} \phi \dd \mu.
	\end{equation}
Any sequence $(\mu_n)_{n\in\mathbb{N}}$ of Radon measures such that $\sup_{n\in\mathbb{N}} \abs{\mu_n}(K)<+\infty$ for every compact set $K\subset \RR^d$ has a locally weakly-$\star$ convergent subsequence.	
	
If $\left\{\mu_n\right\}_{n \in \mathbb{N}}\subset \mathcal{M}(\RR^d)$ is  a sequence of \emph{finite} Radon measures and~\eqref{eqRadonweakstar} holds for every bounded continuous function $\phi\in C_b(\RR^d)$, we say that $\mu_n$ narrowly converges  to $\mu$, and we write $\mu_n \xrightharpoonup[n \to \infty]{} \mu$. When $(\mu_n)_{n\in\mathbb{N}}$ is a sequence of probability measures, that convergence is often referred to as the \emph{weak convergence} of probability measures.

If $X$ is compact, any sequence of probability measures $\{\mu_n\}_{n\in\mathbb{N}}\subset \mathcal{P}(X)$ has a weakly convergent subsequence. 
More generally, if $X$ is not compact, compactness for the narrow convergence requires the assumptions of Prokhorov's Theorem,
see~\cite[Theorem 2.8]{ambrosio2021lectures}.

\subsubsection{Optimal transport and the Wasserstein distance} 
\label{sec:wasserstein}

The Wasserstein distances $W_p$ are defined through the value function of an optimal transport problem, see \cite{ambrosio2008gradient,santambrogio2015optimal,villani2009optimal} for details. Given two probability measures $\mu, \nu \in \mathcal{P}_p(\RR^d)$, we set 
\begin{equation}\label{wasserstein_distance}
	W_p^p(\mu, \nu) 
	\eqdef
	\min_{\gamma \in \Pi(\mu, \nu)}
	\int_{\RR^d\times \RR^d} |x - y|^p\dd\gamma,
\end{equation}
where 
$\disp 
    \Pi(\mu, \nu) \eqdef  
    \left\{ 
	    \gamma \in \mathcal{P}\left(\RR^d\times \RR^d\right) : 
	    \quad
			{\pi_{0}}_\sharp\gamma = \mu, \
	        {\pi_{1}}_\sharp\gamma = \nu
    \right\}
$
is the space of transport couplings, and $\pi_i$ denote the projections, {\em i.e.} $\pi_0(x,y) = x$ and $\pi_1(x,y) = y$. Whenever $\mu$ does not have atoms, the value of \eqref{wasserstein_distance} coincides with
\begin{equation}\label{wasserstein_distance.monge}
    \inf_{
		\substack{
			T_{\sharp}\mu = \nu
		}
	}
	\int_{\RR^d} |x - T(x)|^p\dd \mu,
\end{equation}
where the inf is taken over all measurable maps $T$ such that $T_\sharp\mu(A) = \nu(A) = \mu(T^{-1}(A))$ for any Borel set $A$.

The optimal transport problem can be analogously defined for any pair of positive $\mu, \nu$ in $\mathcal{M}_+(\RR^d)$. In this case, the Wasserstein distance becomes a 1-homogeneous functional and is finite if and only if the measures have finite $p$-moments and the same total mass $\mu(\RR^d) = \nu(\RR^d)$.

%One of the advantages of the Wasserstein distances is their good properties w.r.t. convergence of probability measures in weak topologies, {\em i.e.} in duality with certain functional spaces. 

% Both notions of convergence have their advantages.
% While
The Wasserstein distance is l.s.c.~with respect to the narrow convergence, and continuous in a compact domain, \cite[Lemma 4.3]{villani2009optimal}.

%%% proper Polish writing (but should we use it?)
\newcommand{\Golab}{Go\l\k{a}b}
%\renewcommand{\Golab}{Golab}
%%%
\subsection{\Golab's Theorem}
% Now that we have discussed about the topologies we will be working with, we need to adress the continuity properties of the functional from \eqref{problem.shape_optimization}.
We now study the lower semicontinuity of~\eqref{problem.shape_optimization}.
``\Golab's Theorem''~\cite{Golab} shows that along sequences of connected sets, the length is lower semi-continuous with respect to the Hausdorff convergence~\cite[Chapter 10]{morel2012variational}.
%The same conclusion is also true if $\Sigma_n$ satisfies a
%{\em uniform concentration property}~\cite{morel2012variational}.
It is of course also true if the sequence has a uniformly bounded number of connected components.

The issue is that the compactness of Hausdorff convergence is not transferred to the weak convergence of measures of the form $\H^1\mres \Sigma$ which may concentrate in the limit.
%In general, \Golab's Theorem implies that if $\left(\Sigma_n\right)_{n \in \mathbb{N}}$ is a sequence of connected sets converging to $\Sigma$ in the sense of Kuratowski, such that $\H^1\mres \Sigma_n \xrightharpoonup[n \to \infty]{\star} \mu$ then $\mu \ge \H^1\mres \Sigma$.
% This is proven in the sequel and
%We  call this result the {\em density version of \Golab's Theorem}. This result is hidden in the proof found in~\cite{ambrosio2004topics} of the usual thesis of \Golab's Theorem,
%see also~\cite{paolini2013existence}.
In general, one can prove the following: %  {\em density version of \Golab's Theorem}:
\begin{theorem}[Density version of \Golab's Theorem]\label{theorem.Golab_localversion}
	Let ${(\Sigma_n)}_{n \in \mathbb{N}}$ be a sequence of closed and connected subsets of $\RR^d$ converging in the sense of Kuratowski to some closed set $\Sigma$ and having locally uniformly finite length, {\em i.e.} for all $R > 0$
	\[
		\sup_{n \in \mathbb{N}} \H^1(\Sigma_n \cap B_R(x_0)) < +\infty.
	\]
	Define the measures $\mu_n \eqdef \H^1\mres \Sigma_n$, and let $\mu$ be a local weak-$\star$ cluster point of this sequence. Then $\supp\mu \subset \Sigma$ and it holds that 
	\[
		\mu \ge \H^1\mres \Sigma,
	\]
	in the sense of measures. 
\end{theorem}
      Such a  result is hidden in the proof in~\cite{ambrosio2004topics} of the usual thesis of \Golab's Theorem, see also~\cite{paolini2013existence}.
      Yet in this variant, we consider a Kuratowski convergence and do not restrict the sets to be uniformly bounded, or have bounded lengths.
      This is useful for the proof of Theorem~\ref{theorem.Gamma_convergence} where we consider sequences of blow-ups of sets. The proof of Theorem~\ref{theorem.Golab_localversion} is given in Appendix~\ref{appendix.kuratowski_convergence}.
      %For the reader's convenience we give a simple proof in Appendix~\ref{appendix.kuratowski_convergence}.
    %      \joao{Je trouve important d'insister que nous avons un peu plus de généralité, et que celle-ci est nécéssaire pour ce papier.}
%\begin{remark}
%	As we do not use any properties from the vector space structure of $\RR^d$, this \revision{result also holds} in the case of a locally compact metric space, as in \cite{ambrosio2004topics}. \vincent{en fait Ambrosio Tilli  se place dans un espace métrique complet. Est-ce que nous aurions besoin spécifiquement besoin de la compacité locale?}
%\end{remark}

\subsection{Rectifiable sets and measures}
We now introduce
the notions of {\em rectifiable sets} and {\em rectifiable measure}, 
%see~\cite[Definition 2.57]{ambrosio2000functions} or~\cite[Chapter 10]{maggi2012sets} 
which will be crucial for understanding the fine
properties of the elements of $\mathcal{A}$.
\begin{definition}
	Let $M \subset \RR^d$ be a Borel set and $k \in \mathbb{N}$, we say that $M$ is countably $\H^k$-rectifiable, or shortly $k$ rectifiable, if there are countably many Lipschitz functions $f_i: \RR^k \to \RR^d$ such that
	\[
		\H^k\left(
			M \setminus \bigcup_{i \in \mathbb{N}} f_i\left(\RR^k\right)
		\right)
		= 0. 	
	\]
	A Radon measure $\mu$ is said to be $k$-rectifiable if it is supported over a $k$-rectifiable set and $\mu \ll \H^k$.
\end{definition}

In the simple case $M = f(E)$, for $E \subset \RR^k$, one can define the tangent space at a point of differentiability of $f$ as 
\[
	\nabla f(z)\left(\RR^k\right), \text{ for $x = f(z)$}.
\]
This is a parametric definition that can be extended to $k$-rectifiable sets. It turns out the parametric notion of tangentiability can be expressed in terms of measure theory. Given a Borel set $M$, we set the measure $\mu = \H^k\mres M$, and we consider the family of blow-up measures
\begin{equation}\label{blowup_measure_operator}
	\mu_r \eqdef r^{-k} \Phi^{x,r}_\sharp \mu = \H^k\mres \left(\frac{M - x}{r}\right), \text{ for } 
	\Phi^{x,r}\eqdef \frac{\text{id} - x}{r}.
\end{equation}

If $M$ is countably $\H^k$-rectifiable, and $\H^k(M\cap K)<+\infty$ for every compact set $K$, we say that $M$ is locally $\H^k$-rectifiable, and then the blow-up Theorem, see \cite[Theorem 10.2]{maggi2012sets}, states that for $\H^k$-a.e. $x \in M$ this family of measures converges in the weak-$\star$ topology to a measure of the form $\H^k \mres \pi_x$, for a unique $k$-plane $\pi_x \in \text{G}(k,d)$, the Grassmannian of $k$-planes of $\RR^d$. 

More generally define the $k$-density, whenever the limit exists, of a Radon measure $\mu$ as 
\begin{equation}\label{eq.k_density}
	\theta_k(\mu,x) \eqdef \lim_{r \to 0^+} 
	\frac{\mu(B_r(x))}{\omega_k r^k} \text{ and }
	\theta_k(M,x) \eqdef \theta_k\left(\H^k\mres M, x\right), 
\end{equation}
where $\omega_k$ is the volume of the unit $k$-dimensional ball, see~\cite{ambrosio2000functions,maggi2012sets}. A direct consequence of the blow-up Theorem is that $\H^k$-a.e.~point of a $k$-rectifiable set has $k$-density $1$. Analogously for a $k$-rectifiable measure $\mu$ it holds that $\mu = \theta_k(\mu, x) \H^k\mres M$. 

The equivalence between all notions was completed with the work of Preiss and the notion of a tangent space to a measure, see for instance the monograph~\cite{delellis2006lecture}. If a measure ({\em resp.} a set) has a finite $k$-density, {\em i.e.}~the limit in~\eqref{eq.k_density} exists and is finite $\H^k$-a.e., then this measure ({\em resp.} set) is $k$-rectifiable. The previous discussion is summarized in the following theorem.

\begin{theorem}\label{theorem.rectifiability_tangentiability}
	Let $\mu$ be a Radon measure over $\RR^d$, the following are equivalent.
	\begin{enumerate}
		\item[(i)] $\mu$ is $k$-rectifiable
		\item[(ii)] For $\H^k$-a.e. $x \in \supp \mu$, it holds that
		\[
			r^{-k}\Phi^{x,r}_\sharp \mu \xrightharpoonup[r \to 0]{\star}	
			\theta_k(\mu,x)\H^k\mres \pi_x,
		\] 
		for a unique $k$-plane $\pi_x \in \text{G}(k,d)$.
		\item[(iii)] For $\H^k$-a.e. $x \in \supp \mu$, the $k$-density of $\mu$ in \eqref{eq.k_density} exists and is finite.
	\end{enumerate}
\end{theorem}
\if{
\begin{theorem}[Rectifiability and tangentiability]
Let $M$ be a Borel subset of $\RR^d$, then the following are equivalent.
\begin{enumerate}
	\item[(i)] For $\H^k$-a.e. $x \in M$, it holds that
	\[
		\H^k\mres\left(\frac{M - x}{r}\right)
		\xrightharpoonup[r \to 0]{\star}	
		\H^k\mres T_xM,
	\] 
	for a unique $k$-plane $T_xM \in \text{G}(k,d)$, which is called the {\em approximate tangent space} of $M$ at $x$
\item[(ii)] $M$ is a $k$-rectifiable set, admitting a parametric representation of the form
	\[
		M = M_0 \cup \bigcup_{i \in \mathbb{N}} f_i(\RR^k),
	\] 
	with $\H^k(M_0) = 0$ and $f_i$ being Lipschitz functions for all $i \in \mathbb{N}$. In addition, for $\H^k$-a.e. $x \in M$, if $x = f_i(z)$, then the tangent space of $M$ at $x$ is
	\[
		T_xM = \nabla f_i(z)\left(\RR^k\right). 
	\] 
	\item[(iii)] For $\H^k$-a.e. $x \in M$ one has 
	$\Theta_k(M, x) = 1$ and there exists a unique $\pi_x \in \text{G}(k,d)$ such that for all $\alpha > 0$, one has 
	\[
		\lim_{r \to 0} r^{-k} \H^k\left(
			M\cap B_r(x)\setminus K_r(x, \pi_x, \alpha)
		\right)	
		= 0,
	\]  
	and $T_xM$ is the unique plane $\pi_x$ satisfying this property.
\end{enumerate}
\end{theorem}
\begin{proof}
	The equivalence between points (i) and (ii) is given by Theorems 10.2 and 10.8 from \cite{maggi2012sets}, or the criteria for rectifiability of measures in \cite[Theorem 2.83]{ambrosio2000functions}. To prove the parametric characterization of the  tangent space consider a $f_i$ from the parametric representation of $M$ and define a new submanifold $M_i = f_i(\RR^k)$. For $\H^k$-a.e. $x \in M_i$ we have that $T_xM_i = \nabla f_i(z)(\RR^k)$, for $x = f_i(z)$. The locality property of the tangent spaces, \cite[Proposition 10.5]{maggi2012sets}, says that $T_xM = T_xM_i = \nabla f_i(z)(\RR^k)$ for $\H^k$-a.e. $x \in M \cap M_i$. 

	The fact that (iii) implies (i) and (ii), {\em i.e.} a.e. tangentiability implies rectifiability, is stated in the literature as a sufficient condition of rectifiability, see \cite[Proposition 4.6]{delellis2006lecture}. Weaker sufficient conditions can also be found in \cite{ambrosio2000functions,maggi2012sets}. 

	To see that (i) implies (iii) is a direct consequence of local weak-$\star$ convergence. Indeed, any plan $\pi_x \in \text{G}(k,d)$ satisfying 
	\[
		\H^k\mres\left(\frac{M - x}{r}\right)
		\xrightharpoonup[r \to 0]{\star}	
		\H^k\mres \pi_x, 	
	\]
	will be tangent to $M$, in the sense of Definition \ref{definition.tangent_plan}, at $x$ since for all $\alpha > 0$, evaluating both sides on $K_r(x, \pi_x, \alpha)^c$ one has 
	\[
		r^{-k} \H^k\left(
			M\cap B_r(x)\setminus K_r(x, \pi_x, \alpha)
		\right)	
		\xrightarrow[r \to 0]{}
		\H^k
		\left(
			\pi_x\cap B_r(x) \setminus 
			K(\pi_x, \alpha)
		\right)	
		= 0. 
	\]
	Since $T_xM$ is the unique $k$-plane satisfying the blow-up property, for $\H^k$-a.e. $x\in M$, it is the unique tangent plane to $M$ at $x$. 
\end{proof}
}\fi
\if{
One direct application of this is that if $M$ is a $k$-rectifiable set, then almost all of its points have density $1$. Since for any $x \in M$ admitting a tangent space $T_xM$, one has
\[
	\lim_{r \to 0} 
	\H^k\mres
	\left(
		\frac{ M - x}{r}
	\right)\left(B_1\right)
	= \H^k\mres T_xM(B_1) = \omega_k,
\]
as  $\H^k(\partial (T_xM\cap B_1)) = 0$, and we conclude that $\Theta_k(M, x) = 1$. Therefore, we say that $x$ is a {\em rectifiability point} of $M$, whenever it admits a tangent space $T_xM$. 
}\fi 

In the previous Theorem, if we take $\mu = \H^k\mres M$ where $M$ is a countably $\H^k$-rectifiable set we define the {\em approximate tangent space} of $M$ at $x$ as $T_x M \eqdef \pi_x$, where $\pi_x$ is the unique $k$-plane from point (ii).

\begin{definition}\label{definiton.rectifiability_point}
	Let $M\subset \RR^d$ be a $k$-rectifiable set. We say that $x \in M$ is a rectifiability point when the locally weakly-$\star$ convergence of point (ii) from Theorem~\ref{theorem.rectifiability_tangentiability} holds, with $\mu = \H^k\mres M$ and $\theta_k(\mu,x)=1$. 
\end{definition}

Now we pass to our case of interest, of compact connected sets $\Sigma$ with finite length, $\mathcal{H}^1(\Sigma) < +\infty$, in view of~\eqref{admissible_set}. From~\cite[Prop.~30.1 and Cor.~30.2]{david2006singular} or~\cite[Thm.~4.4]{alberti2017structure}, any compact connected set with finite length is also path-connected, \textit{i.e.} any pair of its points can be joined by a continuous arc. 

Such sets are also known to be $1$-rectifiable, see~\cite[Thm.~4.4.8]{ambrosio2004topics}, and hence they enjoy the properties of Theorem \ref{theorem.rectifiability_tangentiability}. In the next Lemma, we show that the blow-up of some $\Sigma \in \mathcal{A}$ around a rectifiability point is precisely its approximate tangent space. 

\begin{lemma}\label{lemma.blowup_domain_measure}
  Given $\Sigma \in \mathcal{A}$, then for $\H^1$-a.e. $y \in \Sigma$, it holds that 
  \[
    \frac{\Sigma - y}{r} \xrightarrow[r \to 0^+]{K} T_y\Sigma 
	\quad \text{and} \quad
	\frac{\Sigma - y}{r}\cap \overline{B_R(0)} \xrightarrow[r \to 0^+]{d_H} T_y\Sigma\cap \overline{B_R(0)}, \text{ for all $R>0$.}
  \]
\end{lemma}
\begin{proof} 
  First we take a rectifiability point  $y \in \Sigma$ with tangent space $T_y\Sigma$,
  by Theorem~\ref{theorem.rectifiability_tangentiability} such points cover $\H^1$-a.a.~of $\Sigma$.
  In particular, point \textit{(ii)} of the Theorem shows that $\H^1\mres((\Sigma-y)/r)\xrightharpoonup[r \to 0]{\star} \H^1\mres T_y\Sigma$.
  Let $T$ be the (Kuratowski) limit of a subsequence $(\Sigma-y)/r_k$.
  Clearly, the limit measure $\H^1\mres T_y\Sigma$ is supported by $T$, hence $T_y\Sigma\subset T$.
  Thanks
  to Lemma~\ref{lemma.localized_Kuratowski} and Theorem~\ref{theorem.Golab_localversion}, for
  almost all $R>0$,
 \begin{equation}
    \H^1(T\cap B_R) \le \liminf_{k\to\infty} \H^1\left(\frac{\Sigma-y}{r_k}\cap B_R\right)
    = \H^1(T_y\Sigma\cap B_R),\label{eq.blowup_domain_measure}
 \end{equation}
which shows that up to a $\H^1$-negligible set, $T=T_y\Sigma$.

Notice that, if there is some $x \in T\setminus T_y\Sigma$, we may consider some ball $B_s(x)$ which does not intersect $T_y\Sigma$. Since $T$ is the limit of connected sets, $x$ must be path-connected in $T$ to some point in $(B_s(x))^c$, so that $\H^1(T\cap B_s(x))\ge s$. This contradicts~\eqref{eq.blowup_domain_measure}. Hence $T=T_y\Sigma$, and
  %$\Sigma$ is connected, $x$ must be connected either to .Hence the Kuratowski limit
  is independent on the subsequence, and we deduce that $(\Sigma-y)/r\stackrel{K}{\to} T_y\Sigma$. The convergence in the Hausdorff distance follows from Lemma~\ref{lemma.localized_Kuratowski}.
\end{proof}

\section{The length functional and the relaxed problem}\label{section.the_problem}
If a minimizing sequence $\Sigma_n$ converges to some set $\Sigma$, we cannot expect weak cluster points of the measures $\nu_{\Sigma_n}$ to have the form $\nu_\Sigma$, see Figure~\ref{figure.concentration_sets_measure}. Hence the objective of \eqref{problem.shape_optimization} is not lower semi-continuous for the narrow convergence, and, in this section, we introduce a relaxation for \eqref{problem.shape_optimization}. First, we define a functional which extends the length of the support  and we discuss some of its properties, then we use it to define the relaxed problem.

\subsection{Definition and elementary properties}\label{subsec.deflength}
Recalling that $\mathcal{A}$ is the collection of the compact connected sets $\Sigma \subset \RR^d$ with $0<\H^1(\Sigma)<+\infty$, we consider
\begin{equation}\label{eq:defl}
	\ell : \mathcal{P}(\RR^d) \ni \nu \mapsto 
	\left\{
		\begin{array}{cc}
			\H^1(\Sigma), & \text{ if  $\displaystyle \nu=\frac{1}{\H^1(\Sigma)}\H^1\mres \Sigma$ for some $\Sigma \in \mathcal{A}$},\\ 
			+\infty,& \text{ otherwise,}
		\end{array}
	\right.
\end{equation}
so that \eqref{problem.shape_optimization} becomes $\inf_\nu W_p^p(\rho_0, \nu) + \Lambda\ell(\nu)$. As discussed above, $\ell$ is not l.s.c., hence we introduce the following relaxation, which we call the {\em length functional}. For any $\nu \in \mathcal{P}(\RR^d)$, we define
\begin{equation}
	\mathcal{L}(\nu)
	\eqdef 
	\begin{cases}
		\inf \enscond{\alpha \ge 0}{\alpha \nu \ge \H^1\mres \supp \nu},  & \mbox{if $\supp \nu$ is connected,}\\
	%	 & \mbox{if $\supp \nu$ is connected,}\\
		+\infty, & \mbox{otherwise,}
	\end{cases}\label{lengthalpha}
\end{equation}
with the convention that $\inf \emptyset \eqdef +\infty$. Notice that, since $\nu$ is a probability measure, $\mathcal{L}(\nu)\geq \H^1(\supp \nu)$, and that $\mathcal{L}(\nu)=0$ if and only if $\nu = \delta_x$  for some $x \in \RR^d$. As a result, $0<\mathcal{L}(\nu) < \infty$ if and only if $\supp \nu \in \mathcal{A}$. Moreover, for any $\Sigma \in \mathcal{A}$ and $ \displaystyle \nu_\Sigma\eqdef\frac{1}{\H^1(\Sigma)} \H^1\mres \Sigma$, we have $\mathcal{L}(\nu_\Sigma) = \H^1(\Sigma)=\ell(\nu_\Sigma)$. 

\begin{remark}
Definition~\eqref{lengthalpha} also makes sense for any positive measure $\mu \in \mathcal{M}_+(\RR^d)$. In that case, thanks to Theorem~\ref{theorem.Golab_localversion}, it may be easily shown to be
lower semi-continuous with respect to the weak convergence,
defining $\mathcal{L}(0)=0$ (see also Section~\ref{sec.length_functional_is_lsc}). 
Yet then, of course, even for uniformly distributed measures such as $\nu = \theta\H^1\mres \Sigma$ for some $\theta> 0$, its value does not coincide with the length of the support anymore (it rather is $\H^1(\Sigma)/\nu(\RR^d)$).
\end{remark}

In Section~\ref{sec.length_functional_is_lsc} below, we prove that $\mathcal{L}$ is  the lower semi-continuous enveloppe of $\ell$ for the narrow topology of probability measures. Before that,  let us discuss some alternative formulations for $\mathcal{L}$. Following \cite[Sec. 2.4]{ambrosio2000functions}, we consider the upper derivative,
\begin{align}
\forall x \in \supp\nu,\quad 	D^+_\nu(\H^1\mres \supp\nu)(x) \eqdef \limsup_{r\to 0^+} \frac{\H^1(B_r(x)\cap \supp \nu)}{\nu(B_r(x))}.
\end{align}

\begin{proposition}[Alternative definitions of $\mathcal{L}$]
Let $\nu \in \mathcal{P}(\RR^d)$ be such that $\supp \nu$ is connected. Then

\begin{align}
	\mathcal{L}(\nu) &=\sup \enscond{\frac{\H^1(U \cap \supp \nu)}{\nu(U)}}{U \mbox{ open}, U\cap \supp \nu\neq \emptyset}\label{lengthopen}\\ 
		&=\sup \enscond{\frac{\H^1(B_r(x)\cap \supp \nu)}{\nu(B_r(x))}}{r>0, x \in \supp \nu} \label{lengthball}\\
		&=	%\disp
		\norm{D^+_\nu(\H^1\mres \supp\nu)}_{\infty},\label{lengthderiv}
\end{align}
where $\norm{\cdot}_\infty$ denotes the supremum norm over $\supp \nu$. 
\end{proposition}

\begin{proof}
It is immediate that 
	\begin{align*}
		\left(\mbox{R.H.S. of \eqref{lengthalpha}}\right)\geq \left(\mbox{R.H.S. of \eqref{lengthopen}}\right)\geq \left(\mbox{R.H.S. of \eqref{lengthball}}\right)\geq \left(\mbox{R.H.S. of \eqref{lengthderiv}}\right).
\end{align*}
Now, assume that $\norm{D^+_\nu(\H^1\mres \supp\nu)}_{\infty}<+\infty$ and let $\alpha > \norm{D^+_\nu(\H^1\mres \supp\nu)}_{\infty}$. For every compact set $K \subset \RR^d$ and every $x \in K\cap \left(\supp \nu\right)$, there is some $r(x)>0$ such that $\H^1\left(B_r(x)\cap \left(\supp \nu\right)\right)\leq \alpha \nu(B_r(x))$. We may extract from the  covering $(B_{r(x)}(x))_{x \in K\cap \left(\supp\nu\right)}$ with open sets a finite covering $(B_{r_i}(x_i))_{i=1}^{N}$ of $K\cap \left(\supp \nu\right)$. As a result
\begin{align*}
	\H^1(K\cap \left(\supp\nu\right))\leq \sum_{i=1}^{N} \alpha \nu(B_{r_i}(x_i))\leq N\alpha<+\infty, 
\end{align*}
so that $\H^1\mres\left(\supp\nu\right)$ is a Radon measure. We may thus apply~\cite[Prop. 2.21]{ambrosio2000functions} to deduce
\begin{align*}
	\left(\mbox{R.H.S. of \eqref{lengthderiv}}\right)\geq \left(\mbox{R.H.S. of \eqref{lengthalpha}}\right). 
\end{align*}
If $\norm{D^+_\nu(\H^1\mres \supp\nu)}_{\infty}=+\infty$, the inequality holds trivially, which completes the proof.
\end{proof}

The length functional inherits some of the properties of the $\H^1$ measure.  
\begin{proposition}\label{prop.length.lipschitz}
	Let $f\colon \RR^d\rightarrow \RR^d$, be a $k$-Lipschitz function, with $k>0$. Then 
	\begin{align}
\forall \nu \in \mathcal{P}(\RR^d),\quad		\L(f_\sharp\nu)\leq k \L(\nu). \label{eq.lipschitzlength}
	\end{align}
\end{proposition}

\begin{proof}
	If $\L(\nu)=+\infty$, there is nothing to prove. Otherwise, $\supp \nu$ is compact, and $\supp(f_\sharp \nu)= f(\supp \nu)$. 
Moreover, for any open set $U \subset \RR^d$, since $f^{-1}(U)$ is open, 
\begin{align*}
U\cap \left(\supp f_\sharp \nu\right)\neq \emptyset \Longleftrightarrow \nu(f^{-1}(U))>0  \Longleftrightarrow f^{-1}(U)\cap \left(\supp \nu\right)\neq \emptyset.
\end{align*}

Now, let $U$ be an open set which intersects $\supp(f_\sharp \nu)$. Using that 
\begin{align*}
U\cap f(\supp \nu) \subset f\left(f^{-1}(U)\cap \supp \nu\right)
\end{align*}
we get %\mbox{and}\quad 	f^{-1}(U\cap A) \subset f^{-1}\left(U\cap f(A)\right).
\begin{align*}
	\frac{\H^1\left(U\cap \supp (f_\sharp \nu)\right)}{f_\sharp\nu(U)}= \frac{\H^1\left(U\cap f(\supp \nu))\right)}{\nu(f^{-1}(U))}
									  &\leq \frac{\H^1\left(f\left(f^{-1}(U)\cap \supp \nu\right)\right)}{\nu(f^{-1}(U))}\\
									  &\leq k\frac{\H^1\left(f^{-1}(U)\cap \supp \nu\right)}{\nu(f^{-1}(U))}\\
									  &\leq  k\mathcal{L}(\nu)
\end{align*}
since $f^{-1}(U)$ is an open set which intersects $\supp \nu$.
Taking the supremum over all $U$ yields the claimed inequality.
\end{proof}

It is also possible to express the length-functional using the Besicovitch differentiation theorem~\cite[Thm.~2.22]{ambrosio2000functions}. Assume that $\H^1(\supp \nu)<+\infty$ (otherwise $\mathcal{L}(\nu)=+\infty$). Then, the measure  $\H^1\mres \supp\nu$ is Radon, and the limit
\begin{align}
	D_\nu(\H^1\mres \supp\nu)(x)&\eqdef \lim_{r \to 0^+} \frac{\H^1(B_r(x)\cap \supp \nu)}{\nu(B_r(x))} \\ \bigg(\mbox{resp. } D_{\H^1\mres \supp\nu}(\nu)(x)&\eqdef \lim_{r \to 0^+}\frac{\nu(B_r(x))}{\H^1(B_r(x)\cap \supp \nu)}\bigg) 
\end{align}
exists for $\nu$-a.e. $x$ (resp. $\H^1\mres \supp\nu$-a.e. $x$).
\begin{proposition}[Alternative definitions, II]\label{prop.length.besico}
	Let $\nu \in \mathcal{P}(\RR^d)$ such that $\supp \nu$ is connected and $\H^1(\supp \nu)<+\infty$. Then
\begin{align}
	\mathcal{L}(\nu) &=
	\begin{cases}
		\norm{\frac{\dd \left(\H^1\mres \supp\nu\right)}{\dd \nu}}_{L^{\infty}_\nu} & \mbox{if $\left(\H^1\mres \supp \nu\right) \ll \nu$,}\\
		+\infty &\mbox{otherwise.}
	\end{cases}\label{length_besico1}\\
	&= 
\begin{cases}
	0 & \mbox{if $\supp \nu$ is a singleton,}\\	
	\norm{\left(\frac{\dd \nu}{\dd \left(\H^1\mres \supp\nu\right)}\right)^{-1}}_{L^{\infty}_{\H^1\mres \supp\nu}} & \mbox{otherwise}.
\end{cases}\label{length_besico2}
\end{align}	
\end{proposition}
Notice that in Proposition~\ref{prop.length.besico}, both ``norms'' may take the value $+\infty$, and in \eqref{length_besico2}, we adopt the convention that $1/0=+\infty$.

\begin{proof}[Proof of Proposition \ref{prop.length.besico}]
	First, we prove \eqref{length_besico1}. If $\left(\H^1\mres \supp \nu\right) \ll \nu$ then the Lebesgue-Besicovitch differentiation theorem ensures that 
\begin{align*}
	\H^1\mres \supp \nu = \left(\frac{\dd \left(\H^1\mres \supp\nu\right)}{\dd \nu}\right)\nu \leq \norm{\frac{\dd \left(\H^1\mres \supp\nu\right)}{\dd \nu}}_{L^{\infty}_\nu}\nu.
\end{align*}
Therefore,
\begin{align*}
	\mathcal{L}(\nu)\leq \norm{\frac{\dd \left(\H^1\mres \supp\nu\right)}{\dd \nu}}_{L^{\infty}_\nu} \leq \norm{D^+_\nu(\H^1\mres \supp\nu)}_{\infty}=\mathcal{L}(\nu).
\end{align*}
If $\left(\H^1\mres \supp \nu\right)$ is not absolutely continuous w.r.t. $\nu$, there is no $\alpha>0$ such that $\alpha \nu \geq \H^1\mres \supp\nu$, and $\mathcal{L}(\nu)=+\infty$.

Now, we prove \eqref{length_besico2}. The case where $\supp \nu$ is a singleton is already known. 
We assume now that $\H^1(\supp \nu)>0$, and using the Besicovitch differentiation theorem~\cite[Thm. 2.22]{ambrosio2000functions}, we decompose
\begin{align}
	\nu &=  \theta\H^1\mres \supp \nu + \nu^s,\label{eq-besico-decompo}\\
	\intertext{where}\
	\theta(x)\eqdef \frac{\dd \nu}{\dd \left(\H^1\mres \supp\nu\right)}(x)&= \lim_{r \to 0^+} \frac{\nu(B_r(x))}{\H^1(B_r(x)\cap \supp \nu)} =\left(D^+_\nu(\H^1\mres \supp\nu)(x)\right)^{-1}\nonumber
\end{align}
for $\left(\H^1\mres \supp\nu\right)$-a.e. $x$. From the last equality, we get
\begin{align*}
	\norm{\theta^{-1}}_{L^{\infty}_{\H^1\mres \supp\nu}} \leq \norm{D^+_\nu(\H^1\mres \supp\nu)(x)}_\infty =\mathcal{L}(\nu).
\end{align*}
To prove the converse inequality, we assume $\norm{\theta^{-1}}_{L^{\infty}_{\H^1\mres \supp\nu}}<+\infty$ (otherwise there is nothing to prove). Using~\eqref{eq-besico-decompo}, we note that 
\begin{align*}
	\left(\norm{\theta^{-1}}_{L^{\infty}_{\H^1\mres \supp\nu}}\right)\nu	&\ge \H^1\mres \supp\nu,
\end{align*}
so that $\mathcal{L}(\nu)\le \norm{\theta^{-1}}_{L^{\infty}_{\H^1\mres \supp\nu}}$.
\end{proof}

We may now examine a few examples.
\begin{example}\label{example.diracs_dense_sequence}
	Let $\nu=\sum_{n=1}^{\infty} 2^{-n}\delta_{q_n}$, where $(q_{n})_{n\ge 1}$ is a dense sequence in $\ci{0}{1}$.
	The support being the set of points $x$ such that $\nu(B_r(x)) > 0$ for all $r>0$, one has $\supp \nu = [0,1]$ which is connected. However, using \eqref{lengthalpha}, we see that $\mathcal{L}(\nu)=+\infty$. 
\end{example}

\begin{example}[Densities on a $(\H^1,1)$-rectifiable set]\label{example.densities_rectifiable_set}
	Let $\Sigma \subseteq \RR^d$ be a closed connected set with $0<\H^1(\Sigma)<+\infty$, $\theta\colon \Sigma \rightarrow \RR_+$ a Borel function such that $\displaystyle \int_\Sigma \theta\dd \H^1 < 1$, and let $\nu = \theta \H^1\mres \Sigma + \nu^s$ be a probability measure, where $\supp \nu^s \subset \Sigma$ and the measures $\nu^s$ and $\H^1\mres\Sigma$ are mutually singular.
	Then $\mathcal{L}(\nu)= \norm{1/\theta}_{L^\infty_{\H^1\mres \Sigma}}$: the length functional ignores the singular part.
\end{example}

\begin{example}[Parametrized Lipschitz curves]\label{example.parametrized_curve}
	Let $\gamma\colon \ci{0}{1}\rightarrow \RR^d$ be a non-constant Lipschitz curve, and let $\nu$ such that for all $f \in C_b(\RR^d)$,
	\begin{align}
\langle f, \nu\rangle \eqdef  \frac{1}{\len(\gamma)}\left(\int_0^1f(\gamma(t))\abs{\dot\gamma(t)}\dd t\right),\quad
		\mbox{where } \len(\gamma)\eqdef \int_0^1 \abs{\dot\gamma(t)}\dd t\nonumber
	\end{align}
\end{example}
is the length of the curve.
By the area formula \cite[Thm. 3.2.5]{federer2014geometric}, 
\begin{align*}
	\dd\nu(y)= \frac{1}{\len(\gamma)} \card(\gamma^{(-1)}(y))\dd\left( \H^1\mres \Sigma\right)(y)
\end{align*}
where $\Sigma=\gamma(\ci{0}{1})$. As a result, 
\begin{align}
	\mathcal{L}(\nu) = \frac{\len(\gamma)}{\mathrm{ess}\mbox{-}\min_{y \in\Sigma}\left(\card(\gamma^{(-1)}(y))\right)},
\end{align}
where the minimum is an \emph{essential minimum} with respect to $\H^1\mres \Sigma$.

\subsection{Lower semi-continuity of the length functional}
\label{sec.length_functional_is_lsc}
Now, we prove that $\mathcal{L}$ is the lower semi-continuous envelope of $\ell$ for the narrow convergence.
\begin{proposition}
	\label{proposition.length_functional_is_lsc}
	The functional $\mathcal{L}$ is the lower semi-continuous envelope of $\ell$ for the narrow topology. Moreover, for every $\nu$ such that $\mathcal{L}(\nu)<+\infty$, 
	\begin{align}\label{lenghtineq}
		\H^1(\supp \nu)\leq \mathcal{L}(\nu)
	\end{align}
	with equality if and only if either $\nu=\delta_x$ for some $x \in \RR^d$ (if $\H^1(\supp\nu)=0$), or $\H^1(\supp\nu)>0$ and $\displaystyle \nu=\frac{1}{\H^1(\supp \nu)}\H^1\mres \supp \nu$, \textit{i.e.} $\nu = \nu_\Sigma$ for some $\Sigma \in \mathcal{A}$, as defined in~\eqref{admissible_set}.
\end{proposition}

\begin{proof}[ Proof of Proposition \ref{proposition.length_functional_is_lsc}: ]	
	The inequality \eqref{lenghtineq} is clear from the definition of \eqref{lengthalpha}, so we study the equality case.

	If $\nu=\delta_x$ or $\displaystyle \nu=\frac{1}{\H^1(\supp \nu)}\H^1\mres \supp \nu$ with $\H^1(\supp\nu)>0$, one readily checks that $\mathcal{L}(\nu)=\H^1(\supp \nu)$. 
	Conversely, if \eqref{lenghtineq} is an equality, for every Borel set $B$,
\begin{align*}
	0 &= \mathcal{L}(\nu)-\H^1(\supp\nu) \\
	&= \underbrace{\left(\mathcal{L}(\nu)\nu(B)-\H^1(B\cap \supp\nu)\right)}_{\geq 0}+ \underbrace{\left(\mathcal{L}(\nu)\nu(B^\complement)-\H^1(B^\complement\cap \supp\nu)\right)}_{\geq 0}
\end{align*}
so that both terms must be zero. If $\mathcal{L}(\nu)>0$, we deduce
\begin{align*}
	\forall B\subset \RR^d\ \mbox{Borel},\quad	\nu(B) = \frac{\H^1(B\cap \supp\nu)}{\mathcal{L}(\nu)}=  \frac{\H^1(B\cap \supp\nu)}{\H^1(\supp \nu)}.
\end{align*}
If $\mathcal{L}(\nu)=0$, $\H^1(\supp\nu)=0$ and since $\supp \nu$ is connected, $\nu$ is a Dirac mass.

	Next we prove that $\mathcal{L}$ is sequentially lower semi-continuous. We consider  $(\nu_{n})_{n\in \NN}$  such that $\nu_n \xrightharpoonup[n \to \infty]{} \nu \in \mathcal{P}(\RR^d)$ and we show that $\alpha \eqdef \liminf_{n \to \infty} \mathcal{L}(\nu_n) \geq \mathcal{L}(\nu)$. If $\alpha=+\infty$, we have nothing to prove. Otherwise, up to the extraction of a subsequence, we may assume that $\lim_{n \to \infty} \mathcal{L}(\nu_n)=\alpha$ and that $\mathcal{L}(\nu_n)<+\infty$ for all $n \in \mathbb{N}$. 
	
	Defining the sequence of compact and connected sets  $\Sigma_n \eqdef \supp \nu_n$, it holds that $\H^1(\Sigma_n) \le \mathcal{L}(\nu_n)$, so that
	\[
		\sup_{n\geq N} \H^1(\Sigma_n) \leq \alpha +1 <+\infty
	\]
	for $N$ large enough. Hence, for all $n\ge N$, $\diam(\Sigma_n)\leq \alpha +1$.	In addition, let $x\in \supp \nu$. Since $0<\nu(B_1(x))\leq \liminf_{n\to \infty}\nu_n(B_1(x))$, for all $n$ large enough $\left(\supp \nu_n\right)\cap B_1(x)\neq \emptyset$, thus $\supp \nu_n \subset \overline{B_{\alpha+2}(x)}$. 
	
	Therefore, we may apply Blaschke's Theorem and assume,
	up to extracting a subsequence, that $\Sigma_n \xrightarrow[n \to \infty]{d_H} \Sigma$. From the weak convergence of measures one has $\supp \nu \subset \Sigma$. Let us show that $\supp \nu = \Sigma$.  If $\Sigma$ is a singleton $\{x_0\}$, we have $\nu = \delta_{x_0}$. Otherwise, Theorem~\ref{theorem.Golab_localversion} implies that $\Sigma \in \mathcal{A}$ and furthermore, as $\mathcal{L}(\nu_n) \nu_n \ge \H^1\mres \Sigma_n$, that 
 	\begin{equation}
 		\label{nu_greater_H1}
 		\alpha \nu \ge \H^1\mres \Sigma. 
 	\end{equation}
 	Hence, as $\Sigma$ is connected, for all $z \in \Sigma$ it holds $\nu(B_r(z)) > 0$, confirming that $\supp \nu = \Sigma$. Finally from \eqref{nu_greater_H1} we get that
 	\[
 		\liminf_{n \to \infty} \mathcal{L}(\nu_n)
 		=
 		\alpha 
 		\ge 
 		\mathcal{L}(\nu),
 	\]
 	proving that $\mathcal{L}$ is l.s.c.
	
	As a result, we have proved that $\mathcal{L}$ is l.s.c. and that $\mathcal{L}\equiv \ell$ on the effective domain of $\ell$. To show that $\mathcal{L}$  is the l.s.c. enveloppe of $\ell$, we  prove that it is above any l.s.c. functional $\G \le \ell$. Let $\nu \in \P(\RR^d)$. If $\mathcal{L}(\nu)=+\infty$, we have $\G(\nu)\leq \mathcal{L}(\nu)$. If $\mathcal{L}(\nu)<+\infty$, using Lemma~\ref{lemma.approximation_lemma} below, we can find a sequence  $\nu_{\Sigma_n} \xrightharpoonup[n \to \infty]{} \nu$ such that $\H^1(\Sigma_n) \rightarrow \mathcal{L}(\nu)$. The lower semi-continuity of $\G$ yields
	\begin{equation*}
\G(\nu) \le \liminf_{n \to \infty} \G(\nu_{\Sigma_n }) 
		\le \liminf_{n \to \infty} \ell(\nu_{\Sigma_n })
		=  \liminf_{n \to \infty} \H^1(\Sigma_n)
		= \mathcal{L}(\nu).
	\end{equation*}
\end{proof}

The proof of Proposition~\ref{proposition.length_functional_is_lsc} relies on the following approximation Lemma.

\begin{lemma}
	\label{lemma.approximation_lemma}
	Let $\nu \in \mathcal{P}(\RR^d)$ such that $\mathcal{L}(\nu) < \infty$. There exists a sequence $\left(\Sigma_n\right)_{n \in \mathbb{N}} \subset \mathcal{A}$ such that 
	\begin{itemize}
		\item $\Sigma_n \xrightarrow[n \to \infty]{d_H} \supp \nu$, \\
		\item $\nu_{\Sigma_n} \xrightharpoonup[n \to \infty]{} \nu$ and $W_p(\nu_{\Sigma_n}, \nu) \xrightarrow[n \to \infty]{} 0$ for any $p \ge 1$, where $\nu_{\Sigma_n}$ is defined as in~\eqref{admissible_set}. 
	\end{itemize}
	We also have $\H^1(\Sigma_n) \xrightarrow[n \to \infty]{} \mathcal{L}(\nu)$ and if, in addition $\mathcal{L}(\nu) > 0$, we can take $\H^1(\Sigma_n) = \mathcal{L}(\nu)$ for all $n \in \mathbb{N}$.
\end{lemma}
%\joao{Je ne comprends pas la remarque du rapporteur, on dit justement que $\H^1(\Sigma_n) = \mathcal{L}(\nu)$, justement si $\mathcal{L}(\nu) > 0$}
 
\begin{proof}
	 To simplify the notation, we set $\alpha = \mathcal{L}(\nu)$ and $\Sigma = \supp\nu$. 
	 For $\alpha=0$ (that is, $\nu=\delta_{x_0}$ for some $x_0$), we consider 
	\[
	\Sigma_n = x_0 + \ci{0}{1/n}\times \{0\}^{d-1}
	\]
	which provides the desired approximation, with $\H^1(\Sigma_n) = 1/n \to 0 = \mathcal{L}(\delta_{x_0})$.
	
	For $\alpha>0$,  we start by covering the  space with cubes of the form
	\[
	Q_{z,n} \eqdef \frac{1}{n}\left( z + [0,1)^d\right), \ \text{ for $z \in \mathbb{Z}^d$}.
	\] 
	For some fixed $n$, let $\left(Q_{i,n}\right)_{ i \in I_n}$ be the collection of the cubes such that $\nu\left(Q_{z,n}\right) > 0$, since the set $\Sigma$ is compact, $I_n$ is finite for a given $n$. We define the quantities
	\[
	m_{i,n} \eqdef \alpha \nu(Q_{i,n}) - \H^1(\Sigma \cap Q_{i,n}) \le \alpha,
	\]
	as the excess mass of $\nu$ in the cube $Q_{i,n}$ (note that $m_{i,n}\geq 0$ in view of \eqref{lengthalpha}). Our strategy is to modify $\nu\mres Q_{i,n}$ by adding segments with uniform measure inside the cube and having a total length equal to the excess mass $m_{i,n}$. 

	If $\Sigma \cap \inte Q_{i,n} \neq \emptyset$, take $x_i$ in this intersection, so that $B_{\delta_i}(x_i) \subset Q_{i,n}$ for some $\delta_i > 0$. Then, set 
	$
	\disp
	N_{i,n} \eqdef \left\lceil \frac{m_{i,n}}{\delta_i} \right\rceil,
	$
and choose $\delta_{i,j}\ge 0$ for $j = 1, \dots, N_{i,n}$ such that
	\[
	\sum_{j = 1}^{N_{i,n}} \delta_{i,j} = m_{i,n}, \text{ and } 0 \le \delta_{i,j} < \delta_i.
	\]
	Since $\H^1(\Sigma\cap Q_{i,n})<+\infty$, it is possible to choose $N_{i,n}$ vectors $v_{i,j} \in \mathbb{S}^{d-1}$ such that the segments $S_{i,j} \eqdef [x_i, x_i + \delta_{i,j}v_{i,j}]$ are contained in $\inte Q_{i,n}$ and satisfy  $\H^1(\Sigma \cap S_{i,j}) = 0$ , for $j = 1, \dots, N_{i,n}$.

	If $\Sigma \cap \inte Q_{i,n} = \emptyset$, as the cubes have positive mass, it means that $\nu$ is concentrated on the boundary of the cube, in which case we take $x_i \in \Sigma \cap \partial Q_i$ and any family of segments entering the cube will suffice. 

	Next, we define the measures 
	\[
	    \nu_{\Sigma_n} 
	    \eqdef 
	    \frac{1}{\H^1(\Sigma_n)} \H^1\mres \Sigma_n
	    \text{ for } 
	    \Sigma_n \eqdef \disp \Sigma \cup \bigcup_{i \in I_n}\bigcup_{j = 1}^{N_{i,n}} S_{i,j}.
	\]
	From the construction, the Hausdorff distance between $\Sigma$ and $\Sigma_n$ is at most the diagonal of the cube $[0,1/n)^d$, so that
	\[  
	    d_H(\Sigma, \Sigma_n) \le \frac{\sqrt{d}}{n} \xrightarrow[n \to \infty]{} 0,
	\]
	and the total length of $\Sigma_n$ is given by
	\begin{align*}
		\H^1(\Sigma_n) 
		&= 
		\sum_{i \in I_n} \H^1(\Sigma \cap Q_{i,n})
		+ 
		\sum_{i \in I_n}\sum_{j = 1}^{N_{i,n}}\H^1(S_{i,j})\\ 
		&= 
		\sum_{i \in I_n}
		\H^1(\Sigma \cap Q_{i,n}) + m_{i,n}
		= 
		\alpha
		\sum_{i \in I_n}
		\nu(Q_{i,n}) = \alpha.
	\end{align*}
	Each $\Sigma_n\in \A$ since it is connected and compact (as a finite union of compact sets).

	To finish the proof, it remains  to show that $\nu_{\Sigma_n} \xrightharpoonup[n \to \infty]{} \nu$. By construction, there exists a compact set $K \subset \RR^d$ such that $(\supp \nu)\cup \bigcup_{n\geq 1} \left(\supp \nu_{\Sigma_n}\right) \subset K$. Then any function $\phi \in C_b(\RR^d)$ is uniformly continuous on $K$, and we denote by $\omega$ its modulus of continuity. 
Observing that $\nu_{\Sigma_n}(Q_{i,n}) = \nu(Q_{i,n})$, we note that
	\begin{align*}
		\left| 
		\int_{\RR^d}\phi\dd\nu_{\Sigma_n} 
		- 
		\int_{\RR^d}\phi\dd\nu
		\right|
		\le 
		\sum_{i \in I_n}
		\left| 
		\int_{Q_{i,n}}\phi\dd\nu_{\Sigma_n} - \int_{Q_{i,n}}\phi\dd\nu
		\right|\\ 
		\le
		\sum_{i \in I_n} \omega(\text{diam}Q_{i,n})\nu(Q_{i,n}) 
		\le
		\omega\left(\sqrt{d}/{n}\right) \xrightarrow[n \to \infty]{} 0.
	\end{align*}

Hence $\nu_{\Sigma_n} \xrightharpoonup[n \to \infty]{} \nu$. But as the support of all such measures is contained in the compact $K$ and the Wasserstein distance metrizes the weak convergence in $\mathcal{P}_p(K)$, see \cite[Thm.~5.10]{santambrogio2015optimal}, it holds that $W_p(\nu_{\Sigma_n}, \nu) \xrightarrow[n \to \infty]{} 0$. 
%	Conversely, \revision{if there exists $\Sigma_n$ such that $\nu_{\Sigma_n} \xrightharpoonup[n \to \infty]{\star} \nu$ and $\H^1(\Sigma_n) \le \alpha$ then $\alpha \nu_{\Sigma_n} \ge \H^1\mres \Sigma_n$, hence Theorem \ref{theorem.Golab_localversion} implies that $\alpha \nu \ge \H^1\mres \Sigma$ and item (3) implies that $\mathcal{L}(\nu) < \infty$.}\vincent{Ok, mais cette réciproque est donnée directement par la semi-continuité inférieure, non? (invoquer la proposition 3.1 suffit) } 
\end{proof}

\begin{remark}
	The conclusions of Proposition~\ref{proposition.length_functional_is_lsc} and Lemma~\ref{lemma.approximation_lemma} still hold when replacing the narrow topology with the local weak-$\star$ topology.
\end{remark}

\subsection{A relaxed problem with existence of solutions}
The relaxed problem~\eqref{problem.shape_optimization_relaxed} introduced on page~\pageref{problem.shape_optimization_relaxed} is defined by replacing $\ell$ in the orginal problem with its l.s.c.~envelope $\mathcal{L}$. We define the energy $\E(\nu) \eqdef W_p^p(\rho_0, \nu) + \Lambda \mathcal{L}(\nu)$, and with a slight abuse of notation, we sometimes write $\mathcal{E}(\Sigma) = \mathcal{E}(\nu_{\Sigma})$ for $\Sigma \in \A$. 
The main point of considering this relaxed problem is that the existence of solutions for~\eqref{problem.shape_optimization_relaxed} follows from the direct method of the calculus of variations. 

\begin{theorem}\label{theorem.existence_relaxed_problem}
	The relaxed problem \eqref{problem.shape_optimization_relaxed} admits a solution. In addition, $\mathcal{E}$ is the l.s.c.~enveloppe of $W_p^p(\rho_0, \cdot) + \Lambda \ell$, and:
	%so there is no gap between the original and the relaxed formulation 
	\[
		\inf \eqref{problem.shape_optimization} = 
		\min \eqref{problem.shape_optimization_relaxed}.	
	\]
\end{theorem}
\begin{proof}
	%To show existence we apply the direct method of the calculus of variations. 
	Let $\left(\nu_n\right)_{n \in \mathbb{N}}$ be a minimizing sequence for $\mathcal{E}$. Since $\left(\sup_n W_p^p(\rho_0, \nu_n)\right)<+\infty$, the moments of order $p$ of $\nu_n$ are uniformly bounded (see for instance \cite[Thm. 5.11]{santambrogio2015optimal}), and we may then extract a (not relabeled) subsequence converging to some $\nu \in \P(\RR^d)$ in the narrow topology (by Prokhorov's theorem). From Proposition~\ref{proposition.length_functional_is_lsc} and the fact that the Wasserstein distance is lower semi-continuous,  the functional $\mathcal{E}$ is l.s.c.~and we have that 
	\[
		\mathcal{E}(\nu) \le \liminf_{n \to \infty} \mathcal{E}(\nu_n) = \inf \eqref{problem.shape_optimization_relaxed}. 
	\]
	The measure $\nu$ is a minimizer of \eqref{problem.shape_optimization_relaxed}.

	To show that $\E$ is the l.s.c.~enveloppe of the original energy one may argue as in the proof of Proposition~\ref{proposition.length_functional_is_lsc}. Consider any l.s.c.~functional $\G$  such that 
	\begin{align*}
		\forall \nu \in \P(\RR^d),\quad		\G(\nu)\leq W_p^p(\rho_0,\nu) +\Lambda \ell(\nu).
	\end{align*}
	For every $\nu$ with $\mathcal{L}(\nu)<+\infty$, we use Lemma~\ref{lemma.approximation_lemma} to build a sequence $(\nu_n)_{n \in \NN}$ such that $W_p^p(\rho_0,\nu_{\Sigma_n}) \to W_p^p(\rho_0,\nu)$. Indeed, as $\nu_{\Sigma_n}$ converges to $\nu$ for the Wasserstein distance, the triangle inequality gives
	\[
		\left|
		W_p(\rho_0, \nu_{\Sigma_n}) - W_p(\rho_0, \nu)
		\right| 
		\le
		W_p(\nu_{\Sigma_n}, \nu) \xrightarrow[n \to \infty]{} 0. 
	\]
	Hence for any $\nu \in \mathcal{P}_p(\RR^d)$ it holds that
\begin{align*}
	\G(\nu)\leq \liminf_{n \to \infty} \left(W_p^p(\rho_0,\nu_{\Sigma_n}) 
	+
	\Lambda \ell(\nu_{\Sigma_n})\right)= W_p^p(\rho_0,\nu) +\Lambda \mathcal{L}(\nu)=\E(\nu),
\end{align*}
and we conclude that $\mathcal{E}$ is the l.s.c.~envelope.
%and the no gap property follows from the general theory of l.s.c.~relaxation, see {\em e.g.}~\cite{attouch2014variational}. {\color{orange} QUE VOULEZ VOUS DIRE?}
\end{proof}

\section{On the support of optimal measures}\label{section.geometric_properties_transport_map}
Our goal for this section is to answer to the question of ``how small'' $\Lambda$ must be in Theorem \ref{theorem.the_big_one}. For this, in Theorem~\ref{theorem.when_sols_are_diracs} we study when solutions of the relaxed problem \eqref{problem.shape_optimization_relaxed} are Dirac masses. Keeping this in mind, the rest of this section can be skipped and the reader can move on to the main results of the paper. 

The following notation will be useful: a point $x_0$ is said to be a {\em p-mean} of $\rho_0$ if 
\[
	x_0 \in 
	\argmin_{y \in \RR^d} \int_{\RR^d}|x - y|^p\dd \rho_0(x)
	= 
	\argmin_{y \in \RR^d} W_p(\rho_0, \delta_y). 
\]
A 2-mean is just the mean of $\rho_0$, that is, $m_{\rho_0} \eqdef \disp \int_{\RR^d}x \dd \rho_0(x)$. For $p>1$, the p-mean is uniquely defined, but for $p=1$ the collection of $1$-means is a closed convex set which is not reduced to a singleton in general. 

\begin{theorem}\label{theorem.when_sols_are_diracs}
	For a fixed measure $\rho_0 \in \mathcal{P}_p(\RR^d)$ there exists a critical parameter $\Lambda_\star \in [0, \infty)$ such that
	\begin{itemize}
		\item for $\Lambda < \Lambda_\star$ no solution of $(\mathcal{P}_\Lambda)$ is a Dirac measure;
		\item for $\Lambda > \Lambda_\star$ it holds that $\disp \argmin (\mathcal{P}_\Lambda)$ is the set of $p$-means of $\rho_0$. 
	\end{itemize}
	Moreover, $\Lambda_\star=0$ if and only if $\rho_0$ is a Dirac mass.
\end{theorem}

We start by studying the support of the optimal measure, showing that it is contained in the convex hull of the support of $\rho_0$. In the sequel the proof of Theorem \ref{theorem.when_sols_are_diracs} will be divided in several steps. We end the section with an exemple of $\rho_0$ composed of $2$ Dirac masses.  

\subsection{Elementary properties of the support}
Given a set $A \subset \RR^d$ we denote by $\cconv A$ its closed convex hull.
\begin{lemma}\label{lem.convexhull}
  Let $\nu \in\P(\RR^d)$ be a solution to \eqref{problem.shape_optimization_relaxed}. Then the following properties hold
  \begin{itemize}
	  \item[(1)] $\H^1(\supp \nu)\leq \frac{1}{\Lambda}W^p_p(\rho_0,\delta_{m_{\rho_0}})$, where $m_{\rho_0}$ is any $p$-mean of $\rho_0$. In particular, $\Sigma$ is contained in a ball of diameter $d_0\eqdef \frac{1}{\Lambda}W^p_p(\rho_0,\delta_{m_{\rho_0}})$.
	  \item[(2)] $\supp \nu \subset \cconv\left(\supp \rho_0\right)
	  \cap B\big({m_{\rho_0}},2W_p(\rho_0,\delta_{m_{\rho_0}})
	  + \tfrac{2}{\Lambda}W_p^p(\rho_0,\delta_{m_{\rho_0}})\big)$
  \end{itemize}
\end{lemma}

\begin{proof}
	For the first point, let $\Sigma$ denote the support of $\nu$. Since $\nu$ has finite energy we have that $\mathcal{L}(\nu) \geq \H^1(\Sigma)$. Thus, since it is also optimal 
	\begin{equation}\label{eq:estimsigma}
		\Lambda \H^1(\Sigma)
		\leq 
		W^p_p(\rho_0,\nu)+\Lambda \mathcal{L}(\nu) \le
		W^p_p(\rho_0,\delta_{m_{\rho_0}})+\Lambda \L(\delta_{m_{\rho_0}}) = W^p_p(\rho_0,\delta_{m_{\rho_0}}).
	\end{equation}

	For the second point, let $C\eqdef \cconv\left(\supp \rho_0\right)$. It is a nonempty closed convex set, therefore the projection  onto $C$ is well-defined and 1-Lipschitz. We denote it by $f$. By Proposition~\ref{prop.length.lipschitz}, it holds that $\L(\nu)\geq \L(f_\sharp \nu)$. Moreover, for every $(x,y)\in C\times \RR^d$,    
	\begin{align*}
		\abs{x-y}^2 &=\abs{x-f(y)}^2 + \abs{f(y)-y}^2 +2\underbrace{\left\langle{x-f(y)},{f(y)-y}\right\rangle}_{\geq 0}\geq \abs{x-f(y)}^2
	\end{align*}
	with equality if and only if $y \in C$. As a result, if $\gamma$ is an optimal transport plan for $(\rho_0,\nu)$,
		\begin{align*} 
			W_p^p(\rho_0,\nu)= \int \abs{x-y}^p\dd \gamma(x,y) &\geq \int \abs{x-f(y)}^p\dd \gamma(x,y) \\
									&=\int \abs{x-y}^p\dd\left((\text{id},f)_\sharp \gamma\right)(x,y) \geq W_p^p(\rho_0,f_\sharp\nu) 
		\end{align*}
    with strict inequality unless $y \in C$ for $\gamma$-a.e. $(x,y)$ (hence $\nu$-a.e. $y$).
    
   But $\nu$ is a solution to  \eqref{problem.shape_optimization_relaxed}, therefore the inequality
   \begin{align*}
	W_p^p(\rho_0,\nu)+ \Lambda \L(\nu)\geq  W_p^p(\rho_0,f_\sharp\nu)+ \Lambda \L(f_\sharp \nu)
   \end{align*}
    cannot be strict. We deduce that $y \in C$ for $\nu$-a.e. $y$, and $C$ being closed, that $\Sigma \subset C$.
    
    Additionally, from~\eqref{eq:estimsigma}, we have
    $W_p(\nu,\delta_{m_{\rho_0}})\le 2 W_p(\rho_0,\delta_{m_{\rho_0}})$ and
    in particular there are points  $y\in \Sigma$ such that 
    $|y-m_{\rho_0}|\le 2 W_p(\rho_0,\delta_{m_{\rho_0}})$. Combined
    with the first point, we obtain that $\Sigma\subset
    B\big({m_{\rho_0}},2W_p(\rho_0,\delta_{m_{\rho_0}})
	  + \tfrac{2}{\Lambda}W_p^p(\rho_0,\delta_{m_{\rho_0}})\big)$.
\end{proof}

\begin{example}
    Let $\rho_0= \delta_{x_0}$ for some $x_0 \in \RR^d$. Then both
	conditions from Lemma~\ref{lem.convexhull} are sharp and 
	characterize for all $\Lambda>0$ the unique solution $\delta_{x_0}$ 
	of~\eqref{problem.shape_optimization_relaxed}.
\end{example}
%\joao{Est-ce exemple vraiment nécessaire? C'est évidant que l'argmin est $\{\delta_{x_0}\}$.}

\subsection{When solutions are Dirac masses} 
Now, we discuss whether or not Dirac masses are solutions in the case where $\rho_0$ is not a Dirac measure. 
We start with the following Lemma. 

\begin{lemma}\label{lemma.trashold_lambda}
	Let $\Lambda > 0$ such that $\delta_{x_0} \in \argmin \left(\P_\Lambda\right)$, for $\Lambda' > \Lambda$ it holds
	\begin{itemize}
		\item for $p > 1$ that $\delta_{x_0}$ is the unique solution of $\left(\P_{\Lambda'}\right)$,
		\item for $p = 1$ that $\argmin \left(\P_{\Lambda'}\right)$ consists of only Dirac masses. 
	\end{itemize}
\end{lemma}
\begin{proof}
	If $\delta_{x_0} \in \argmin \left(\P_\Lambda\right)$, for any $p \ge 1$, and for any measure $\nu$ with $\mathcal{L}(\nu) > 0$ it holds that
	\begin{align*}
		W_p^p(\rho_0, \delta_{x_0})
		\le 
		W_p^p(\rho_0, \nu) + \Lambda \mathcal{L}(\nu) 
		<
		W_p^p(\rho_0, \nu) + \Lambda' \mathcal{L}(\nu),
	\end{align*}
	and hence $\nu$ cannot be a minimizer of $\left(\P_{\Lambda'}\right)$. Then for any $p \ge 1$ it holds that $\argmin \left(\P_{\Lambda'}\right)$ consists of Dirac measures. Whenever $p > 1$, the function $y \mapsto W_p^p(\rho_0, \delta_y)$ is strictly convex and hence $\argmin \left(\P_{\Lambda'}\right)$ is a singleton. 
\end{proof}

This simple Lemma allows for the definition of the critical value $\Lambda_\star$ as follows
\begin{equation}
	\Lambda_\star
	\eqdef
	\inf
	\left\{
		\Lambda \ge 0: 
		\argmin \left(\P_\Lambda\right)
		\subset 
		\left(
			\delta_{x}\right)_{x \in \RR^d}
	\right\}. 
\end{equation}
As stated in Theorem~\ref{theorem.when_sols_are_diracs}, $\Lambda_\star > 0$ whenever $\rho_0$ is not a single Dirac mass, which is a direct consequence of the convergence of solutions to $\rho_0$ when $\Lambda$ goes to $0$.

\begin{lemma}\label{lemma.cvLambdazero}
	For every $\rho_0 \in \mathcal{P}_p(\RR^d)$, and $\Lambda>0$, let $\nu_\Lambda$ be any solution to \eqref{problem.shape_optimization_relaxed}. Then 
\begin{equation*}
    	\nu_\Lambda\xrightharpoonup[\Lambda \to 0^+]{} \rho_0.
\end{equation*}
In particular, $\Lambda_\star>0$ unless $\rho_0$ is a Dirac mass.
\end{lemma}
\begin{proof}
	If $\L(\rho_0)<+\infty$, it suffices to notice that
	\begin{align*}
		W_p^p(\rho_0, \nu_\Lambda)\leq W_p^p(\rho_0, \nu_\Lambda) + \Lambda \mathcal{L}(\nu_\Lambda) \leq 	W_p^p(\rho_0, \rho_0) + \Lambda \mathcal{L}(\rho_0) = \Lambda \mathcal{L}(\rho_0)\xrightarrow[\Lambda\to 0^+]{} 0.
	\end{align*}
	However, we need to handle the case where $\L(\rho_0)=+\infty$.

Let $\varepsilon>0$. By the density of discrete measures in the Wasserstein space, there exists a probability  measure of the form $\mu =\sum_{i=1}^{N} a_i \delta_{x_i}$ such that $W_p^p(\rho_0,\mu)\leq \varepsilon$. We may assume that $N\geq 2$. By connecting all the points $\{x_{i}\}_{1\leq i\leq N}$, we obtain a compact connected set $\Sigma$ with $0<\H^1(\Sigma)<+\infty$. For every $\theta \in \oi{0}{1}$,
 we then define
\begin{align*}
	\tilde{\rho}_0\eqdef \frac{\theta}{\H^1(\Sigma)} \H^1 \mres\Sigma + (1-\theta)\mu 
	= \theta \nu_\Sigma + (1-\theta)\mu.
\end{align*}
and we note that 
$\L(\tilde{\rho}_0)\le \frac{\H^1(\Sigma)}{\theta}<+\infty$. 

By the optimality of $\nu_{\Lambda}$,
\begin{align*}
	W_p^p(\rho_0,\nu_{\Lambda})\le \Lambda \L(\nu_{\Lambda})+ W_p^p(\rho_0,\nu_{\Lambda})\le \Lambda \L(\tilde{\rho}_0)+ W_p^p(\rho_0,\tilde{\rho}_0).
\end{align*}

Taking the upper limit as $\Lambda\to 0^+$, and using the convexity of the Wasserstein distance yields
\begin{align*}
	\limsup_{\Lambda \to 0^+} W_p^p(\rho_0,\nu_\Lambda)
	\le W_p^p(\rho_0,\tilde{\rho}_0) \le 
	\theta     W_p^p\left(\rho_0,\nu_\Sigma\right) + 
	(1-\theta) W_p^p(\rho_0,\mu).
\end{align*}
Letting $\theta\to 0^+$ we obtain $\limsup_{\Lambda \to 0^+} W_p^p(\rho_0,\nu_\Lambda)\leq \varepsilon$. Since $\varepsilon$ is arbitrary, the claim follows.

For the last statement, we note that $\supp \rho_0$ must
be included in the Kuratowski limits of $\supp \nu_\Lambda$ as $\Lambda\to 0$, so that if $\rho_0$ is not a Dirac mass, neither is $\nu_\Lambda$  for $\Lambda>0$ small enough. 
\end{proof}

Next, we show that for $\Lambda$ large enough, the solution becomes a Dirac measure. 

\begin{proposition}\label{proposition.critical_lambda_upper_bound}
For every $\rho_0 \in \mathcal{P}_p(\RR^d)$, $\Lambda_\star<+\infty$.
\end{proposition}

\begin{proof}
	%Up to a change of the origin, we may assume that $\int_{\RR^d}x\dd \rho_0(x)=0$. 
	Choose $\nu \in  \argmin \eqref{problem.shape_optimization_relaxed}$, let $\Sigma\eqdef \supp \nu$ and $y_0\in\Sigma$.
%we define $y_0 \in  \argmin_{y \in \Sigma}\abs{y}$.

	Let $r\eqdef \min\enscond{r'\geq 0}{\Sigma \subset B(y_0,r')}$: since
	$\Sigma$ is connected one has $r\leq \H^1(\Sigma)<+\infty$.
The convexity of the $p$-norm yields
	\begin{align*}
		\forall x,y \in \RR^d, \quad \abs{x-y}^p \geq \abs{x-y_0}^p -p\abs{x-y_0}^{p-1}\abs{y-y_0}.
	\end{align*}
	As a result, if $\gamma$ is an optimal transport plan for $(\rho_0,\nu)$,
\begin{align*}
	\E(\nu)&= \int_{\RR^d\times \RR^d}\abs{x-y}^p\dd \gamma(x,y)+\Lambda \L(\nu)\\
	       &\geq  \int_{\RR^d\times \RR^d}\abs{x-y_0}^p\dd \gamma(x,y) - p\int_{\RR^d\times \RR^d}\abs{x-y_0}^{p-1}\abs{y-y_0}\dd \gamma(x,y)+\Lambda \H^1(\Sigma)\\
	       &\geq \E(\delta_{y_0})+ r \left(\Lambda - p\int_{\RR^d}\abs{x-y_0}^{p-1}\dd \rho_0(x) \right).
\end{align*}
By optimality of $\nu$, we have $\E(\nu)\leq \E(\delta_{y_0})$, so that $r=0$ and $\nu$ is a Dirac mass a soon as $\left(\Lambda - \displaystyle p\int_{\RR^d}\abs{x-y_0}^{p-1}\dd \rho_0(x)\right)>0$.

On the other hand as soon as $r>0$, this expression
must be negative, and it follows that
\[
    \Lambda_\star \le p\int_{\RR^d}\abs{x-y_0}^{p-1}\dd \rho_0(x). 
\]
Note that this bound depends on $\nu$ (through $\Sigma$) and
therefore also on $\Lambda$.
Yet, as observed in the proof of
Lemma~\ref{lem.convexhull}, point (2), we can choose $y_0\in\Sigma$
with $|y_0-m_{\rho_0}|\le 2W_p(\delta_{m_{\rho_0}},\rho_0)$. It follows that
\[
    \Lambda_\star \le \max_{y_0
    \in B(m_{\rho_0},2W_p(\delta_{m_{\rho_0}},\rho_0))}
    \left(p\int_{\RR^d}\abs{x-y_0}^{p-1}\dd \rho_0(x)\right),
\]
which is a (pessimistic) a priori bound depending only on $\rho_0$.
\end{proof}

\begin{remark}
	In some cases, it is possible to provide sharper bounds on $\Lambda_\star$:
	\begin{itemize}
		\item If $p=1$, we see that $\Lambda_\star\leq 1$.
		\item If $p=2$, it can be shown by a simple translation argument that $\nu$ and $\rho_0$ have the same barycenter. Then, one may adapt the above argument to get $\Lambda_\star \leq 2\int\abs{x-m_{\rho_0}}\dd \rho_0(x)$, where $m_{\rho_0}=\int x\dd \rho_0(x)$.
%% THE NEXT ITEM IS TRIVIAL SINCE WE KNOW THAT y_0 IS IN conv supp rho_0
%% so that |x-y_0| <= diam supp rho_0 a.e.
%		\item If $\supp \rho_0$ is bounded,  it is possible to obtain $\Lambda_\star \leq  p \left(\diam(\supp \rho_0)\right)^{p-1}$ for any $p\geq 1$, by exploiting the Lipschitzianity of the dual potentials: there exists $(\phi,\psi)$, solution to the dual Kantorovitch problem (see \cite[Sec. 1.2]{santambrogio2015optimal})
%			\begin{align*}		W_p^p(\mu, \nu)
%		= 		\max
%		\left\{
%			\int \phi \dd \mu + \int \psi \dd \nu: 
%			\begin{array}{c}
%				\phi \in L^1(\mu), \psi \in L^1(\nu),\\ 
%				\phi(x) + \psi(y) \le |x - y|^p
%			\end{array}
%		\right\},
%	\end{align*}
%such that $\Lip(\psi)\leq p \left(\diam(\supp \rho_0)\right)^{p-1}$. Then, 	%		
%	\begin{align*}
%		W_p^p(\rho_0, \delta_{y_0}) - W_p^p(\rho_0, \nu)
%		&\le 
%		\psi(y_0) - \int_\Sigma \psi \dd \nu \le 
%		\int_{\Sigma} \left|\psi(y_0) - \psi(x)\right| \dd\nu(x)\\ 
%		&\le 
%		\Lip(\psi) \cdot \H^1(\Sigma) \leq \Lambda \mathcal{L}(\nu)
%	\end{align*}
%and for  $\Lambda >\Lip(\psi$), the last inequality is strict,  yielding the contradiction $\E(\delta_{y_0})<\E(\nu)$,  unless $\H^1(\Sigma)=0$.}
	\end{itemize}
\end{remark}

%%%%%
%% Tighter bounds with barycenter property
%%%%%
\if{
A simple argument shows that for $\Lambda$ large enough, the solution becomes a Dirac measure.

\begin{proposition}\label{proposition.critical_lambda_upper_bound}
	For any $p \ge 1$, there is a constant $L_p$ depending on $p$ and $\rho_0$ such that for $\Lambda > \L_p$ the solutions of $\left(\P_\Lambda\right)$ are of the form $\delta_{x_0}$. In particular $\Lambda_\star \le L_p$.
\end{proposition}
\begin{proof}
	To prove this result, we use the dual formulation of OT, see \cite{villani2009optimal}. Given two measures $\mu, \nu$ it holds that
	\begin{equation}\label{dual_formulation_OT}
		W_p^p(\mu, \nu)
		= 
		\max
		\left\{
			\int \phi \dd \mu + \int \psi \dd \nu: 
			\begin{array}{c}
				\phi \in L^1(\mu), \psi \in L^1(\nu),\\ 
				\phi(x) + \psi(y) \le |x - y|^p
			\end{array}
		\right\},
	\end{equation}
	where the maximum is always attained in $L^1$. It is also known that the optimal pair $(\phi, \psi)$ can be chosen to have the same modulus of continuity as the cost $c(x,y)$, see \cite[Sec. 1.2]{santambrogio2015optimal}. Since the optimal $\Sigma \subset \co (\supp \rho_0)$ we define
	\[
		L_p \eqdef
		\sup_{y,y'\in \co( \supp \rho_0)}	
		\frac{|y - y'|^p}{|y - y'|}
		=
		\begin{cases}
			p\diam\left(\co( \supp \rho_0)\right)^{p-1},& p>1\\ 
			1& p=1, 
		\end{cases}
	\]
	so that optimal pairs $(\phi, \psi)$ are $L_p$ Lipschitz. 

	Given $\nu$ such that $\mathcal{L}(\nu) < \infty$, with $\nu$ supported over $\Sigma$, we take some $x_0 \in \Sigma$ and let $(\phi, \psi)$ be an optimal pair for the p-Wasserstein distance between $\rho_0$ and $\delta_{x_0}$. 
	
	In particular, $(\phi, \psi)$ is admissible for the dual formulation of $W_p^p(\rho_0, \nu)$ and it holds that 
	$
	\disp	\int \phi \dd \rho_0 + \int \psi \dd \nu \le W_p^p(\rho_0, \nu). 
	$
	And therefore the following estimation holds
	\begin{align*}
		W_p^p(\rho_0, \delta_{x_0}) - W_p^p(\rho_0, \nu)
		&\le 
		\psi(x_0) - \int_\Sigma \psi \dd \nu \le 
		\int_{\Sigma} \left|\psi(x_0) - \psi(x)\right| \dd\nu(x)\\ 
		&\le 
		L_p \cdot \H^1(\Sigma) < \Lambda \mathcal{L}(\nu).
	\end{align*}
	It then holds that any such $\nu$ cannot be optimal and hence any solution must be of the form $\delta_{x_0}$. 
\end{proof}
\begin{remark}
	Notice that $L_p$ can be $+\infty$ for $p>1$ if $\rho_0$ is not compactly supported, in this case the result is non informative. However, this argument still holds for a more general cost of the form $c(x,y) = h(|x-y|)$, with the constant $L_p$ being replaced with the Lipschitz constant of the optimal pairs $(\phi, \psi)$, which depends on the cost and $\rho_0$. 
\end{remark}

In the case $p = 2$ we can give an tighter bound for $\Lambda_\star$. 
\begin{proposition}\label{proposition.critical_lambda_upper_bound_tighter}
	For any $\rho_0 \in \mathcal{P}_2(\RR^d)$ it holds that $\Lambda_\star \le 2\int\abs{x - m_{\rho_0}}\diff \rho_0$. 
\end{proposition}

We rely on the following two observations.
\begin{lemma}\label{lemma.longueur}
	Let $\nu$ be a solution to \eqref{problem.shape_optimization_relaxed} with $p = 2$. Then the measures $\nu_\star$ and $\rho_0$ have the same center of mass:
			$\disp 
				\int y \diff \nu_\star = m_{\rho_0}.
			$
\end{lemma}

\begin{proof}
	For the second point, we set $\nu_{\star, a}\eqdef (\text{id} + a)_{\sharp}\nu_\star$  and compare the energy of $(\alpha,\nu_\star)$ and $(\alpha,\nu_{\star,a})$ for all vectors $a \in \RR^d$. Observe that if $\gamma$ is the optimal transport map from $\rho_0$ to $\nu_\star$, then $\gamma_a \eqdef (\text{id}, \text{id} + a)_{\sharp}\gamma$ is the optimal transport from $\rho_0$ to $\nu_{\Sigma+a}$.
	
	Indeed, we verify that the support of $\gamma_a$ is a c-CM set, which is a sufficient condition for optimality of the transport plan. Given $(x_i, y_i + a)_{i = 1}^n \subset \text{supp}(\gamma_a)$ so that $(x_i, y_i)_{i = 1}^n \subset \text{supp}(\gamma)$, for any permutation of the indexes $\sigma$, it holds that
	\begin{align*}
		\sum_{i = 1}^n |x_i - (y_{\sigma(i)}+a)|^2 
		&= 
		\sum_{i = 1}^n |x_i - y_{\sigma(i)}|^2 
		-2
		\sum_{i = 1}^n \inner{x_i - y_{\sigma(i)},a} + |a^2|\\
		&\ge 
		\sum_{i = 1}^n |x_i - y_{i}|^2
		-2
		\sum_{i = 1}^n \inner{x_i - y_{i},a} + |a^2|.
	\end{align*}
	Where we have used the fact that supp$\gamma$ is itself c-CM and the linearity of the inner product. We conclude that $\gamma_a$ is optimal for its marginals $\rho_0$ and $\nu_{\star, a}$. 
	
	It then holds that
	\begin{align*}
		W^2(\rho_0,\nu_{\star,a})
		&= 
		\int \abs{x-(y+a)}^2\diff \gamma\\ 
		&=
		\int \abs{x-y}^2\diff \gamma 
		-2
		\dotp{\int (x-y)\diff \gamma}{a} + \abs{a}^2\\
		&=
		W^2(\rho_0,\nu_{\Sigma}) 
		-2\dotp{m_{\rho_0} - m_{\nu}}{a} + \abs{a}^2.
	\end{align*}
	Since  $F(\alpha,\nu_{\star+a})\geq F(\alpha,\nu_{\star})$, the sum of the last two terms must be nonnegative for all $a \in \RR^d$, which implies that $m_{\rho_0} = m_{\nu}$ by taking $a = m_{\rho_0} - m_{\nu}$, hence the result is proved.	
\end{proof}

\begin{proof}[Proof of Proposition \ref{proposition.critical_lambda_upper_bound_tighter}]	
Up to a change of origin, it is not restrictive to assume that $m_{\rho_0}\eqdef \int x\diff\rho_0 =0$.

Let $\Sigma$ be an admissible set with $\H^1(\Sigma)>0$, $\nu \in \mathcal{P}_2(\Sigma)$ a candidate measure. Since centering the measure $\nu$ decreases the energy, we assume without loss of generality that $\int_\Sigma y \diff \nu =0$.

Let $\gamma$ be an optimal transport plan between $\rho_0$ and $\nu$. It holds that
\begin{align*}
	\mathcal{E}(\nu) 
	&= 
	\int_{\RR^d \times \Sigma} 
	\abs{x-y}^2\dd \gamma + \Lambda \mathcal{L}(\nu) \\
	&\geq 
	\int_{\RR^d}\abs{x}^2\diff \rho_0 +
	\int_\Sigma |y|^2\dd \nu
	-2\int_{\RR^d \times \Sigma} \dotp{x}{y} \diff \gamma 
	+ \Lambda \H^1(\Sigma)\\
	&\geq \mathcal{E}(\delta_0) 
	-2\int_{\RR^d \times \Sigma} \dotp{x}{y} \diff \gamma
	+\Lambda \H^1(\Sigma).
\end{align*}
Let $r = \min \enscond{r'\geq 0}{\Sigma \subseteq B(0,r')}$. By Lemma~\ref{lemma.longueur}, we know that $r\leq d_0<+\infty$, and since $\H^1(\Sigma)>0$, we have $r>0$.
By minimality of $r$ and compactness of $\Sigma$, there exists some point $y_0 \in \Sigma$ such that $\abs{y_0}=r$. 

We claim that there exists some $\tilde{y} \in  \Sigma$ such that $\dotp{y_0}{\tilde y}\leq 0$. Indeed, since $0 = \int y \diff \nu$,  the point $0$ is in the closed convex hull of $\Sigma$, and if such a point $\tilde{y}$ did not exist, $0$ would be separated from $\Sigma$ by a hyperplane, a contradiction. As a result, $\abs{y_0-\tilde{y}}^2= \abs{y_0}^2 + \abs{\tilde{y}}^2 -2\dotp{y_0}{\tilde{y}}\geq r^2$, so that $\H^1(\Sigma)\geq \mathrm{diam}(\Sigma)\geq r$.

Now, we may resume the estimation
\begin{align*}
	\mathcal{E}(\nu) 
	&\geq 
	\mathcal{E}(\delta_0) - \int \abs{y}\abs{x}\diff \gamma+\Lambda \H^1(\Sigma)\\
			     &\geq \mathcal{E}(\delta_0) + r\left(\Lambda -2\int \abs{x}\diff \rho_0 \right).
\end{align*}
Therefore, as soon as $\Lambda -2\int \abs{x}\diff \rho_0>0$, the second term is strictly positive, and the Dirac mass has a lower energy.
\end{proof}
}\fi

\subsection{The example of an input with two Dirac masses}
	In this subsection we consider the case $p=2$. Let $x_{-1}=(-1,0,\ldots, 0)$, $x_{1}=(1,0,\ldots,0) \in \RR^d$, and let $\rho_0=\frac{1}{2}\left(\delta_{x_{-1}}+\delta_{x_{1}}\right)$. By Lemma~\ref{lem.convexhull}, we know that the solutions to \eqref{problem.shape_optimization_relaxed} are supported on line segments which are contained in $\ci{x_{-1}}{x_1}$. We may thus reduce the problem to the one-dimensional setting, with $x_{-1}=-1$, $x_{1}=1$. The solution to that problem is given by the following proposition.

	\begin{proposition}
		For $p=2$ and $\rho_0=\frac{1}{2}\left(\delta_{-1}+\delta_{1}\right)$, the unique solution to \eqref{problem.shape_optimization_relaxed} is given by
	
		\begin{align}
		\nu_\Lambda =	\begin{cases}
			\sqrt{\frac{3\Lambda}{2}}\H^1\mres\ci{-1}{1} + \left(\frac{1}{2}-\sqrt{\frac{3\Lambda}{2}}\right)(\delta_{-1}+\delta_1) & \mbox{if $0<\Lambda<\frac{1}{6}$},\\
			\frac{1}{3(1-2\Lambda)}\H^1\mres\ci{-\frac{3}{2}(1-2\Lambda)}{\frac{3}{2}(1-2\Lambda)} &\mbox{if $\frac{1}{6}\le \Lambda <\frac{1}{2}$}\\
			\delta_0 & \mbox{if $\Lambda\geq \frac{1}{2}$}.
			\end{cases}. \label{eq:}
		\end{align}
	\end{proposition}
	
	\begin{proof}
	We fix $\Lambda>0$ and denote $\nu$ a solution.
	Let $\alpha=\mathcal{L}(\nu)$.
	If $\alpha=0$, $\nu$ is a Dirac mass. If $\alpha>0$, 
	we know that the support of $\nu$ is a connected subset of
	$\cconv{\{-1,1\}}=[-1,1]$, so that $\supp\nu = [a,b]$ for $-1\le a < b\le 1$.
	In addition, letting $c\in [a,b]$ such that
	$\nu([a,c[)\le 1/2$ and $\nu([a,c])\ge 1/2$, one can check that
	if some mass is sent from $\{-1\}$ to $]c,b]$, then exchanging it
	with the same amount of mass sent from $\{+1\}$ to $[a,c[$ we reduce
	the Wasserstein distance. Hence one may assume that the mass
	coming from $\{-1\}$ is sent to a measure $\nu^-$ supported on $[a,c]$
	while the mass from $\{+1\}$ is sent to a measure $\nu^+$ supported
	on $[c,b]$, with $\nu^-+\nu^+=\nu$. Observing that $\nu\ge \frac{1}{\alpha}\H^1\mres [a,b]$
	(we are in the case $\alpha>0$),
	we introduce the non-negative excess measures:
	\[
	\nuexc^- = \nu^- - \frac{1}{\alpha}\H^1\mres [a,c],\quad
		\nuexc^+ = \nu^+ - \frac{1}{\alpha}\H^1\mres [c,b],
	\]
	and $\nuexc = \nuexc^-+\nuexc^+$. Once more, we see that the Wasserstein
	distance is reduced if all the mass sent from $\{-1\}$ to $\nuexc^-$
	is sent to the point $\{a\}$, closest to $\{-1\}$. Hence, we may assume
	that $\nuexc^-= x\delta_a$, for $x\ge 0$, and similarly,
	$\nuexc^+ = y\delta_b$, for $y\ge 0$. Eventually, we easily see that
	if $a>-1$ and $x>0$, then we can extend the segment $[a,b]$ towards $\{-1\}$,
	adding a small piece $[a-\delta,\delta]$ for $\delta\le
	\min \{\alpha x, a+1\}$, send a fraction $\delta/\alpha$ of
	the measure $x\delta_{a}$ rather to $\frac{1}{\alpha}\H^1\mres [a-\delta,a]$,
	and reduce again the Wasserstein distance without changing $\mathcal{L}(\nu)$.
	We deduce that $x=0$ if $a>-1$, similarly $y=0$ if $b<1$.
	
	\newcommand{\noop}[1]{}
	\noop{
		Since the solutions are supported on a line segment in $\ci{-1}{1}$, they are of the form $\nu=\delta_{a}$ or 
		%\begin{align*}
		$	\nu = \frac{1}{\alpha}\H^1\mres\ci{a}{b}+\nuexc$,
		%\end{align*}
		with $\alpha=\L(\nu)$ and $\supp \nuexc \subset \ci{a}{b}\subset \ci{-1}{1}$. 

	Since solutions are supported on a line segment in $\ci{-1}{1}$, we use the anzatz and assume them to be of the form
	\[
		\nu = \frac{1}{\alpha}\H^1\mres\ci{a}{b}
		\text{ or }
		\nu = \frac{1}{\alpha}\H^1\mres\ci{-1}{1} + c\delta_{-1} +d\delta_1. 
	\]
	Indeed if the $[a,b]$ does not coincide with $[-1,1]$ and there is any mass left after we form the uniform measure over the segment $[a,b]$, we enlarge a bit the segment. If $a$ or $b$ coincide with $-1,1$, we can just leave any residual mass concentrated at the Dirac delta with no transportation cost, see for instance Lemma~\ref{lemma.minimal_distance2Sigma} below.}
	
	Recalling that for $p=2$, $\nu$ must have the same center of mass as $\rho_0$, we deduce that $\nu$ must be equal to 
		\begin{align*}
			\nu_{0,0} &\eqdef\delta_0,\\
			\mbox{or }\quad	
			\nu_{b,2b}&\eqdef \frac{1}{2b}\H^1\mres\ci{-b}{b}\quad\mbox{for some $b \in \oi{0}{1}$}\\
			\mbox{or}\quad \nu_{1,\alpha} &= \frac{1}{\alpha}\H^1\mres\ci{-1}{1} + \left(\frac{1}{2}-\frac{1}{\alpha}\right)\left(\delta_{-1}+\delta_{1}\right)\quad \mbox{for some $\alpha\geq 2$.}
		\end{align*}
		Let $\E(\nu)= \Lambda \L(\nu)+W_2^2(\rho_0,\nu)$ denote the energy to minimize. We have $\E(\nu_{0,0})=1 = \lim_{b \to 0^+} \E(\nu_{b,2b})$, and
	\begin{align*}
		\E(\nu_{b,2b})&= 2\Lambda b + 2 \int_0^b (1-x)^2\frac{\dd x}{2b} = \frac{b^2}{3}+ (2\Lambda-1)b+1\\
		\mbox{with}\quad \frac{\dd}{\dd b}\E(\nu_{b,2b}) &= \frac{2b}{3}+2\Lambda-1,\\
		\E(\nu_{1,\alpha})&= \Lambda \alpha + 2\int_0^1(1-x)^2\frac{\dd x}{\alpha}+0 = \Lambda \alpha +\frac{2}{3\alpha},\\
		\mbox{with}\quad \frac{\dd}{\dd \alpha}\E(\nu_{1,\alpha}) &= \Lambda - \frac{2}{3\alpha^2}.
	\end{align*}

	For $0<\Lambda<\frac{1}{6}$, we check that $\nu_{1,\alpha^*}$, for $\alpha^*\eqdef \sqrt{\frac{2}{3\Lambda}}$, is the unique solution.
	
	For $\frac{1}{6}\le \Lambda < \frac{1}{2}$, we get that $\nu_{b^*,2b^*}$ is the unique solution, with $b^*\eqdef \frac{3}{2}(1-2\Lambda)$.
	
	For $\Lambda \geq \frac{1}{2}$, the functions $\alpha\mapsto \E(\nu_{1,\alpha})$ and $b\mapsto \E(\nu_{b,2b})$ are strictly increasing on $[2,+\infty[$ and $]0,1]$ respectively. Therefore $\nu_{0,0}$ is the unique solution to \eqref{problem.shape_optimization_relaxed}.
\end{proof}

\section{Solutions are rectifiable measures}\label{section.solutions_are_absolutely_continuous}
%In this section we propose a general argument which shows that whenever
Our goal here is to show that whenever $\rho_0 \ll \H^1$, any solution $\nu$ is a rectifiable measure
of the form %, {\em i.e.} it can be written in the form 
\[
    \nu = \theta \H^1\mres \Sigma, \ \text{ for $\theta \in L^1(\Sigma; \H^1)$}.
\]
To this end, we introduce the excess measure $\nuexc$ as the positive measure given by the mass of $\nu$ that exceeds the density constraints. We first show that this measure solves a family of localized problems. This is used to prove the absolute continuity w.r.t.~$\H^1\mres\Sigma$, that is, point~(1) of Theorem~\ref{theorem.the_big_one}.

\subsection{The excess measure}
Let $\nu$ be a minimizer of \eqref{problem.shape_optimization_relaxed} with support
$\Sigma$ not reduced to a singleton. From the definition of the length functional we have:
\[
	\mathcal{L}(\nu) < \infty \text{ if and only if there is $\alpha \ge 0$ such that } \alpha\nu \ge \H^1\mres \Sigma.  	
\]
Setting $\alpha \eqdef \mathcal{L}(\nu)>0$, we define the following decomposition
\begin{equation}\label{nuexceed}
	\nu = \nuH + \nuexc, 
	\text{ where $\nuH \eqdef \alpha^{-1}\H^1\mres \Sigma$ and $\nuexc \eqdef \nu - \nuH$}.
\end{equation}
The part $\nuH$ is the measure which saturates the density constraint, and the support of the \textit{excess measure} $\nuexc$ is where the constraint is inactive.  

In the sequel, we fix an optimal transport plan $\gamma$, for the problem defining $W_p^p(\rho_0, \nu)$, and we define an analogous (non-unique) decomposition of $\gamma$ and $\rho_0$ by disintegrating $\gamma$ w.r.t.~the second marginal. From  the disintegration theorem \cite[Theorem 2.28]{ambrosio2000functions}, there exists a $\nu$-measurable family $\{\gamma_{y}\}_{y\in \RR^d} \subset \P(\RR^d)$, such that $\gamma= \gamma_y\otimes\nu$, that is
\begin{equation}\label{disintegration_formula}
  \int_{\RR^d \times \Sigma}
  \psi(x, y)\dd \gamma(x,y)
  =
  \int_\Sigma 
  \left(
    \int_{\RR^d} \psi(x,y)\dd \gamma_y(x)
  \right)\dd\nu(y), \text{ for all $\psi \in L^1(\gamma)$}
  .
\end{equation}

We define a decomposition $\gamma = \gammaH + \gammaexc$ as
\begin{equation}\label{gammaexceed}
	\gammaH(A\times B)\eqdef 
	\int_{\Sigma \cap B} 
	\gamma_y(A)
	\dd\nuH(y), 
	\quad
	\gammaexc(A\times B)\eqdef 
	\int_{\Sigma \cap B} 
	\gamma_y(A)
	\dd\nuexc(y).
\end{equation}
The decomposition $\rho_0 = \rhoH + \rhoexc$ can be defined as the marginals of $\gammaH$ and $\gammaexc$
\begin{equation}\label{rhoexceed}
\rhoH \eqdef {(\pi_{0})}_\sharp \gammaH, 
	\quad
\rhoexc \eqdef {(\pi_{0})}_\sharp \gammaexc.
\end{equation}

This way $\gammaH \in \Pi(\rhoH, \nuH)$, $\gammaexc \in \Pi(\rhoexc, \nuexc)$ and they are optimal transport plans between their respective marginals. Indeed if we find a better transport plan for either problem we can construct a better plan for the original problem, contradicting the minimality of $\gamma$. We therefore also have a decomposition between the Wasserstein distances
\begin{equation}\label{rhoexc_decomposition_wasserstein}
    W_p^p\left(\rho_0, \nu\right) 
    =
    W_p^p\left(\rhoH, \nuH\right)
    +
    W_p^p\left(\rhoexc, \nuexc\right).
\end{equation}

Let us point out that, although the decomposition of $\nu$ is natural, there are many ways to decompose $\gamma$ and $\rho_0$, for instance by choosing another disintegration family. In the sequel we show that for any such decomposition the excess must be concentrated on the graph of the operator given by the (multivalued) projection onto $\Sigma$
\begin{equation}\label{projection_operator}
	\Pi_\Sigma(x) \eqdef \argmin_{y \in \Sigma} |x - y|^2.
\end{equation}

%\begin{remark}\label{rem.projection}
Note that $\Pi_\Sigma$ is a multivalued operator which is included in the subgradient of the convex conjugate of the function:
$y\mapsto |y|^2/2$ if $y\in \Sigma$ and $+\infty$ else. %\vincent{Il me semble qu'on avait conclu que ce n'était pas vrai car cela impliquerait que $\Pi_\Sigma(x)$ est convexe, non?}
%\end{remark}
\begin{lemma}\label{lemma.minimal_distance2Sigma}
	Let $\nu$ be a minimizer of \eqref{problem.shape_optimization_relaxed} and $\gamma$ an optimal transport plan from $\rho_0$ to $\nu$. Then, for any decomposition
    $
        \gamma = 
        \gammaH+ \gammaexc, 
        \text{ s.t. }
        {(\pi_{1})}_\sharp \gammaexc = \nuexc,    
    $
    it holds that
	\begin{align}
		\supp \gammaexc \subset \mathrm{graph}(\Pi_\Sigma).\label{eqgammaexcproj}
	\end{align}
    In addition, for any $\pi_\Sigma$ measurable selection of $x\mapsto \Pi_\Sigma(x)$, the measure
	\[
		\nuH + {(\pi_{\Sigma})}_\sharp\rhoexc 
  	\]
	is optimal for \eqref{problem.shape_optimization_relaxed}.
\end{lemma}
\begin{proof}
	Consider the problem 
    \begin{align}
        \label{problem.coupling_optimization_relaxed}
        \tag{$\overline{Q}_\Lambda$}
        \inf_{
            \substack{
                \gamma \in \mathcal{P}_p(\RR^d\times \RR^d)\\
                {(\pi_{0})}_\sharp\gamma = \rho_0,
            }
          } 
          \int_{{\RR^d\times \RR^d}}\abs{x-y}^p\dd \gamma(x,y) + \Lambda \mathcal{L}({(\pi_{1})}_\sharp\gamma), 
    \end{align}
    which is a reformulation of~\eqref{problem.shape_optimization_relaxed} in
    terms of the transport plan $\gamma$ from $\rho_0$ to $\nu$.
    
    Now, let $(\gammaH,\gammaexc)$ be any suitable decomposition of $\gamma$ and let $\pi_\Sigma$ be a measurable selection of $\Pi_\Sigma$. We set $\rhoexc\eqdef {(\pi_{0})}_\sharp\gammaexc$ and define $\tilde{\gamma} = \gammaH + (\text{id},\pi_\Sigma)_{\sharp}\rhoexc$. Then, since still ${\pi_1}_\sharp \tilde{\gamma}\ge \nu_{\H^1}$,
    it holds that $\mathcal{L}({\pi_{1}}_\sharp\tilde{\gamma})\leq \mathcal{L}(\nu)$ and 
    % \begin{align*}
    \begin{multline*}
      \int_{{\RR^d\times \RR^d}}
      \abs{x-y}^p\dd \tilde{\gamma}
      % &
      = \int_{{\RR^d\times \RR^d}}
      \abs{x-y}^p\dd \gammaH
      + 
      \int_{\RR^d}
      \abs{x-\pi_{\Sigma}(x)}^p\dd\rhoexc \\
      % &
      \leq \int_{{\RR^d\times \RR^d}}\abs{x-y}^p\dd \gammaH+ \int_{\RR^d\times \Sigma}\abs{x-y}^p\dd\gammaexc = \int_{{\RR^d\times \RR^d}}\abs{x-y}^p\dd {\gamma}
    \end{multline*}
    % \end{align*}
    Since $\gamma$ is a minimizer of~\eqref{problem.coupling_optimization_relaxed}, we must have an equality, in particular it holds that
    \begin{align*}
      \int_{\RR^d\times \RR^d}\left( \abs{x-y}^p-\abs{x-\pi_{\Sigma}(x)}^p\right)\dd\gammaexc = 0.
    \end{align*}
    Since $\gamma$-a.e.~$(x,y)$ is in $\RR^d\times \Sigma$, the integrand is nonnegative and must vanish $\gammaexc$-a.e. Hence $(x,y) \in \mathrm{Graph}(\Pi_\Sigma)$ for $\gammaexc$-a.e. $(x,y)$ and \eqref{eqgammaexcproj} follows since $\mathrm{Graph}(\Pi_\Sigma)$ is closed.
    As a consequence, the measure $\nuH + {\pi_{\Sigma}}_\sharp\rhoexc$ reaches the minimimum for \eqref{problem.shape_optimization_relaxed} and is optimal.
\end{proof}

\subsection{Solutions are absolutely continuous}
Now we prove %pass to the goal of proving
that the solutions to the relaxed problem \eqref{problem.shape_optimization_relaxed} are absolutely continuous w.r.t.~$\H^1\mres\Sigma$. The proof is based on the construction of a localized variational problem. %, which is discussed below and proved in Appendix \ref{appendix.localized_variational_problem}. 

\begin{lemma}\label{lemma.auxiliary_problem}
    Let $\nu$ be an optimal solution for the relaxed problem~\eqref{problem.shape_optimization_relaxed} and set $\alpha = \mathcal{L}(\nu)$. Let $\mathcal{S} = \mathcal{S}_0\times \mathcal{S}_1 \subset \RR^d\times \RR^d$ be a Borel set and define the transport plan
    \[
        \gamma_{\mathcal{S}} \eqdef \gammaexc \mres \mathcal{S}_0\times \mathcal{S}_1
    \]
    along with its marginals
    \[
        \rho_{\mathcal{S}} \eqdef {\pi_{0}}_\sharp\gamma_\mathcal{S} 
        \le \rhoexc \mres \mathcal{S}_0, %% "=" EST FAUX
        \quad \nu_{\mathcal{S}} \eqdef {\pi_{1}}_\sharp\gamma_{\mathcal{S}}.
    \]
    
    Then the measure $\nu_{\mathcal{S}}$ solves the following variational problem
    \begin{equation}%\label{auxiliary_problem}
        \inf
        \left\{
            W_p^p\left(
                \rho_{\mathcal{S}}, \nu'
            \right) : 
            \begin{array}{c}
                \text{there is $\Gamma$ such that}\\ 
                \nu' \in \mathcal{M}_+(\Sigma \cup \Gamma),\\ 
                \nu' \ge \alpha^{-1}\mathcal{H}^1\mres \left(\Gamma\setminus \Sigma\right), \\ 
                \Sigma \cup \Gamma \in \mathcal{A}, \
                \nu'(\RR^d) = \nu_{\mathcal{S}}(\RR^d)
            \end{array} 
        \right\}
    \end{equation}
    
    More generally, let $\left(\sigma_{\mathcal{S},t}\right)_{t \in [0,1]}$ be the constant speed geodesic between $\rho_{\mathcal{S}}$ and $\nu_{\mathcal{S}}$ defined through $\sigma_{\mathcal{S},t} \eqdef {\pi_{(1-t)}}_\sharp\gamma_{\mathcal{S}}$, where  $\pi_t(x,y) \eqdef (1-t)x + ty$. Then for any $t \in [0,1]$, the measure $\nu_{\mathcal{S}}$ minimizes the variational problem
    \begin{equation}\label{auxiliary_problem_sigma}
        \inf 
        \left\{
            W_p^p\left(
                \sigma_{\mathcal{S}, t}, \nu'
            \right) : 
            \begin{array}{c}
                \text{there is $\Gamma$ such that }\\ 
                \nu' \in \mathcal{M}_+(\Sigma \cup \Gamma),\\ 
                \nu' \ge \alpha^{-1}\mathcal{H}^1\mres \left(\Gamma\setminus \Sigma\right), \\ 
                \Sigma \cup \Gamma \in \mathcal{A}, \
                \nu'(\RR^d) = \nu_{\mathcal{S}}(\RR^d)
            \end{array} 
        \right\}. 
    \end{equation}
\end{lemma}
\begin{proof} See Appendix~\ref{appendix.localized_variational_problem}.
\end{proof}
%The generality of Lemma \ref{lemma.auxiliary_problem} allows us to  craft a specific set $\mathcal{S}$ in order to extract optimality conditions, whereas
We now craft a specific set $\mathcal{S}$ to apply the lemma.
Given $\delta>0$, we define the %cylinder-like
set
\begin{align} \label{set_D_delta}
    D_{\delta} &\eqdef
    \left\{
        x \in \supp\rhoexc : 
        \begin{array}{c}
            %\dist(x, \Sigma\cap B_r(y_0)) < \dist(x, \Sigma\setminus B_r(y_0))\\ 
            \delta \le \dist(x, \Sigma) \le \delta^{-1}
        \end{array}
    \right\}, 
\end{align}
And for a fixed point $y_0 \in \Sigma$, and $\delta,r>0$ consider the new transport plan
\begin{equation}\label{gamma_delta}
    \gamma_{\delta, r} \eqdef \gammaexc \mres D_{\delta}\times B_r(y_0)
\end{equation}
along with its marginals 
\begin{equation}\label{nu_delta}
    \rho_{\delta, r} \eqdef {\pi_{0}}_\sharp\gamma_{\delta,r} \le \rhoexc \mres D_{\delta}, 
    \quad 
    \nu_{\delta, r} \eqdef {\pi_{1}}_\sharp\gamma_{\delta, r}.
\end{equation}
From Lemma~\ref{lemma.auxiliary_problem} it holds that 
\begin{equation}\label{auxiliary_problem}
    \nu_{\delta, r}    
    \in
    \argmin
    \left\{
        W_p^p\left(
            \rho_{\delta, r}, \nu'
        \right) : 
        \begin{array}{c}
        \text{there is $\Gamma$ such that }\\ 
            \nu' \in \mathcal{M}_+(\Sigma \cup \Gamma),\\ 
            \nu' \ge \alpha^{-1}\mathcal{H}^1\mres (\Gamma\setminus\Sigma), \\ 
            \Sigma \cup \Gamma \in \mathcal{A}, \
            \nu'(\RR^d) = \nu_{\delta, r}(\RR^d)
        \end{array} 
    \right\}. 
\end{equation}
We also introduce 
\begin{equation}\label{measure_nu_delta}
    \gamma_\delta\eqdef\gammaexc\mres D_{\delta}\times\Sigma
    \text{ and }
    \nu_\delta \eqdef {\pi_1}_\sharp \gamma_\delta,
\end{equation}
so that by definition, $\nu_{\delta, r} = \nu_\delta \mres B_r(y_0)$ and $\nuexc$ can be further decomposed as
$
    \nuexc = \nu_\delta + {\pi_{1}}_\sharp\left(\gammaexc \mres D_\delta^c \times \RR^d\right). 
$
As $D_\delta$ is a nested sequence of sets, $(\nu_\delta)_{\delta > 0}$ is a monotone sequence and taking the limit as $\delta \to 0$ we have  
\begin{equation}\label{nuexc.further_decomposition}
    \nuexc = \sup_{\delta > 0} \nu_\delta + \rhoexc\mres \Sigma,
\end{equation}
the second limit being $\rhoexc\mres \Sigma$ because
of Lemma~\ref{lemma.minimal_distance2Sigma} and
since the only projection of a point in $\Sigma$ is itself. 

In the next Theorem~\ref{theorem.solutions_are_absolutely_continuous} we show that the measures $\nu_\delta$ have a uniformly % $L^\infty$ 
bounded density w.r.t.~$\H^1$. So when $\rho_0$ is absolutely continuous w.r.t.~$\H^1$,~\eqref{nuexc.further_decomposition}
shows that any optimal $\nu \ll \H^1$.
The argument consists in crafting a competitor for the 
localized problem \eqref{auxiliary_problem}, built as a measure supported on a curve with controlled length, defined over small sphere, centered at an arbitrary point of the support of $\nu_\delta$. Letting the radius of this sphere go to zero, and comparing the energy of this competitor and the optimal measure, gives a uniform bound on the density.
This strategy is illustrated in Figure~\ref{figure.L_infty_bounds}.
% , this will be done with the following Lemma.

\begin{lemma}\label{lemma.curve_on_B2}
    Let $B_2$ be the ball on $\RR^d$ centered at the origin. There exists a connected set $\Gamma_d \subset \partial B_2$ with $\H^1(\Gamma_d) < +\infty$ and such that 
    \[
        \dist(x, \Gamma_d) \le |x - y| - \frac{1}{2}
    \]
    for any $x \not\in B_2$ and for all $y \in B_1$.
\end{lemma}
\begin{proof}
    We start by covering the sphere $\partial B_2$ with finitely many balls $\left(B_{1/2}(x_i)\right)_{i = 1}^{N_d}$, each having radius $1/2$. The number of balls $N_d$ being dependent on the dimension. In the sequel we define $\Gamma_d$ with geodesics on $\partial B_2$ connecting the centers $\left(x_i\right)_{i = 1}^{N_d}$. 

    As we have finitely many points, we will also have finitely many curves and hence $\H^1(\Gamma_d)$ must be a dimensional constant. We can even choose the connected set $\Gamma_d$ with minimal length,
    which is a solution to Steiner's problem on the spheres and has a tree structure,
    so that we can bound $\H^1(\Gamma_d) \le (N_d - 1)D_d$, where $D_d$ is the diameter of $\partial B_2$ in its Riemannian metric. 

    To prove the desired property, take $x \not\in B_2$ and $y \in B_1$. Let $\{\hat y\} = [x,y]\cap \partial B_2$. Then $\hat y \in B_{1/2}(x_i)$  for some $x_i$
    while $|x-\hat y|=|x-y|-|\hat y-y|\le |x-y|-1$, and it follows:
    \[
        \dist(x, \Gamma_d) 
        \le |x - x_i| 
        \le |x - \hat y| + |\hat y - x_i|
        \le |x - y| - \frac{1}{2},
    \]
    which gives the desired construction. 
\end{proof}

\begin{figure}[t]
    \centering
    \input{L_infty_tikz.tex}
    \caption{Scheme of the proof of Thm.~\ref{theorem.solutions_are_absolutely_continuous}. For the new competitor, created with the curve $\Gamma$ from Lemma~\ref{lemma.curve_on_B2}, we pay a little more in the transportation cost to generate $\alpha^{-1}\H^1\mres \Gamma_r$, but pay much less by projecting the remaining mass onto it.}\label{figure.L_infty_bounds}
\end{figure}

\begin{theorem}\label{theorem.solutions_are_absolutely_continuous}
    Given $\rho_0 \in \mathcal{P}_p(\RR^d)$, let $\nu$ be a solution to \eqref{problem.shape_optimization_relaxed}. 
    Then it holds that the measures $\left(\nu_\delta\right)_{\delta > 0}$ are of the form 
    \[
        \nu_\delta = \theta_\delta \H^1\mres \Sigma,   
        \text{ 
        with $ 
        \norm{\theta_\delta}_{L^\infty(\Sigma, \H^1)}
        \le 
        \frac{7}{2}\frac{C_d}{\mathcal{L}(\nu)}$,}
    \]
    for $C_d = 2 + \H^1(\Gamma_d)$, $\Gamma_d$ being the set from Lemma \ref{lemma.curve_on_B2}. 
    
    Therefore, 
    if $\rho_0 \ll \H^1$ or has a $L^{\infty}$ density w.r.t. $\H^1$, so does $\nu$, in particular it is a rectifiable measure. 
\end{theorem}
\begin{proof}
  
	For $y_0 \in \Sigma$, let us define the one-dimensional upper density~\cite[Def.~2.55]{ambrosio2000functions}
\begin{align*}
	\theta_\delta(y_0)\eqdef\limsup_{r \to 0} \frac{\nu_{\delta}(B_r)}{2r}.
\end{align*}
We will show that $\theta_{\delta}(y_0)\leq \frac{7}{2}\frac{C_d}{\L(\nu)}$, so that thanks to \cite[Thm. 2.56]{ambrosio2000functions}, $\nu_\delta \ll \H^1\mres\Sigma$.
Since $\Sigma$ is 1-rectifiable, it follows that for $\H^1$-a.e.~$y_0 \in \Gamma$, $\theta_\delta(y_0)$ is the  Radon-Besicovitch %Nikod\'ym 
derivative of $\nu_\delta$ w.r.t.~$\H^1\mres\Sigma$, and the claim of the theorem follows.

	From the optimality of $\nu$, the measure $\nu_{\delta,r}$ solves problem \eqref{auxiliary_problem}. In order to build a competitor we consider the set $\Gamma_d$ from Lemma \ref{lemma.curve_on_B2}, choose some point $\bar y \in \Gamma_d$ and define
    \[
        \Gamma_r 
        \eqdef
        [y_0, y_0 + r\bar y]\cup 
        \left(
            y_0 + r\Gamma_d
        \right), 
    \] 
    which is contained in $\overline{B}_{2r}(y_0)$. Notice that $\Sigma \cup \Gamma_{r}$ is always a compact, connected and $1$-rectifiable set and one has 
    \[
        \H^1(\Gamma_{r}) 
        = 
        C_dr,  
    \]
    where $C_d = 1 + \H^1(\Gamma_d)$ is a constant depending only on the dimension. 

    In the sequel, setting $\alpha = \mathcal{L}(\nu)$ we define the following parameter 
    \[
        m_{r} \eqdef 
\frac{\H^1(\Gamma_r)}{\alpha \nu_\delta(B_r)}.\] 
       % \xrightarrow[r \to 0]{} 
       % \frac{C_d}{2\alpha\theta_\delta(y_0)}=: m_0, 
    %\]
   Suppose that $C_d/\alpha < 2\theta_{\delta}(y_0)$. Then,
\begin{align*}
	1> m_0\eqdef\frac{C_d}{2\alpha\theta_\delta(y_0)}= \liminf_{r \to 0} m_r.
\end{align*}
Now, we consider a subsequence $(r_{k})_{k\in \NN}\searrow 0$ such that $\lim_{k \to \infty} m_{r_k}=\liminf_{r \to 0} m_r$. In particular,  $m_{r_k} \in (0,1)$ for $r_k$ sufficiently small. % otherwise we can bound the density $\theta_\delta(y_0)$ by $C_d/2$ and there is nothing to prove. 
For simplicity, in  the sequel, we drop the subscript $k$, yet we consider only $r \in \{r_{k}\}_{k\in \NN}$.

    Let $\gamma_{\Gamma_{r}}$ be an optimal transport plan between $m_r\rho_{\delta, r}$ and $\alpha^{-1}\H^1\mres \Gamma_{r}$ for the Wasserstein-$p$ distance and define the new plan 
    \[
        \tilde \gamma_{\delta, r}
        \eqdef 
        \gamma_{\Gamma_{r}}
        + 
        (1-m_{r}) 
        \left(
            \text{id}, 
            \pi_{\Gamma_{r}}
        \right)_\sharp
        \rho_{\delta, r},
        \text{ and }
        \tilde\nu_{\delta, r}
        \eqdef {\pi_{1}}_\sharp \tilde \gamma_{\delta, r},
    \]
    where $\pi_{\Gamma_{r}}$ is a measurable selection of the projection operator onto $\Gamma_{r}$. This construction is illustrated in Figure~\ref{figure.L_infty_bounds}.Therefore $\tilde\nu_{\delta, r}$ is admissible for \eqref{auxiliary_problem} and we have the following estimate
    \begin{align*}
        W_p^p(\rho_{\delta, r}, \tilde\nu_{\delta, r})
        \le 
        \int_{\RR^d\times \RR^d} |x - y|^p\dd \gamma_{\Gamma_{r}}
        + 
        (1 - m_{r})\int_{\RR^d} \dist(x, \Gamma_r)^p\dd \rho_{\delta,r}.
    \end{align*}

    We will estimate each term of the previous inequality separately. For the first one, notice that as $\supp \gamma_{\Gamma_{r}} \subset \Pi_{\Sigma}^{-1}(B_r(y_0)) \times \overline{B}_{2r}(y_0)$, it holds that 
    \[
        |x - y| \le \dist(x,\Sigma) + 3r, \quad \text{ for $\gamma_{\Gamma_{r}}$-a.e. $(x,y)$}.     
    \]
    For the second term, as the projection of $x$ onto $\Sigma$ is inside $B_r(y_0)$, if follows from Lemma \ref{lemma.curve_on_B2} that 
    \[
        \dist(x, \Gamma_r) \le \dist(x, \Sigma) - \frac{r}{2}, \text{ for } \dist(x, \Sigma) > 2r. 
    \]
    Therefore, for a fixed $\delta$ and taking $2r < \delta$, the Wasserstein distance is bounded by 
    \begin{multline*}
        W_p^p(\rho_{\delta, r}, \tilde\nu_{\delta, r})
        \le 
        m_{r}
        \int_{\RR^d}  
        \left(\dist(x,\Sigma) + 3r\right)^p\dd \rho_{\delta,r} 
        +
        (1 - m_r)
        \int_{\RR^d}  
        \left(\dist(x, \Sigma) - r/2\right)^p\dd \rho_{\delta,r}.
    \end{multline*}

    Notice that 
    $\disp
        W_p^p(\rho_{\delta, r},\nu_{\delta, r}) 
        = 
        \int_{\RR^d} \dist(x,\Sigma)^p\dd \rho_{\delta,r},
    $
    so in order to compare the Wasserstein distances we use the following inequalities 
    \begin{align*}
        \left(\dist(x,\Sigma) + 3r\right)^p
        &\le 
        \dist(x,\Sigma)^p 
        + 
        3rp
        \left(\dist(x,\Sigma)+3r\right)^{p-1} \\ 
        \left(\dist(x,\Sigma) -\frac{r}{2}\right)^p
        &\le 
        \dist(x,\Sigma)^p 
        -
        \frac{r}{2}p
        \left(\dist(x,\Sigma)-\frac{r}{2}\right)^{p-1}
    \end{align*}
    which follow from the convexity of $t\mapsto |t|^p$.
    Then, given $\varepsilon>0$, if $r\le \delta\varepsilon$ one deduces,
    for $\dist(x,\Sigma)\ge \delta$, that:
    \begin{align*}
        \left(\dist(x,\Sigma) + 3r\right)^p
        &\le 
        \dist(x,\Sigma)^p 
        + 
        3 r p(1+3\varepsilon)^{p-1}
        \dist(x,\Sigma)^{p-1} \\ 
        \left(\dist(x,\Sigma) -\frac{r}{2}\right)^p
        &\le 
        \dist(x,\Sigma)^p 
        -
        \frac{r}{2}p \left(1-\frac{\varepsilon}{2}\right)^{p-1}
        \dist(x,\Sigma)^{p-1}.
    \end{align*}

    Therefore it holds that 
    \begin{align*}
        W_p^p(\rho_{\delta, r}, \tilde\nu_{\delta, r})
        &\le
        W_p^p(\rho_{\delta, r},  \nu_{\delta, r})
        + 
        pr\Delta_{r,\varepsilon}
        \int_{\RR^d}
        \dist(x,\Sigma)^{p-1}\dd \rho_{\delta,r}\\ 
        \text{ for }
        \Delta_{r,\varepsilon} 
        &=
            3m_{r}
            \left(
                1 + 3\varepsilon
            \right)^{p-1}
            -
            \frac{1 - m_r}{2}
            \left(
                1 - \frac{\varepsilon}{2}
            \right)^{p-1}
    \end{align*}
    Hence from the optimality of $\nu_{\delta,r}$ we have $\Delta_{r,\varepsilon} \ge 0$, so that letting $r \to 0$ and then $\varepsilon\to 0$, it must hold that
    $3 m_0\ge (1-m_0)/2$, that is:
    \[
        \theta_\delta(y_0)
        \le 
        \frac{7}{2}\frac{C_d}{\alpha}. 
    \]
As a result, the family $\left(\nu_\delta\right)_{\delta > 0}$ has a uniform $L^\infty$ density bounds, and so does the limit measure  $\sup_{\delta > 0} \nu_\delta = \left(\sup_{\delta > 0}\theta_\delta\right)\H^1\mres \Sigma$. But as the exceeding measure can be decomposed as~\eqref{nuexc.further_decomposition} we deduce that
%    \[
%        \nuexc = \sup_{\delta > 0} \nu_\delta + \rhoexc\mres \Sigma,  
%    \]
    whenever the initial measure $\rho_0 \ll \H^1$ or has a $L^{\infty}$ density w.r.t. $\H^1$, so does the solution $\nu$. 
\end{proof}

\section{Existence of solutions to \eqref{problem.shape_optimization}}\label{section.blowup}

This section is dedicated to the proof of~Theorem \ref{theorem.the_big_one}, item~(2). Knowing 
that the excess measure is absolutely continuous (Theorem~\ref{theorem.solutions_are_absolutely_continuous}), we use
a blow up argument near a rectifiability point $y_0$ of $\Sigma$.
From~Lemma \ref{lemma.auxiliary_problem}, the blow-ups of $\nuexc$ minimize a family of functionals $(F_r)_{r > 0}$, which in turn $\Gamma$-converge to some functional $F$.
Since these blow-ups also converge (for $\H^1$-a.e.~$y_0$) to a uniform density on $T_{y_0}\Sigma$, this limit measure must also minimize the $\Gamma$-limit $F$.
Yet if it is not zero, we can build a better competitor (Lemma~\ref{lemma.better_competitor} below), giving a contradiction to the minimality of the uniform measure.
We deduce that $\nuexc$ vanishes.

\subsection{Blow-up and $\Gamma$-convergence}\label{sec.gamma_convergence}
In the sequel, we assume that $\rho_0 \ll \mathcal{H}^1$, so that from Theorem~\ref{theorem.solutions_are_absolutely_continuous} any minimizer $\nu$, as well as ${(\nu_\delta)}_{\delta > 0}$ (defined in~\eqref{measure_nu_delta}), are rectifiable measures and we can write 
\[
    \nu_\delta = \theta_\delta \mathcal{H}^1\mres \Sigma, 
    \text{ for }
    \theta_\delta \in L^1(\H^1\mres \Sigma). 
\]

Observe that $\nu_\delta$-a.e.~$y \in \Sigma$ is a rectifiability point, and we choose  $y_0 \in \Sigma$ such that:
\begin{equation}\label{y0fixed_for_blowup}
 %   y_0 \in \supp \nu_\delta \text{ such that } 
    T_{y_0}\Sigma \text{ exists }
    \quad\text{ and } \quad 
    y_0 \text{ is a Lebesgue point of }
    \theta_\delta.
\end{equation}

We then use Lemma~\ref{lemma.auxiliary_problem} with the choice $\mathcal{S}_0\times \mathcal{S}_1 = D_{\delta}\times B_r(y_0)$, and we focus on the variational problem~\eqref{auxiliary_problem_sigma}: we obtain the families of measures $\left(\nu_{\delta ,r}\right)_{r > 0}$ and $\left(\sigma_{\delta ,r}\right)_{r > 0}$ as $\nu_{\delta ,r} \eqdef \nu_\delta \mres B_r(y_0)$ and $\sigma_{\delta ,r} \eqdef {\pi_{(1-r)}}_\sharp\gamma_{\delta,r}$, where $\left(\sigma_{\delta,t}\right)_{t \in [0,1]}$ is a family of geodesic interpolations, as in Lemma~\ref{lemma.auxiliary_problem}, so that
\begin{equation}\label{auxiliary_problem_geodesic}
    \nu_{\delta ,r}    
    \in
    \argmin
    \left\{
        W_p^p\left(
            \sigma_{\delta ,r}, \nu'
        \right) : 
        \begin{array}{c}
            \text{there is $\Gamma$ such that }\\ 
            \nu' \in \mathcal{M}_+(\Sigma \cup \Gamma),\\ 
            \nu' \ge \alpha^{-1}\mathcal{H}^1\mres \left(\Gamma\setminus \Sigma\right), \\ 
            \Sigma \cup \Gamma \in \mathcal{A}, \
	    \nu'(\RR^d) = \nu_{\delta ,r}(\RR^d)
        \end{array} 
    \right\}.
\end{equation}
From Lemma~\ref{lemma.minimal_distance2Sigma} the optimal transport plan between $\nu_{\delta ,r}$ and $\sigma_{\delta ,r}$ is supported on $\text{graph}(\Pi_\Sigma)$. 

The sequence of measures $\nu_{\delta ,r}$ are essentially a localization of $\nu_\delta$ around $y_0$ so, by the blow-up Theorem \ref{theorem.rectifiability_tangentiability} (see also \cite[Theo. 2.83]{ambrosio2000functions}), it holds that
\begin{equation}\label{blowup_measures_nu}
    r^{-1}\Phi^{y_0,r}_{\sharp}\nu_{\delta ,r} \xrightharpoonup[r \to 0]{\star} 
    \theta_\delta(y_0)\mathcal{H}^1\mres [-\tau, \tau], 
    \text{ where $\RR\tau = T_{y_0}\Sigma$.}
\end{equation}
Up to a subsequence (not labelled) we also have: 
\begin{equation}\label{blowup_measures_sigma}
    r^{-1}\Phi^{y_0,r}_{\sharp}\sigma_{\delta ,r} \xrightharpoonup[r \to 0]{\star} 
    \bar \sigma_{\delta}
\end{equation}
for some measure $\bar \sigma_{\delta}$.
By construction  $\sigma_{\delta ,r}$ is supported on $\{r\delta^{-1}\ge\dist(\cdot,\Sigma)\ge r\delta\}$, so that $\supp \bar \sigma_{\delta} \subset \{x : \delta^{-1} \ge \dist(x,\RR\tau) \ge \delta \}$. 

In view of~\eqref{blowup_measures_nu} and~\eqref{blowup_measures_sigma}, we introduce the blow-ups of the measures $\nu_{\delta ,r}$ and $\sigma_{\delta , r}$,
\begin{equation}\label{blowup_measures_renormalized}
    \bar\nu_{\delta ,r}
    \eqdef
    \frac{1}{r}
    \Phi^{y_0,r}_{\sharp}
    \nu_{\delta ,r}, 
    \quad 
    \bar\sigma_{\delta ,r}
    \eqdef
    \frac{1}{r}\Phi^{y_0,r}_{\sharp}
    \sigma_{\delta , r},
    \text{ and the set }
    \Sigma_r \eqdef 
    \frac{\Sigma - y_0}{r}\cap \overline{B_1(0)}.
\end{equation}

In addition, we define a family of functionals ${\left(F_r\right)}_{r > 0}$ as
\begin{equation}\label{family_Fr}
    F_r(\nu') \eqdef 
        \begin{cases}
            W_p^p\left( 
                \disp 
                \bar \sigma_{\delta ,r},  
                \nu' 
            \right),&
                \begin{array}{c}
                    \text{there is $\Gamma \subset \overline{B_1(0)}$ such that} \\ 
			        \nu' \in \mathcal{M}_+\left(\Sigma_r \cup \Gamma\right), \
                    \nu' \ge \alpha^{-1}\H^1\mres 
                    \left(\Gamma\setminus \Sigma_r\right),\\
                    \left(\frac{\Sigma-y_0}{r}\right) \cup \Gamma \text{ closed and connected }, \\ 
                    \displaystyle
                    \nu'(\overline{B_1(0)}) = \frac{\nu_\delta(B_r(y_0))}{r}, 
                \end{array} \\ 
            & \\ 
            +\infty,& \text{ otherwise,}
        \end{cases}
\end{equation}
where $\alpha = \mathcal{L}(\nu)$. 
Observing that for any given measures $\mu', \nu'$ we have
\begin{equation}\label{blowup_Wasserstein_identity}
    W_p^p
    \left(
        \frac{1}{r}\Phi^{y_0, r}_\sharp \mu', 
        \frac{1}{r}\Phi^{y_0, r}_\sharp \nu'
    \right) 
    = 
    \frac{1}{r^{p+1}} 
    W_p^p\left( \mu',  \nu'\right). 
\end{equation}
and recalling \eqref{auxiliary_problem_geodesic}, we see that $\bar \nu_{\delta ,r} \in \argmin F_r$ for any $r>0$.

The natural candidate for the limit of this family is the following:
\begin{equation}\label{Gamma_limitF}
    F(\nu') \eqdef 
    \begin{cases}
        W_p^p\left( \bar \sigma_{\delta}, \nu' \right),&
                \begin{array}{c}
                    \text{there is $\Gamma \subset \overline{B_1(0)}$ such that}\\ 
			        \nu' \in \mathcal{M}_+\left([-\tau, \tau] \cup \Gamma\right), 
			        \ \nu' \ge \alpha^{-1}
                    \H^1\mres\left(
                        \Gamma\setminus[-\tau,\tau]
                    \right),\\
                    \RR\tau \cup \Gamma \text{ closed and connected},\\ \nu'(\overline{B_1(0)}) = 2\theta_\delta(y_0),
                \end{array} \\ 
        +\infty,& \text{ otherwise.}
    \end{cases}
\end{equation}
We prove in Theorem~\ref{theorem.Gamma_convergence} below that $F_r$ $\Gamma$-converges to $F$ as $r\to 0^+$. We refer to \cite{dal_maso_introduction_1993,braides_gamma-convergence_2002} and in particular to~\cite[Def. 1.24]{braides_gamma-convergence_2002}) for the definition of the (lower and upper) $\Gamma$-limit. From the properties of the $\Gamma$-convergence, see~\cite[Cor.~7.20]{dal_maso_introduction_1993}, it follows that $\theta_\delta(y_0)\H^1\mres [-\tau, \tau]$ must be a minimizer of $F$ (as the limit of minimizers of $F_r$). The estimate from below of the $\Gamma$-liminf is obtained with the tools developed so far, while estimating the $\Gamma$-limsup will require an appropriate construction illustrated in Figure~\ref{figure.gamma_convergence}.

\begin{figure}[t]
    \centering
    \resizebox{1\linewidth}{!}{%

\tikzset{every picture/.style={line width=0.75pt}} %set default line width to 0.75pt        

\begin{tikzpicture}[x=0.55pt,y=0.55pt,yscale=-1,xscale=1]
\centering

%Shape: Ellipse [id:dp7280026510110482] 
\draw   (1.6,80.3) .. controls (1.6,36.72) and (37.78,1.4) .. (82.4,1.4) .. controls (127.02,1.4) and (163.2,36.72) .. (163.2,80.3) .. controls (163.2,123.88) and (127.02,159.2) .. (82.4,159.2) .. controls (37.78,159.2) and (1.6,123.88) .. (1.6,80.3) -- cycle ;
%Shape: Ellipse [id:dp27258468775905165] 
\draw   (247.6,79.8) .. controls (247.6,35.95) and (282.88,0.4) .. (326.4,0.4) .. controls (369.92,0.4) and (405.2,35.95) .. (405.2,79.8) .. controls (405.2,123.65) and (369.92,159.2) .. (326.4,159.2) .. controls (282.88,159.2) and (247.6,123.65) .. (247.6,79.8) -- cycle ;
%Shape: Ellipse [id:dp8595519272689727] 
\draw   (489,80.31) .. controls (489,36.18) and (526.88,0.41) .. (573.6,0.41) .. controls (620.32,0.41) and (658.2,36.18) .. (658.2,80.31) .. controls (658.2,124.43) and (620.32,160.2) .. (573.6,160.2) .. controls (526.88,160.2) and (489,124.43) .. (489,80.31) -- cycle ;
%Straight Lines [id:da8029241757498857] 
\draw [color={rgb, 255:red, 208; green, 2; blue, 27 }  ,draw opacity=1 ][line width=1.5]    (82.4,1.4) -- (82.4,159.2) ;
%Curve Lines [id:da9596599817820421] 
\draw [color={rgb, 255:red, 74; green, 144; blue, 226 }  ,draw opacity=1 ][line width=1.5]    (53.23,89.16) .. controls (96.44,57.52) and (66.2,108.15) .. (109.41,76.5) ;
%Curve Lines [id:da8562773488463475] 
\draw [color={rgb, 255:red, 74; green, 144; blue, 226 }  ,draw opacity=1 ][line width=1.5]    (41.35,123.97) .. controls (43.51,132.41) and (49.99,144.01) .. (82.4,132.41) ;
%Curve Lines [id:da6332560134606309] 
\draw [color={rgb, 255:red, 74; green, 144; blue, 226 }  ,draw opacity=1 ][line width=1.5]    (82.4,40.64) .. controls (93.2,29.04) and (101.84,39.58) .. (123.45,42.75) ;
%Curve Lines [id:da4024352528991042] 
\draw [color={rgb, 255:red, 74; green, 144; blue, 226 }  ,draw opacity=1 ][line width=1.5]    (297.96,88.5) .. controls (340.1,56.66) and (310.6,107.61) .. (352.74,75.77) ;
%Curve Lines [id:da9066893045684388] 
\draw [color={rgb, 255:red, 74; green, 144; blue, 226 }  ,draw opacity=1 ][line width=1.5]    (286.37,123.53) .. controls (288.47,132.03) and (294.8,143.7) .. (326.4,132.03) ;
%Curve Lines [id:da4092528707655638] 
\draw [color={rgb, 255:red, 74; green, 144; blue, 226 }  ,draw opacity=1 ][line width=1.5]    (326.4,39.68) .. controls (336.93,28) and (345.36,38.61) .. (366.43,41.8) ;
%Curve Lines [id:da06423470509074702] 
\draw [color={rgb, 255:red, 208; green, 2; blue, 27 }  ,draw opacity=1 ][line width=1.5]    (326.4,0.4) .. controls (269.51,28.21) and (372.75,121.62) .. (326.4,159.2) ;
%Curve Lines [id:da952492814300572] 
\draw [color={rgb, 255:red, 74; green, 144; blue, 226 }  ,draw opacity=1 ][line width=1.5]    (545.32,88.85) .. controls (590.57,56.81) and (558.9,108.08) .. (604.14,76.03) ;
%Curve Lines [id:da06892596236828452] 
\draw [color={rgb, 255:red, 74; green, 144; blue, 226 }  ,draw opacity=1 ][line width=1.5]    (545.32,123.03) .. controls (547.59,131.58) and (554.37,143.32) .. (588.3,131.58) ;
%Curve Lines [id:da25745993172274795] 
\draw [color={rgb, 255:red, 74; green, 144; blue, 226 }  ,draw opacity=1 ][line width=1.5]    (558.9,40.79) .. controls (570.21,29.04) and (579.26,39.72) .. (601.88,42.92) ;
%Curve Lines [id:da6337809499575182] 
\draw [color={rgb, 255:red, 208; green, 2; blue, 27 }  ,draw opacity=1 ][line width=1.5]    (575.86,0.2) .. controls (514.79,28.18) and (625.63,122.18) .. (575.86,159.99) ;
%Straight Lines [id:da990987956820655] 
\draw [color={rgb, 255:red, 208; green, 2; blue, 27 }  ,draw opacity=1 ]   (168.4,91.2) -- (237.8,91.2) ;
\draw [shift={(239.8,91.2)}, rotate = 180] [color={rgb, 255:red, 208; green, 2; blue, 27 }  ,draw opacity=1 ][line width=0.75]    (10.93,-3.29) .. controls (6.95,-1.4) and (3.31,-0.3) .. (0,0) .. controls (3.31,0.3) and (6.95,1.4) .. (10.93,3.29)   ;
%Straight Lines [id:da7852640735220076] 
\draw [color={rgb, 255:red, 74; green, 144; blue, 226 }  ,draw opacity=1 ]   (410.4,89.2) -- (479.8,89.2) ;
\draw [shift={(481.8,89.2)}, rotate = 180] [color={rgb, 255:red, 74; green, 144; blue, 226 }  ,draw opacity=1 ][line width=0.75]    (10.93,-3.29) .. controls (6.95,-1.4) and (3.31,-0.3) .. (0,0) .. controls (3.31,0.3) and (6.95,1.4) .. (10.93,3.29)   ;

% Text Node
\draw (93.72,11.43) node [anchor=north west][inner sep=0.75pt]  [font=\normalsize,color={rgb, 255:red, 74; green, 144; blue, 226 }  ,opacity=1 ]  {$\Gamma _{i}$};
% Text Node
\draw (88.29,115.11) node [anchor=north west][inner sep=0.75pt]  [font=\normalsize,color={rgb, 255:red, 208; green, 2; blue, 27 }  ,opacity=1 ]  {$[ -\tau ,\tau ]$};
% Text Node
\draw (355.88,109.37) node [anchor=north west][inner sep=0.75pt]  [font=\normalsize,color={rgb, 255:red, 208; green, 2; blue, 27 }  ,opacity=1 ]  {$\Sigma _{r}$};
% Text Node
\draw (337.45,9.81) node [anchor=north west][inner sep=0.75pt]  [font=\normalsize,color={rgb, 255:red, 74; green, 144; blue, 226 }  ,opacity=1 ]  {$\Gamma _{i}$};
% Text Node
\draw (567.24,12.62) node [anchor=north west][inner sep=0.75pt]  [font=\normalsize,color={rgb, 255:red, 74; green, 144; blue, 226 }  ,opacity=1 ]  {$\mathcal{T}_{h_{i}}( \Gamma _{i})$};
% Text Node
\draw (187.5,60.01) node [anchor=north west][inner sep=0.75pt]  [color={rgb, 255:red, 208; green, 2; blue, 27 }  ,opacity=1 ] [align=left] {Cost};
% Text Node
\draw (116,153.8) node [anchor=north west][inner sep=0.75pt]  [font=\small,color={rgb, 255:red, 208; green, 2; blue, 27 }  ,opacity=1 ]  {$d_{H}\left(\Sigma_{r} ,[ -\tau ,\tau ]\right) \nu ([ -\tau ,\tau ])$};
% Text Node
\draw (372,148.4) node [anchor=north west][inner sep=0.75pt]  [font=\small,color={rgb, 255:red, 74; green, 144; blue, 226 }  ,opacity=1 ]  {$d_{H}\left(\Sigma_{r} ,[ -\tau ,\tau ]\right) \nu ( \Gamma _{i})$};
% Text Node
\draw (429.5,58.01) node [anchor=north west][inner sep=0.75pt]  [color={rgb, 255:red, 74; green, 144; blue, 226 }  ,opacity=1 ] [align=left] {Cost};
% Text Node
\draw (605.88,110.37) node [anchor=north west][inner sep=0.75pt]  [font=\normalsize,color={rgb, 255:red, 208; green, 2; blue, 27 }  ,opacity=1 ]  {$\Sigma _{r}$};

\end{tikzpicture}
}
    \caption{Transportation argument for the construction of a recovery sequence in the $\Gamma$ convergence of $(F_r)_{r>0}$. Both operations have a transportation cost of the order $\displaystyle d_{H}\left(\Sigma_r ,[ -\tau ,\tau ]\right)$, and hence converge to $0$.}\label{figure.gamma_convergence}
\end{figure}

\begin{theorem}\label{theorem.Gamma_convergence}
The family of functionals ${\left(F_r\right)}_{r > 0}$ $\Gamma$-converges to $F$ as $r\to 0^+$, in the narrow topology.
\end{theorem}
\begin{proof}

    $\Gamma$-liminf: we consider an infinitesimal sequence $(r_n)_{n \in \mathbb{N}}$ such that ${(\nu'_n)}_{n \in \mathbb{N}}$ converges to $\nu'$ in the narrow sense in $\overline{B_1(0)}$, and that $\liminf_{n \to \infty} F_{r_n}(\nu_n') < \infty$ for all $n \in \mathbb{N}$, otherwise there is nothing to prove. 
    
    First we look at the first marginals in the definition of $F_{r_n}$. From~\eqref{blowup_measures_sigma} we know that $\bar \sigma_{\delta, r_n}\xrightharpoonup[n \to \infty]{\star} \bar\sigma_{\delta}$. 
    By the lower semi-continuity of the Wasserstein distance w.r.t.~the narrow convergence, if we prove that $F(\nu') < \infty$, that is, if the limit satisfies the constraints in the definition of $F$, we will have that
    \[
        F(\nu') \le \liminf_{n \to \infty} F_{r_n}(\nu_n').
    \]
    
    As $\alpha \nu_n' \ge \H^1\mres \left(\Gamma_{n}\setminus\Sigma_{r_n}\right)$ for some  $\Gamma_n \subset \overline{B_1(0)}$ such that $\displaystyle \left(\frac{\Sigma-y_0}{r_n}\right) \cup \Gamma_n \in \mathcal{A}$, Blaschke's Theorem~\cite[Thm.~6.1]{ambrosio2000functions}
    and  Lemma~\ref{lemma.blowup_domain_measure} imply that, up to a subsequence, $\Gamma_n  \xrightarrow[n \to \infty]{d_H} \Gamma$ for some closed set $\Gamma \subset \overline{ B_1(0)}$ and $\displaystyle \frac{\Sigma-y_0}{r_n}\xrightarrow[n \to \infty]{K} \RR\tau$. Hence,
    \[
        \Xi_n 
        \eqdef 
        \left(
            \frac{\Sigma - y_0}{r_n}
        \right) 
        \cup \Gamma_n
        \xrightarrow[n \to \infty]{K}
        \Xi 
        \eqdef 
        \RR\tau \cup \Gamma.
    \]
    
    Let us check that $\Xi$ is connected (which is not immediate since the Kuratowski limit of connected sets is not necessarily connected). 
    Assume by contradiction that there are two disjoint open sets $U,V \subset \RR^d$ such that $U\cap \Xi$ and $V\cap \Xi$ form a partition of $\Xi$. Since $\RR\tau \subset \Xi$ is connected, it is contained in either $U$ or $V$ (say, $U$). As a result, $V\cap \Xi\subset \Gamma \subset \overline{B_1(0)}$ is bounded, and possibly replacing $V$ with $V\cap B_2(0)$, we may assume that $V$ is bounded too, so that $\partial V$ is compact. 
     Since $\Xi\subset V\cap (\RR^d\setminus \overline{V})$, we note that $\partial V \cap \Xi=\emptyset$, and we deduce that $\min_{x\in \partial V} \dist(x,\Xi)>0$.
     
     Now, the Kuratowski convergence of $\Xi_n$ towards $\Xi$ implies that, for all $n$ large enough, $\Xi_n$ intersects both $V$ and $U\subset \RR^d\setminus \overline{V}$, hence, by the connectedness of $\Xi_n$,  there exists $x_n\in \Xi_n \cap \partial V$.
     But the Kuratowski convergence also implies that $\dist(\cdot,\Xi_n) \xrightarrow[]{} \dist(\cdot,\Xi)$ locally uniformly (hence uniformly on $\partial V$), which contradicts that $\min_{x\in \partial V} \dist(x,\Xi)>0$. As a result, $\Xi$ is connected.

    The fact that $\supp\nu' \subset [-\tau, \tau]\cup \Gamma$ comes from the weak convergence of $\nu'_n$ to $\nu'$. As this convergence takes place in a compact set it also holds that $\nu'(\overline{B_1(0)}) = \disp \lim_{n \to \infty} \nu'_n(\overline{B_1(0)}) = 2\theta_\delta(y_0)$ since $\theta_\delta(y_0)$ is the density of $\nu_\delta$ at $y_0$. 

    It only remains to verify the density constraints, $\alpha \nu' \ge \mathcal{H}^1\mres \left(\Gamma\setminus [-\tau,\tau]\right)$.
    We cannot apply \Golab's Theorem to $\nu_n'$ since, although $\alpha\nu_n' \ge  \H^1\mres \left( \Gamma_n\setminus \Sigma_{r_n}\right)$, we do not have an upper bound on the number of connected components of $\Gamma_n\setminus \Sigma_{r_n}$. What we do know is that the sequence $\Xi_n = \disp r_n^{-1}(\Sigma  - y_0) \cup \Gamma_n$ satisfies the assumptions of Theorem \ref{theorem.Golab_localversion}, 
    so we apply it to the measures $\H^1\mres\Xi_n$ instead, remembering that
    \begin{align*}
       \mathcal{H}^1\mres 
        \left(
            \frac{\Sigma - y_0}{r_n}
        \right)
        + 
        \alpha\nu_n'
        \ge 
        \mathcal{H}^1\mres 
        \left(
            \frac{\Sigma - y_0}{r_n}\cup\Gamma_n
        \right).
    \end{align*}
    The left-hand side converges in the local weak-$\star$ sense to $\H^1\mres \RR\tau + \alpha\nu'$. The right-hand side (which is bounded by the left-hand side) 
    converges in the same sense, up to a subsequence. We let
    $\lambda$ denote a limit and Theorem~\ref{theorem.Golab_localversion} implies that $\lambda \ge \H^1\mres (\RR\tau\cup \Gamma)$, which gives
    $
        \H^1\mres \RR\tau + \alpha\nu' \ge \H^1\mres (\RR\tau\cup \Gamma),
    $
    and thus 
    \[
        \alpha\nu' \ge \H^1\mres \left(\Gamma\setminus [-\tau,\tau]\right).
    \]

    \par $\Gamma$-limsup: Let $(r_{n})_{n\in\mathbb{N}}$ be an infinitesimal sequence. By  Lemma~\ref{lemma.blowup_domain_measure}, we know that $(\Sigma-y_0)/r_{n}$ converges in the Kuratowski sense towards $\RR\tau$, and $\Sigma_{r_n}\eqdef (\Sigma-y_0)/r_{n}\cap\overline{B_1(0)}$ converges towards $[-\tau,\tau]$ for the Hausdorff distance.

    The strategy to prove the limsup is illustrated in Figure~\ref{figure.gamma_convergence},
    and roughly explained as follows. We concatenate three steps. First we renormalize $\nu'$ to satisfy the mass constraint in $F_{r_{n}}$. But this normalization may break the condition $\alpha \nu_n'\ge \H^1\mres \left(\Gamma\setminus [-\tau,\tau]\right)$, so we slightly shrink the support to satisfy this constraint again. We also need the measure $\nu_{n}'$ to be supported on some connected set $\Sigma_{r_{n}}\cup \Gamma_{n}$, hence we move the mass of $\nu'$ from $[-\tau,\tau]$ to $\Sigma_{r_n}$ by projection, and we translate the mass of each connected component of the (shrinked)  $\Gamma\setminus [-\tau,\tau]$ so that it is connected to $\Sigma_{r_n}$. Eventually, by doing so, some parts of the support may get out of $\overline{B_1(0)}$, so we project the residual mass onto $\overline{B_1(0)}$.
    
To be more precise, we first address the case $\theta_\delta(y_0)=0$. As $F(\nu')<+\infty$ if and only if $\nu'=0$, we need only prove the result for $\nu'=0$. Let $P_n$ be any measurable selection of the projection onto ${\Sigma_{r_n}}$, and define $\nu_n'\eqdef {P_n}_{\sharp} \bar \sigma_{\delta, r_n}$. With $\Gamma=\emptyset$, and since $|x-P_n(x)|\le \delta^{-1}$ for all $x\in \supp \bar \sigma_{\delta,r_n}$, we observe that 
\[
    F_{r_n}(\nu'_n)\le W_p^p(\bar \sigma_{\delta, r_n},\nu_n')\le \delta^{-p} \frac{\nu_\delta(B_{r_n})}{r_n} \xrightarrow[n\to +\infty]{} 0 = F(\nu').
\]
Moreover, as $\nu'_n \xrightharpoonup[n \to +\infty]{} \nu'$ in the narrow topology, we have built a recovery sequence for $\nu'$.
    
Now, we deal with the  case $\theta_{\delta}(y_0)>0$. Let $\nu'$ such that $F(\nu') < +\infty$, and let $\Gamma$ be a set as in~\eqref{Gamma_limitF}. Observe that $[-\tau,\tau]\cup \Gamma$ is connected, being  the projection of $\RR\tau\cup \Gamma$ onto $\overline{B_1(0)}$, and since it has finite $\H^1$ measure, it is arcwise connected, by~\cite[Prop. 30.1, Cor. 30.2]{david2006singular}. As a result, $\RR\tau\cup\Gamma$  is arcwise connected too.

Let $(C_{i})_{i\in I}$ denote the arcwise connected components of $\Gamma\setminus (\RR\tau)$.  For each $i \in I$, as the set $\RR\tau\cup \Gamma$ is arcwise connected, one may check that there exists some $z_i\in [-\tau,\tau]$ such that $\{z_{i}\}\cup  C_{i}$ is arcwise connected. As a result, the set $C_{i} \subset \RR^{d}\setminus(\RR\tau)$ cannot consist of one single point, and $\H^1(C_i)>0$. Therefore, the index set  $I$ is at most countable.

Let us construct a recovery sequence $(\nu_n')_{n\in \mathbb{N}}$. By the Kuratowski (even Hausdorff) convergence  of $\Sigma_{r_{n}}$ towards $[-\tau,\tau]$, for each $i\in I$, there exists a sequence $(z_{n,i})_{n\in \mathbb{N}}$ such that $z_{n,i}\in \Sigma_{r_{n}}$ for each $n\in \mathbb{N}$, and $z_{n,i}\to z_i$. We then define 
\begin{align*}
    a_n\eqdef \frac{\nu_{\delta}(B_{r_n})}{2r_n\theta_{\delta}(y_0)},
    \quad\mbox{and}\quad   
    s_n\eqdef \max(1,a_{n}^{-1}), 
\end{align*}
noting that $a_n\to 1$ and $s_n\to 1$, and we introduce the map $T_n$,
\[
    T_n(y) 
    \eqdef 
    \begin{cases}
        P_n(y/s_{n}),& \text{ if $y \in [-\tau, \tau]$,} \\ 
        (y-z_{i})/s_n+z_{n,i},& \text{ if $y \in C_i$},
    \end{cases}
\]
where, as before, $P_n$ is some measurable selection of the projection onto ${\Sigma_{r_n}}$. The map $T_n$ shrinks each connected component $C_i$ and translates it to the corresponding $z_{n,i} \in \Sigma_{r_n}$ so as to ensure connectedness (see below). Letting $P_B$ denote the projection onto the unit ball $\overline{B_1(0)}$, we eventually define
\[ 
    \nu_n'\eqdef (P_B\circ T_n)_{\sharp}(a_n\nu').
\]
 
Let us check that $\nu_n'$ converges to $\nu'$ in the narrow topology.
We note that for $y\in [-\tau,\tau]$,
\[
    |y/s_{n} - P_n(y/s_{n})| = \dist\left(y/s_{n}, \Sigma_{r_n}\right) \le d_H\left([-\tau, \tau], \Sigma_{r_n}\right)\xrightarrow[n\to +\infty]{} 0,
\]
so that $T_{n}(y)\to y$, and for $y\in C_{i}$,
\[
    \abs{y-T_{n}(y)}\le \abs{y}(1-1/s_{n}) + \abs{z_i/s_n-z_{n,i}}\xrightarrow[n\to +\infty]{} 0.
\]
As a result, for $y\in [-\tau,\tau]\cup \Gamma$, $T_n(y)\to y$, and eventually $P_B\circ T_n(y)\to y$.
By the dominated convergence theorem, we get that for any $\phi\in C_b(\RR^d)$,

\[ 
    \int\phi\dd \nu_n' 
    = 
    a_n\int_{[-\tau,\tau]\cup \Gamma}\phi\left(P_B(T_n(y))\right)\dd \nu'(y) 
    \xrightarrow[n\to +\infty]{}  
    \int_{[-\tau,\tau]\cup \Gamma}\phi\left(y\right)\dd \nu'(y)
\]
so that $\nu_n'\xrightharpoonup[n\to +\infty]{}\nu'$ in the narrow topology.

Let us now check the constraints in $F_{r_n}$. From the properties of image measures, we see that $\supp \nu_n'\subset \overline{B_{1}(0)}$, and that $\nu_n'(\overline{B_{1}(0)})=\nu_n'(\RR^d)=a_n\nu'(\RR^{d})= \nu_{\delta}(B_{r_n})/r_n$, so that $\nu_n'$ has the mass prescribed by $F_{r_n}$. Consider the set
\begin{equation}\label{eq.defgamman}
    \Gamma_n\eqdef 
    \bigcup_{i\in I} \Gamma_{n,i}
    \quad 
    \mbox{where }
    \quad 
    \Gamma_{n,i}\eqdef \overline{(P_B\circ T_{n})(C_{i})}.
\end{equation}

In addition, the mass of $\nu_n'$ is concentrated in $\Sigma_{r_n}\cup \Gamma_n$, and we prove below that satisfies all the constraints in $F_{r_{n}}$.

First let us show that $\displaystyle \frac{\Sigma - y_0}{r_n} \cup \Gamma_n$ is connected. For each $i\in I$, as the set $\{z_{i}\}\cup C_{i}$ is arcwise connected, so is its image by the map $y\mapsto (y-z_{i})/s_n+z_{n,i}$, which is equal to $\{z_{n,i}\}\cup T_n(C_i)$. As a result $\{z_{n,i}\}\cup \overline{P_B\circ T_n(C_i)} = \{z_{n,i}\}\cup \Gamma_{n,i}$ is connected, as well as $\displaystyle \frac{\Sigma - y_0}{r_n} \cup \Gamma_n$. 

Now, let us show that $\displaystyle \frac{\Sigma - y_0}{r_n} \cup \Gamma_n$ is closed. If $I$ is finite, then, by~\eqref{eq.defgamman}, $\displaystyle r_n^{-1}(\Sigma - y_0) \cup \Gamma_n$ is closed as the finite union of closed sets.
Otherwise, $I$ is countable, and from~\cite[Lemma~2.6]{paolini2013existence}, we have
\[
    \H^1(\Gamma_{n,i}) 
    = 
    \H^1(P_B\circ T_n(C_i)) 
    \le 
    \H^1(T_n(C_i))
    = 
    s_n^{-1} \H^1(C_i)
    \xrightarrow[i \to \infty]{} 0.
\]
Let ${\left(x_k\right)}_{k \in \mathbb{N}}$ be a sequence contained in $\displaystyle r_n^{-1}(\Sigma - y_0) \cup \Gamma_n$, such that $x_k \to x$. If there is an infinite amount of terms of this sequence in either $\displaystyle \frac{\Sigma - y_0}{r_n}$ or any of the $\Gamma_{n,i}$, since these sets are closed, then $x \in \displaystyle \frac{\Sigma - y_0}{r_n} \cup \Gamma_n$. Otherwise, we can find a sub-sequence $x_{k'} \in \Gamma_{n, i_{k'}}$, so that 
\[
    \dist\left(
        x, \frac{\Sigma - y_0}{r_n}
    \right)
    = 
    \lim_{k' \to \infty} 
    \dist\left(
        x_{k'}, \frac{\Sigma - y_0}{r_n}
    \right)
    \le 
    \lim_{k' \to \infty} \H^1(\Gamma_{n,i_{k'}}) = 0,
\]
and we conclude that $\displaystyle \frac{\Sigma - y_0}{r_n} \cup \Gamma_n$ is closed. 

To show it satisfies the density constraints, take any non-negative $\phi\in C_b(\RR^d)$,
\begin{align*}
    \alpha\int\phi\dd \nu_n' 
    &= \alpha a_n\int_{[-\tau,\tau]\cup \Gamma}\phi\left(P_B(T_n(y))\right)\dd \nu'(y)\\
    &\ge \alpha a_n\sum_{i\in I}\int_{C_{i}}\phi\left(P_B(T_n(y))\right)\dd \nu'(y)\\
    &\ge a_n\sum_{i\in I}\int_{C_{i}}\phi\left(P_B((y-z_{i})/s_{n}+z_{n,i})\right)\dd \H^{1}(y)\\
    &= a_n s_{n}\sum_{i\in I}\int_{\Gamma_{n,i}}\phi\left(P_B(y')\right)\dd \H^{1}(y')\\ 
    &
    \ge \int_{\Gamma_{n}} \phi \dd \H^{1}. 
\end{align*}
It follows that $\alpha \nu_n'\geq \H^1\mres \Gamma_{n}$ and we conclude that $F_{r_n}(\nu_n) < \infty$, for all $n \in \mathbb{N}$.

By the continuity of the Wasserstein distance with
respect to the narrow convergence (provided the measures are supported in some common compact set), we have that:
    \[
        F_{r_n}(\nu_n') \xrightarrow[n \to \infty]{} F(\nu').
    \]
The $\Gamma$-convergence follows.
\end{proof}

Now that we have characterized the limit problem, we show that the optimal transportation is given by projections as the blow-up family.
\begin{lemma}\label{lemma.brenier_map_blowup_family}
If $\theta_\delta(y_0)>0$, the optimal transport plan between the measure $\sigma_{\delta}$, defined in~\eqref{blowup_measures_sigma}, and $\bar \nu = \theta_\delta(y_0)\H^1\mres [-\tau, \tau]$, defined in~\eqref{blowup_measures_nu}, is unique and given by the projection map $\Pi_{[-\tau, \tau]}$.
\end{lemma}
\begin{proof}
    Consider a family $\bar \gamma_r$ of optimal transport plans from $\bar \sigma_{\delta , r}$ to $\bar \nu_{\delta ,r}$. Up to a subsequence it converges to some $\bar \gamma$, which, by the stability of optimal transport plans, also transports $\sigma_{\delta}$ to $\bar \nu$ optimally. Since $\bar \sigma_{\delta , r}, \bar \nu_{\delta ,r}$ are generated by the pushforward of $\nuexc\mres B_r(y_0)$ by $\Phi^{y_0, r}$, from Lemma~\ref{lemma.minimal_distance2Sigma} we know that 
    \[
        \supp \bar \gamma_r \subset \text{graph}\left(\Pi_{\Sigma_r}\right).  
    \]

    Let us show that $\supp \bar \gamma \subset \text{graph}\left(\Pi_{[-\tau, \tau]}\right)$. Indeed if $(x,p) \in \supp \bar \gamma$, there is an open ball $B$ centered at $(x,p)$ such that
    \[
        0 < \bar \gamma(B) \le \liminf_{r \to 0} \bar\gamma_r(B).
    \]
    In particular, we can find $\supp \bar \gamma_r \ni (x_r, p_r) \xrightarrow[r \to 0]{} (x,p)$. So it holds that
    \[
        |x - p| 
        = 
        \lim_{r \to 0}   
        |x_r - p_r|
        =  
        \lim_{r \to 0}   
        \dist\left(x_r, \Sigma_r\right) 
        = 
        \dist(x, [-\tau, \tau]),
    \]
    where the last equality comes from the uniform convergence of the distance functions, recalling from Lemma~\ref{lemma.blowup_domain_measure} that $\Sigma_r \xrightarrow[r \to 0]{d_H} [-\tau, \tau]$. 

    Now we show that this property is true for any other optimal plan. Consider $\gamma$ transporting $\sigma_{\delta}$ to $\bar \nu$ optimaly, then by the optimality of $\bar \gamma$ it holds that
    \begin{align*}
        \int_{\RR^d} (\dist(x, [-\tau, \tau]))^p\dd \sigma_{\delta} 
        &= 
        \int |x - y|^p\dd \bar\gamma
        =        \int |x - y|^p\dd \gamma\\
        &\ge 
        \int (\dist(x, [-\tau, \tau]))^p \dd \gamma
        = 
        \int_{\RR^d} \dist(x, [-\tau, \tau])^p\dd \sigma_{\delta}.
    \end{align*}
    Since $|x - y| - \dist(x,[-\tau, \tau]) \ge 0 $ for $\gamma$-a.e. $(x,y)$ and the inequality above must be an equality, we must have $\supp \gamma \subset \text{graph}\left(\Pi_{[-\tau, \tau]}\right)$ for any optimal $\gamma$. In particular, as $\Pi_{[-\tau, \tau]}$ is uni-valued, it means that the optimal transport plan is unique and given by the projection map.
\end{proof}

\subsection{Competitor for the limit problem and existence for~\eqref{problem.shape_optimization}} 
Given $y_0\in \Sigma$ such that \eqref{y0fixed_for_blowup} holds, it follows from Theorem~\ref{theorem.Gamma_convergence} that:
\[
    \bar \nu_{\delta} \eqdef \theta_\delta(y_0)\H^1\mres [-\tau, \tau] 
    \in 
    \argmin F,
\]
where $F$ is defined in~\eqref{Gamma_limitF}. 
In addition, Lemma~\ref{lemma.brenier_map_blowup_family} shows that if $\theta_\delta(y_0)>0$,
%However, from the fact 
 the optimal transportation of $\bar \sigma_{\delta}$ to $\bar \nu_{\delta}$ is given by the orthogonal projection.
We show that in this case, we can lower the energy by projecting part of the mass to a (closer) horizontal line as in Figure~\ref{figure.better_competitor}. 

This contradicts the existence of rectifiability points of $\Sigma$ such that $\theta_\delta(y_0) > 0$ so that $\nu_\delta \equiv 0$, and shows
%. This construction is done in 
the following Lemma:
%\if{
\begin{figure}[t]
  % \centering
  %\includegraphics[width=.95\textwidth]{images/better_competitor.png}
  \tikzset{every picture/.style={line width=0.75pt}} %set default line width to 0.75pt        

\begin{tikzpicture}[x=0.75pt,y=0.75pt,yscale=-1,xscale=1]

%Straight Lines [id:da9284780184998143] 
\draw [color={rgb, 255:red, 74; green, 144; blue, 226 }  ,draw opacity=1 ]   (192.68,133.07) -- (261.24,133.07) -- (296.03,132.58) ;
%Straight Lines [id:da41797877633750025] 
\draw [fill={rgb, 255:red, 208; green, 2; blue, 27 }  ,fill opacity=1 ]   (194.54,1.89) -- (194.54,68.13) ;
%\draw [shift={(194.54,1.89)}, rotate = 270] [color={rgb, 255:red, 0; green, 0; blue, 0 }  ][line width=0.75]    (0,5.59) -- (0,-5.59)   ;
%Straight Lines [id:da29925025923187065] 
\draw    (194.54,195.75) -- (194.54,261.99) ;
%\draw [shift={(194.54,261.99)}, rotate = 270] [color={rgb, 255:red, 0; green, 0; blue, 0 }  ][line width=0.75]    (0,5.59) -- (0,-5.59)   ;
%Straight Lines [id:da28725355028725397] 
\draw [color={rgb, 255:red, 208; green, 2; blue, 27 }  ,draw opacity=1 ] [dash pattern={on 0.84pt off 2.51pt}]  (194.54,68.13) -- (194.54,173.94) -- (194.54,195.75) ;
%\draw [shift={(194.54,195.75)}, rotate = 270] [color={rgb, 255:red, 208; green, 2; blue, 27 }  ,draw opacity=1 ][line width=0.75]    (0,5.59) -- (0,-5.59)   ;
%\draw [shift={(194.54,68.13)}, rotate = 270] [color={rgb, 255:red, 208; green, 2; blue, 27 }  ,draw opacity=1 ][line width=0.75]    (0,5.59) -- (0,-5.59)   ;
%Straight Lines [id:da7676109422398874] 
\draw [color={rgb, 255:red, 208; green, 2; blue, 27 }  ,draw opacity=1 ] [dash pattern={on 0.84pt off 2.51pt}]  (507.66,90.76) -- (196.2,90.76) ;
\draw [shift={(194.2,90.76)}, rotate = 360] [fill={rgb, 255:red, 208; green, 2; blue, 27 }  ,fill opacity=1 ][line width=0.08]  [draw opacity=0] (12,-3) -- (0,0) -- (12,3) -- cycle    ;
%Straight Lines [id:da8059103254444953] 
\draw [color={rgb, 255:red, 74; green, 144; blue, 226 }  ,draw opacity=1 ] [dash pattern={on 4.5pt off 4.5pt}]  (194.2,90.76) -- (259.55,132) ;
\draw [shift={(261.24,133.07)}, rotate = 212.25] [fill={rgb, 255:red, 74; green, 144; blue, 226 }  ,fill opacity=1 ][line width=0.08]  [draw opacity=0] (12,-3) -- (0,0) -- (12,3) -- cycle    ;
%Straight Lines [id:da5417893048649498] 
\draw [color={rgb, 255:red, 74; green, 144; blue, 226 }  ,draw opacity=1 ] [dash pattern={on 4.5pt off 4.5pt}]  (194.54,173.94) -- (259.53,134.11) ;
\draw [shift={(261.24,133.07)}, rotate = 148.5] [fill={rgb, 255:red, 74; green, 144; blue, 226 }  ,fill opacity=1 ][line width=0.08]  [draw opacity=0] (12,-3) -- (0,0) -- (12,3) -- cycle    ;
%Straight Lines [id:da38300224604556776] 
\draw [color={rgb, 255:red, 208; green, 2; blue, 27 }  ,draw opacity=1 ] [dash pattern={on 0.84pt off 2.51pt}]  (508,173.94) -- (196.54,173.94) ;
\draw [shift={(194.54,173.94)}, rotate = 360] [fill={rgb, 255:red, 208; green, 2; blue, 27 }  ,fill opacity=1 ][line width=0.08]  [draw opacity=0] (12,-3) -- (0,0) -- (12,3) -- cycle    ;
%Straight Lines [id:da7452526926147913] 
\draw [fill={rgb, 255:red, 208; green, 2; blue, 27 }  ,fill opacity=1 ] [dash pattern={on 4.5pt off 4.5pt}]  (387.29,1.89) -- (387.48,261.22) ;
%Straight Lines [id:da49831210250449764] 
\draw  [color={rgb, 255:red, 0; green, 0; blue, 0 }  ][dash pattern={on 4.5pt off 4.5pt}]  (197.2,269.8) -- (386.2,269.8) ;
\draw [shift={(388.2,269.8)}, rotate = 180] [color={rgb, 255:red, 0; green, 0; blue, 0 }  ][line width=0.75]    (10.93,-3.29) .. controls (6.95,-1.4) and (3.31,-0.3) .. (0,0) .. controls (3.31,0.3) and (6.95,1.4) .. (10.93,3.29)   ;
\draw [shift={(195.2,269.8)}, rotate = 0] [color={rgb, 255:red, 0; green, 0; blue, 0 }  ][line width=0.75]    (10.93,-3.29) .. controls (6.95,-1.4) and (3.31,-0.3) .. (0,0) .. controls (3.31,0.3) and (6.95,1.4) .. (10.93,3.29)   ;

% Text Node
\draw (140,7.98) node [anchor=north west][inner sep=0.75pt]  [font=\footnotesize]  {$[ -e_{d} ,e_{d}]$};
% Text Node
\draw (276.4,275.76) node [anchor=north west][inner sep=0.75pt]  [color={rgb, 255:red, 0; green, 0; blue, 0 }  ][font=\footnotesize]  {$\delta /2$};
% Text Node
\draw (263.24,136.47) node [anchor=north west][inner sep=0.75pt]  [font=\footnotesize,color={rgb, 255:red, 74; green, 144; blue, 226 }  ,opacity=1 ]  {$\ell ( \varepsilon ')$};
% Text Node
\draw (156.77,84.18) node [anchor=north west][inner sep=0.75pt]  [font=\footnotesize,color={rgb, 255:red, 74; green, 144; blue, 226 }  ,opacity=1 ]  {$\overline{s} +\varepsilon '$};
% Text Node
\draw (157.68,165.74) node [anchor=north west][inner sep=0.75pt]  [font=\footnotesize,color={rgb, 255:red, 74; green, 144; blue, 226 }  ,opacity=1 ]  {$\overline{s} -\varepsilon '$};
% Text Node
\draw (420.53,6.12) node [anchor=north west][inner sep=0.75pt]    {$\supp \theta _{i}$};
% Text Node
\draw (181.77,126.18) node [anchor=north west][inner sep=0.75pt]  [font=\footnotesize,color={rgb, 255:red, 74; green, 144; blue, 226 }  ,opacity=1 ]  {$\overline{s}$};

\end{tikzpicture}
  \caption{Construction of a competitor for the minimization of $F$.}
  \label{figure.better_competitor}
\end{figure}
%}\fi
%\input{better_competitor_tkz.tex}
\begin{lemma}\label{lemma.better_competitor}
    For any $\delta > 0$, the measures $\nu_\delta$ defined in~\eqref{measure_nu_delta} vanish. %are null. 
\end{lemma}
\begin{proof}
    Up to a rotation, we may assume that $\tau = e_d$, where $(e_i)_{i = 1}^d$ is a basis of ${\RR}^d$. Since $\bar \sigma_{\delta}$ is supported on $\displaystyle \left\{x = (x',x_d) \in {\RR}^d : |x'|>\delta,\  |x_d|\le 1 \right\}$, 
    we can cover its support with finitely many sets
    $(E_i)_{i=1}^N$ defined as:
    \[
        E_i
        \eqdef 
        \left\{
            x = (x',x_d) \in {\RR}^d : 
                 \inner{\xi_i, x} > \delta/2, 
                 \  |x_d| \le 1
        \right\}
    \]
    where $\xi_i \in \mathbb{S}^{d-1} \cap [e_d]^\perp$ are ``horizontal'' unit
    vectors and $N$ depends only on the dimension.
    We then define a disjoint family
    \[
        F_1 = E_1, \quad F_{i+1} = E_{i+1}\setminus \bigcup_{j=1}^i F_j \text{ for $i\ge 1$}
    \]
    and decompose our measures $\bar \sigma_{\delta}$ and $\bar \nu_{\delta}$ as 
    \[
        \bar \sigma_{\delta}  = \sum_{i = 1}^N \bar \sigma_{\delta, i}, \ 
        \bar \nu_{\delta} = \sum_{i = 1}^N \bar \nu_{\delta, i}
        \text{ where } 
        \bar \sigma_{\delta, i}  \eqdef \bar \sigma_{\delta} \mres F_i 
        \text{ and }
        \bar \nu_{\delta, i} \eqdef {\proj_d}_\sharp\bar \sigma_{\delta, i},
    \]
    with $\proj_d : x\mapsto x_d e_d$ the projection onto the vertical axis.
    By Radon-Be\-si\-co\-vitch's % (?)
    differentiation theorem, $\bar \nu_{\delta, i} = \theta_i \H^1\mres[-e_d, e_d]$, where $\theta_i(s) = \theta_i(s e_d)\ge 0$ are such that
    \[
        \sum_{i = 1}^N \theta_i = \theta_\delta(y_0).    
    \]
    
    Consider $\bar s\in (-1,1)$ a common Lebesgue point of all $\theta_i$, $i=1,\dots, N$.

    Let $i$ be the index for which $\theta_i(\bar s)$ is maximal: then % we have
    $\theta_i(\bar s) \ge \theta_\delta(y_0)/N$.

    Up to a change of horizontal coordinates, we assume that $\xi_i = e_1$,
    and we introduce the  notation: $\RR^d \ni x = (x_1, x'', x_d)$ for $x'' \in \RR^{d-2}$.
    Let now:
    \begin{equation*}%\label{eq.box_set_covering}
        C_{\varepsilon} 
        \eqdef 
        F_i \cap \{ x \in\RR^d : |x_d- \bar s|<\varepsilon \}
        \subset \left\{
            x = (x_1, x'', x_d) : 
                x_1 > \delta/2, 
                \ |x_d - \bar s| < \varepsilon
        \right\}.
    \end{equation*}
    We obtain, from the fact that ${(\proj_d)}_\sharp \bar \sigma_{\delta, i} = \theta_i\H^1\mres [-e_d,e_d]$, that
    \[
        \frac{\bar \sigma_{\delta, i}(C_\varepsilon)}{2\varepsilon} = 
        \frac{1}{2\varepsilon} \int_{\bar s-\varepsilon}^{\bar s+\varepsilon}\theta_i(t)\dd t 
        \xrightarrow[\varepsilon \to 0]{} 
        \theta \eqdef  \theta_i(\bar s) \ge \frac{\theta_\delta(y_0)}{N}. 
    \]
     Now, assume by contradiction that $\theta>0$. If $\varepsilon$ is small enough, we have:
    \begin{equation}\label{eq.condition_t}
        \theta 
        \le 
        \frac{\bar \sigma_{\delta, i}(C_{\varepsilon'})}{\varepsilon'}
        \le 
        3\theta. 
    \end{equation}
    for all $\varepsilon'\le \varepsilon$.
    Now let us exploit the fact that, from Lemma~\ref{lemma.brenier_map_blowup_family}, the optimal transport is given by projections to propose a new transport map, sending the mass in $C_{\varepsilon}$ to a segment pointing towards $e_1$:
    \[
        T(x)
        \eqdef 
        \begin{cases}
            \ell(|x_d-\bar s|)e_1 + \bar s e_d,& \text{ if $x \in C_{\varepsilon}$}\\ 
            \proj_d(x),& \text{ otherwise,}
        \end{cases}
    \]
    where $\ell:[0,\varepsilon] \to \RR_+$ is defined via the conservation of mass relation, for $0\le \varepsilon'\le \varepsilon$:
    \begin{equation}\label{eq:conservmass}
        \ell(\varepsilon') = \alpha \bar \sigma_{\delta, i}(C_{\varepsilon'}).
    \end{equation}
    In other words, the mass that was sent to the vertical segment $[\bar s-\varepsilon', \bar s+\varepsilon']e_d$
    is now sent to the horizontal segment $\bar s e_d+[0, \ell(\varepsilon')]e_1$,
    for each $\varepsilon'\in [0,\varepsilon]$.
    This construction is illustrated in Figure~\ref{figure.better_competitor}.

    Thanks to~\eqref{eq:conservmass}, the map $T$ sends $\bar \sigma_{\delta, i}\mres C_\varepsilon$ to
    the measure $\alpha^{-1}\H^1\mres L$
    where $L\eqdef \bar s e_d+[0,\ell(\varepsilon)]e_1$, 
    hence, the transported measure $T_\sharp \bar \sigma_{\delta}$ satisfies the constraints in the definition~\eqref{Gamma_limitF} of the limiting functional $F$ and one has $F(T_\sharp \bar \sigma_{\delta})<+\infty$. 
    
    We shall now see that for each point $x \in C_\varepsilon$ with $x_d\neq \bar s$, it holds that
    \begin{equation}\label{eq.strict_distance_inequality}
        |x - \proj_d(x)|^p > |x - T(x)|^p.
    \end{equation}
    To show~\eqref{eq.strict_distance_inequality}, recalling the notation $x = (x_1, x'', x_d)$, it suffices that
    \begin{align*}
        |x - \proj_d(x)|^2 > |x- & T(x)|^2 
        \\
        &\Longleftrightarrow 
        x_1^2 + |x''|^2 > 
        (x_1-\ell(|x_d-\bar s|))^2+|x''|^2 + (x_d-\bar  s)^2\\
        &\Longleftrightarrow 
        2x_1\ell(|x_d-\bar s|) > \ell(|x_d-\bar s|)^2 + (x_d-\bar s)^2.
    \end{align*}
    In addition to~\eqref{eq.condition_t}, we  choose $\varepsilon$ in such a way that for any $x \in C_\varepsilon$ we have
    \[  \alpha\theta |x_d-\bar s|
        \le 
        \ell(|x_d-\bar s|) = \alpha \bar \sigma_{\delta, i}(C_{|x_d-\bar s|}) 
        \le 3\alpha\theta\varepsilon < {\left(1 +\frac{1}{(\alpha\theta)^2}\right)}^{-1}\delta
    \]
    and hence
    \[
       \ell(|x_d-\bar s|)^2+(x_d-\bar s)^2  \le
       \left(1 + \frac{1}{(\alpha\theta)^2}\right)
        \ell(|x_d-\bar s|)^2 
        < \delta \ell(|x_d-\bar s|)
         \le 2x_1 \ell(|x_d-\bar s|),    
    \]
    for all $x \in C_{\varepsilon}$, with $x_d \neq \bar s$, so that~\eqref{eq.strict_distance_inequality} holds.
    Since $\theta=\theta_i(\bar s)>0$, it follows that
    \[
        F(T_\sharp \bar \sigma_{\delta})=
        W_p^p(\bar \sigma_{\delta} , T_\sharp\bar \sigma_{\delta}) 
        < 
        W_p^p(\bar \sigma_{\delta}, \bar \nu_{\delta})
        =
        F(\bar \nu_{\delta}).     
    \]
    This contradicts the fact that $\theta_\delta(y_0) \H^1\mres [-e_d, e_d]$ is a minimizer of $F$, showing that we must have $\theta=\theta_i(\bar s)=0$ and, in turn, $\theta_\delta(y_0)=0$.
    As this holds for $\H^1$-a.e.~point $y_0 \in \Sigma$, we deduce that $\nu_\delta \equiv 0$.
\end{proof}

The previous lemma, combined with the caracterization of solutions, as in~\eqref{nuexc.further_decomposition},
\[
    \nu = \alpha^{-1}\H^1\mres \Sigma + \sup_{\delta > 0}\nu_\delta + \rhoexc\mres \Sigma
\]
proves the following result, showing in particular point (2) of Theorem~\ref{theorem.the_big_one}.

\begin{theorem}\label{theorem.inR2_density_is_1}
    Let $\rho_0 \in \mathcal{P}_p(\RR^d)$ and suppose that the parameter $\Lambda < \Lambda_\star$. Then the solution to the relaxed problem~\eqref{problem.shape_optimization_relaxed} is of the form 
        \[
            \nu = 
            \mathcal{L}(\nu)^{-1}\H^1\mres \Sigma    
            + 
            \rhoexc\mres \Sigma,
        \]
        where $\rhoexc$ was defined in \eqref{rhoexceed}.

        In addition, if $\rho_0$ does not give mass to $1$-rectifiable sets, any solution of the relaxed problem~\eqref{problem.shape_optimization_relaxed} corresponds to a solution of the original shape optimization problem \eqref{problem.shape_optimization}.
\end{theorem}
% {\color{green}
% \begin{remark}
%     In the characterization of solutions given by 
%     \[
%         \nu = 
%         \alpha^{-1}\H^1\mres \Sigma    
%         + 
%         \rhoexc\mres \Sigma,
%     \]
%     the last term is reminiscent of Lemma~\ref{lemma.minimal_distance2Sigma}, that says that the excess measure $\nuexc$ is formed through projections. Indeed, rewriting it as
%     \[
%         \nuexc = \sup_{\delta > 0} \nu_\delta + \rhoexc \mres \Sigma,
%     \]    
%     as in equation~\eqref{nuexc.further_decomposition}, we have shown that in Lemma~\ref{lemma.better_competitor} that the components $\nu_\delta$ coming from a distance $\delta$ to $\Sigma$ are in fact null. 
% \end{remark}}
%%%%%%%%%%%%%%%%%%%
%%%%%%%%%%%%%%%%%%%
%%%%%%%%%%%%%%%%%%%

\section{Ahlfors regularity}\label{section.ahlfors_regularity}
In this section we prove that whenever the initial measure $\rho_0 \in L^{\frac{d}{d-1}}(\RR^d)$, the optimal solutions to the relaxed problem \eqref{problem.shape_optimization_relaxed} have an Ahlfors regular support. 

\begin{definition}\label{definition.alhfords_regularity}
  We say that a set $\Sigma \subset \RR^d$ is {\em Ahlfors regular} whenever there exist $r_0>0$ and $c,C>0$ such that for $r \le r_0$ it holds that
  \[
    cr \le \H^1(\Sigma \cap B_r(x)) \le Cr, \text{ for all $x \in \Sigma$.}
  \]
\end{definition}

We prove in this section the following result.
\begin{theorem}\label{th:ahlfors}
  If $\rho_0\in L^{\frac{d}{d-1}}(\RR^d)$, let $\nu$ be a solution of the relaxed problem \eqref{problem.shape_optimization_relaxed} and $\Sigma$ its support. Then $\Sigma$ is Ahlfors-regular: there exist $\bar r_0>0$ and $\bar C>0$ such that, for all $\bar x \in \Sigma$ and $r \le \bar r_0$,
  \[
    r \le \H^1(\Sigma \cap B_r(\bar x)) \le \bar C r. 
  \]
    Moreover, $\bar r_0$ depends only on $d,p,\rho_0$ and $\alpha \eqdef \mathcal{L}(\nu)$, while $\bar C$ depends only on $d$ and $p$.
\end{theorem}

The lower bound (with $c=1$ and $r_0=\diam\Sigma$) follows directly from the connectedness of $\Sigma$. The upper bound will follow as a corollary of Lemma~\ref{lemma.ahlfors} below. Let us describe  the strategy for proving this estimate. We point out that the construction in this section,
although  different, follows similar steps as
the proof of Ahlfors' regularity in~\cite[Lem.~6.1, Thm.~6.4]{paolini2004qualitative}.

The idea is similar to proving the $L^\infty$ bound on the excess measure: if in a small ball
$B_r(\bar x)$ the measure $\nu$ has too much mass,
we build another ``closer'' 1D structure onto which the mass is transfered at a smaller cost. Yet there is an additional difficulty: when replacing $\Sigma \cap B_r(\bar x)$ with
another set we must preserve the connectedness.
The proof of Theorem~\ref{theorem.solutions_are_absolutely_continuous}, required to rearrange only the excess mass
and this was not an issue.
We now need to control the number of connected components
of $\Sigma \setminus B_r(\bar x)$ and connect them back without adding
too much length.

This number of connected components is controlled by the
quantity $\H^0(\Sigma\cap\partial B_r(\bar x))$, which we can control on average
by means of the {\em generalized area formula}~\cite[Theorem 2.91]{ambrosio2000functions}: If $f:\RR^M \to \RR^N$ is a Lipschitz function and $E\subset \RR^M$ is a $k$-rectifiable set then it holds that
\begin{equation}
  \label{equation.area_formula_generalized}
  \int_{\RR^N} \H^0(E \cap f^{-1}(y))\dd\H^k(y)
  =
  \int_{E} J_k\dd^Ef_x \dd\H^k(x),
\end{equation}
where $\dd^Ef_x$ is the restriction of $\nabla f(x)$ (when $f$ is smooth) to the approximate tangent space of $E$. 
Hence, choosing $E = \Sigma \cap (B_{r_1}(\bar x) \setminus B_{r_2}(\bar x))$ and $f : x \mapsto |x - \bar x|$, %taking a unitary vector $\tau_x \in T_x\Sigma$, the area factor becomes $J_1\dd^{\Sigma} f_x = \disp \frac{|(x - \bar x)\cdot \tau_x|}{|x - \bar x|}$ and 
we deduce from~\eqref{equation.area_formula_generalized} that
\begin{equation}\label{equation.area_formula_generalized_inequality}
  \int_{r_2}^{r_1} \H^0(\Sigma \cap \partial B_s(\bar x))\dd s
  \le 
  \H^1(\Sigma \cap B_{r_1}(\bar x)) - \H^1(\Sigma \cap B_{r_2}(\bar x))
\end{equation}

Using this we first prove the following lemma:
\begin{lemma}\label{lemma.ahlfors}
  Assume $\rho_0\in L^{\frac{d}{d-1}}(\RR^d)$.
  There exist $\bar C(d,p)>0$ and $r_0$ depending on $\rho_0$, $\alpha$, $d$, $p$, such
  that for any $C\ge \bar C$, if $r\le r_0$ and $x\in \Sigma$,
  then either $\H^1(\Sigma\cap B_r(x))\le Cr$ or $\H^1(\Sigma\cap B_{2r}(x))\ge 10Cr$.
\end{lemma}
\begin{proof}
  Let $r> 0$ and $C\ge 1$, and let $\bar x\in\Sigma$ such that
  \begin{equation}\label{eq.bounds_length}
        \H^1(\Sigma\cap B_r(\bar x))> Cr
        \text{ and }
        \H^1(\Sigma\cap B_{2r}(\bar x))< 10Cr.
  \end{equation}
  We show that if $r\le r_0$ and $C\ge \bar C$, which will both be chosen later, then we can contruct a better competitor to the minimizer $\nu$.

  The function $f: s \mapsto \H^1(\Sigma\cap B_s(\bar x))$ is nondecreasing, hence
  in $BV(\RR_+)$ and satisfies, thanks to~\eqref{equation.area_formula_generalized_inequality}, 
  % \[
  %   \int_{t_1}^{t_2} \H^0(\Sigma\cap \partial B_s(\bar x))ds 
  %   \le \H^1(\Sigma\cap B_{t_1}(\bar x)) - \H^1(\Sigma\cap B_{t_2}(\bar x))
  %   =
  %   Df((t_2,t_1]),
  % \]
  % for any $t_1<t_2$ so
  that  $\H^0(\Sigma\cap \partial B_s(\bar x))\dd s\le Df$ in the sense of measures
  (equivalently, $\H^0(\Sigma\cap\partial B_s(\bar x))$ is less than, or equal to $f'(s)\dd s$, the absolutely continuous part of $Df$).

We note that 
\begin{align*}
	\inf_{s \in (3r/2,2r)} \left(\frac{s\H^0\left(\Sigma\cap\partial B_s(\bar x)\right)}{\H^1(\Sigma\cap B_{s}(\bar x))}\right)
&\le \frac{2}{r}\int_{3r/2}^{2r}\frac{s\H^0\left(\Sigma\cap\partial B_s(\bar x)\right)}{\H^1(\Sigma\cap B_{ s}(\bar x))} \dd s\\
&\le  4\int_{3r/2}^{2r} \frac{1}{f(s)} f'(s)\dd s \\
&\leq 4\ln \left(\frac{f(2r)}{f(3r/2)}\right),
\end{align*}
where we have used the classical chain rule at almost every point and~\cite[Cor. 3.29]{ambrosio2000functions}. Since $f(2r)/f(3r/2)<(10Cr)/(Cr)=10$, we deduce that there exists $\bar s\in (3r/2,2r)$ such that 
  \begin{equation}\label{eq:estimbars}
	\bar\delta \bar s\H^0(\Sigma\cap\partial B_{\bar s}(\bar x)) \le \H^1(\Sigma\cap B_{\bar s}(\bar x)) \quad \mbox{ where $\bar \delta \eqdef \frac{1}{4\ln 10}.$}
  \end{equation}
  Now, we let 
  \begin{equation}\label{eq:M}
    M= 2\left( 1+10\cdot  \left(\frac{40}{17}\right)^{p-1}\right)
  \end{equation}
  (this choice will be made clear at the end of this proof) and we consider
  \begin{equation}\label{eq:delta}
    \delta \eqdef \frac{\bar\delta}{10M} <\bar \delta<\frac{1}{2}.
  \end{equation}

  We define a set $\Gamma$ as follows: we choose a finite covering of
  $\partial B_1(0)$ with balls $B(x_i,\delta/2)$ centered
  at points $(x_i)_{i=1}^N$ (the minimal number $N$ depends only on $d$ and $p$, through $\delta$).
  Then, we find a minimal tree connecting the points $(x_i)_{i=1}^N$ through geodesics on the sphere. We add to this minimal tree the segments $[x_i,(1+\delta)x_i]$, $i=1,\dots,N$. We call $\Gamma$ the resulting (connected) set,
  whose total length $L\eqdef \H^1(\Gamma)$ is of order at most $2N\delta$ and depends
  only on $d$ and $p$.
  Notice that each point of $\partial B_1$ is at distance at most $\delta$,
  along the geodesic curve on the sphere, to a point of $\Gamma$, and that thanks to
  the ``spikes'' $[x_i,(1+\delta) x_i]$, any point with, say, $|x|\ge 10$ is closer to
  a point of $\Gamma$ than from any point in $B_1(0)$.

  Now, we define
  \[
    \Gamma_{\bar s} \eqdef (\bar x + \bar s \Gamma) \cup \bigcup_{x \in \Sigma \cap \partial B_{\bar s}} S_x, 
  \]
  where $S_x$ denotes a geodesic connecting $x$ to $\bar x + \bar s \Gamma$, of
  length at most $\H^1(S_x)\le \bar s\delta$.
Since $\bar s<2r$ and $\delta<1/2$, it follows that $\Gamma_{\bar s} \subset B_{3r}(\bar x)$. We define the competitor set as
  \[
    \Sigma' \eqdef \Sigma \setminus B_{\bar s}(\bar x) \cup \Gamma_{\bar s}. 
  \]
  The addition of the geodesics $S_x$ ensures that $\Sigma'$ remains connected, and using~\eqref{eq:estimbars}, we estimate the length of $\Gamma_{\bar s}$ as
  \begin{equation}\label{eq:lengthGamma}
    \begin{aligned}
      \H^1(\Gamma_{\bar s}) 
      &\le L\bar s + \delta\bar s\H^0(\Sigma\cap\partial B_{\bar s}(\bar x))
      \le 2Lr + \tfrac{1}{10M}\H^1(\Sigma\cap B_{\bar s}(\bar x))\\
      &< (2L + \tfrac{C}{M})r,
    \end{aligned}
  \end{equation}
 where we have used~\eqref{eq.bounds_length} in the last estimate.
  Now we define a new competitor $\nu'$ whose support is $\Sigma'$. If $\gamma$ denotes an optimal transport plan from $\rho_0$ to $\nu$, given $s >0$ let
  \[
	  \rho_{s} \eqdef {\pi_{0}}_\sharp\left(\gamma \mres \left(\RR^d\times B_{s}\right)\right) 
  \]
  denote the portion of the measure $\rho_0$ which is transported to the ball $B_{s}$. In particular,  the above length estimates imply  that 
  \begin{equation}\label{eq:Lr0}
    Lr 
    \le 
    \H^1(\Gamma_{\bar s}) 
    < 
    (2L+\tfrac{C}{M})r 
    \le 
    (2\tfrac{L}{C}+\tfrac{1}{M})\alpha\nu(B_r) 
    \le
    \alpha \rho_r(\RR^d)\leq \alpha \rho_{\bar s}(\RR^d),
  \end{equation}
  where $\alpha\eqdef\L(\nu)$, and using that $M\ge 2$ (see~\eqref{eq:M}) and assuming $\bar C\ge 4L$ (which we recall depends
  only on $d$ and $p$).
  But, if $r$ is small enough (not depending on $\bar x$, by uniform equi-integrability
  of $\rho_0^{d/(d-1)}$) H\"older's inequality implies that
  \begin{equation}\label{eq:Lr}
    \alpha \rho_{\bar s}(B_{10r}(\bar x)) 
    \le 
    \alpha\|\rho_0\|_{L^{\frac{d}{d-1}}(B_{10r}(\bar x))}
    |B_{10r}(\bar x)|^{\frac{1}{d}} < Lr.
  \end{equation}
  We fix $r_0>0$, which depends only on the dimension (through $L$), the integrability of $\rho_0$, and $\alpha$, such that the above inequality holds for $r\le r_0$. 
  
  Equations~\eqref{eq:Lr0}-\eqref{eq:Lr} show that for $r$ small enough, part of the mass transported to $\nu \mres B_{\bar s}$ must come from outside of the ball $B_{10r}$. In particular, since $t \mapsto \rho_{\bar s}(B_t(\bar x))$ is continuous, there is $R > 10r$ such that 
  \begin{equation}\label{equation.mass_inNout_BR}
    \begin{aligned}
      & \rho_{\bar s}(B_R(\bar x)) = \alpha^{-1}\H^1(\Gamma_{\bar s}).
    \end{aligned}
  \end{equation}

  To form the new competitor we proceed as follows: the mass sent to $\Sigma\setminus B_{\bar s}$ remains untouched, the mass  $\rho_{\bar s}\mres B_R$ previously used to form $\nu\mres B_{\bar s}$ is transported to $\alpha^{-1}\H^1\mres \Gamma_{\bar s}$ and the remaining mass is projected onto $\Gamma_{\bar s}$. 
  
  So, letting $\tilde \gamma$ an optimal transport plan between $\rho_{\bar s}\mres B_R$ and $\alpha^{-1}\H^1\mres \Gamma_{\bar s}$, we define the  plan
  \[
    \gamma' = 
    \gamma\mres \RR^d\times B_{\bar s}(\bar x)^c 
    +
    \tilde \gamma\mres B_R\times\RR^d 
    +
    (\text{id}, \pi_{\Gamma_{\bar s}})_\sharp \left(\rho_{\bar s}\mres B_R^c\right),
  \]
  and the new competitor $\nu'$ as its second marginal. By construction, $\alpha\nu' \ge \H^1\mres \Sigma'$ so that $\mathcal{L}(\nu') \le \mathcal{L}(\nu)$. We now estimate the gain in terms of transportation cost. 
  \begin{itemize}
  \item For $(x,y)\in  B_R\times B_{\bar s}$ and for any $y' \in \Gamma_{\bar s}\subset B_{3r}$, as $\bar s\le 2r$ and $10r < R$, the convexity of $t \mapsto t^p$ yields 
    \begin{align*}
      |x - y'|^p
      &\le
        \left(|x - y| + 5r\right)^p
        \le 
        |x - y|^p 
        + 
        5rp
        \left(|x - y|+5r\right)^{p-1} \\ 
      &\le 
        |x - y|^p 
        + 
        5rp {{\left(2R\right)^{p-1}}}.
        % \left(\frac{5R}{2}\right)^{p-1}. seems wrong to me
    \end{align*}
    Hence integrating w.r.t.~the transport plans we get
    \[
      \int_{B_R\times\Gamma_{\bar s}}|x - y'|^p \dd \tilde \gamma 
      \le 
      \int_{B_R\times B_{\bar s}}|x - y|^p \dd \gamma
      +
      5rp
      \left({2R}\right)^{p-1}\rho_{\bar s}\left(B_R\right),
    \]
    (this can be checked by disintegration w.r.t.~their common first marginal, which is the measure $\rho_{\bar s}\mres B_R$).
  \item 
    Similarly, for $x \in B_R^c$ and $y \in B_{\bar s}\setminus B_r$ the addition of the spikes ensures that
    \[
      |x - \pi_{\Gamma_{\bar s}}(x)| \le |x - y|.
    \]
    However if $x \in B_R^c$ and $y \in B_r$ it holds that 
    \[
      |x - \pi_{\Gamma_{\bar s}}(x)| \le |x - y| - \frac{r}{2} 
      \text{ and }
      |x - y| \ge R - r,
    \]
    so that once again using the convexity of $t\mapsto t^p$ we have
    \begin{align*}
      |x - \pi_{\Gamma_{\bar s}}(x)|^p
      &\le
        \left(|x - y| -\frac{r}{2}\right)^p
        \le 
        |x - y|^p 
        -
        p\frac{r}{2}
        \left(|x - y|-\frac{r}{2}\right)^{p-1}\\ 
      &\le 
        |x - y|^p 
        -
        p\frac{r}{2}
        \left(\frac{17}{20}R\right)^{p-1}.
    \end{align*}
    So, decomposing the integration for the points going to $B_r$ and to $B_{\bar s}\setminus B_r$, this time the transportation cost can be bound by:
    \begin{align*}
      \int_{B_R^c}
      |x - \pi_{\Gamma_{\bar s}}(x)|^p \dd \rho_{\bar s}
      &=
        \int_{B_R^c}
        |x - \pi_{\Gamma_{\bar s}}(x)|^p \dd (\rho_{\bar s} -\rho_{r})
        +
        \int_{B_R^c}
        |x - \pi_{\Gamma_{\bar s}}(x)|^p \dd \rho_{r}\\
      &\le 
        \int_{B_R^c\times B_{\bar r}}|x - y|^p \dd \gamma
        -
        p\frac{r}{2}
        \left(\frac{17}{20}R\right)^{p-1}\rho_r\left(B_R^c\right). 
    \end{align*}
  \end{itemize}  

  We get:
  \begin{align*}
    W_p^p(\rho_0, \nu')
    &\le
      W_p^p(\rho_0, \nu)
      +
      5rp
      \left(2R\right)^{p-1}\rho_{\bar s}\left(B_R\right)
      - 
      p\frac{r}{2}
      \left(\frac{17}{20}R\right)^{p-1}\rho_r\left(B_R^c\right).
  \end{align*}
  As $\mathcal{L}(\nu') \le \mathcal{L}(\nu)$, the optimality of $\nu$ gives that $W_p^p(\rho_0, \nu) \le W_p^p(\rho_0, \nu')$, which, along with the previous estimates, implies
  \[
    0 \le 
    5\cdot
    2^{p-1}\rho_{\bar s}\left(B_R\right)
    - 
    \frac{1}{2}
    \left(\frac{17}{20}\right)^{p-1}\rho_r\left(B_R^c\right)
    \ \Leftrightarrow \ 
    \rho_r\left(B_R^c\right)\le 10\cdot  \left(\frac{40}{17}\right)^{p-1}
    \rho_{\bar s}\left(B_R\right).
  \]
On the other hand, since  
 \begin{align*}
	 \rho_{r}\left(B_R(\bar x)^c\right) = \nu(B_r(\bar x))- \rho_r(B_R(\bar x)) \ge \alpha^{-1}Cr - \rho_r(B_R(\bar x))
      \ge \alpha^{-1}Cr - \rho_{\bar s}(B_R(\bar x)),
      % \ge \alpha^{-1}\left(Cr -\H^1(\Gamma_{\bar s})\right).
 \end{align*} 
 and recalling~\eqref{eq:lengthGamma} and~\eqref{equation.mass_inNout_BR}, we deduce:
  \[
    C
    \le 	\left( 1+10\cdot  \left(\frac{40}{17}\right)^{p-1}\right)(2L+\tfrac{C}{M})
  \]
  % \[
  %   \frac{C-L}{\alpha}r
  %   \le 
  %   \rho_r\left(B_R^c\right)
  %   \le 
  %   \frac{c_p}{\alpha}\H^1(\Gamma_{\bar s})
  %   \le 
  %   \frac{c_p}{\alpha}\left(2L + \frac{C}{M}\right)r, \text{ where $c_p = 10\cdot\left(\frac{50}{17}\right)^{p-1}$.}
  % \]
  We conclude that with the choice~\eqref{eq:M} of $M$,
  one has $C\le 2M L$, which depends only on $p$ and $d$ and a contradiction follows
  if we choose $\bar C=1+2M L$.
\end{proof}

% The surprising consequence of the previous Lemma is that, although the radius $r_0$ depends on $\rho_0$ and $\alpha$, the constant $\bar C$ depends only on $p$ and $d$. This means that to arrive at the regularity regime we must go to a scale that depends linearly on $\alpha$
% \revision{WHY?}, so in $\RR^2$ depending linearly on the total length of $\Sigma$. Hence the more concentrated $\rho_0$ might be, and the longer we allow $\Sigma$ to become, through the dependence on $\Lambda$, the smaller the scale necessary to observe Ahlfors-regularity becomes. However, once we have arrived at this regularity regime, the regularity we observe does not depend on either $\rho_0$ or $\Sigma$. 

\begin{proof}[Proof of Theorem~\ref{th:ahlfors}]
  Consider $\bar C$, $r_0$ from Lemma~\ref{lemma.ahlfors}. Fix $x\in \Sigma$ and assume
  there is $r\in (0,r_0)$ such that $\H^1(\Sigma\cap B_r(x)) > \bar Cr$. Then the thesis of the lemma applies and it must hold that $\H^1(\Sigma\cap B_{2r}(x)) \ge 10\bar Cr$. By induction, we find that for $k\ge 1$, one of the following holds: %the two situations are mutually excluse:
  \begin{itemize}
  \item either $2^k r > r_0$;
  \item or we apply the lemma again (with $C'=5^k\bar C$ and $r'=2^kr$), using that $\H^1(\Sigma\cap B_{2^{k}r}(x)) > 5^{k}\bar C(2^k r)$, and we get
    \[
      \H^1(\Sigma\cap B_{2^{k+1}r}(x)) > 5^{k+1}\bar C(2^{k+1}r).   
    \]
  \end{itemize}

  Let $k\ge 1$ be the first integer such that $2^k r> r_0$, so that $2^{k-1}r\le r_0$
  and
  \[
    5^{k}\bar C(2^{k}r) < \H^1(\Sigma\cap B_{2^{k}r}(x)).
  \]
  Hence, $r_0 < 2^kr \le 5^{-k}\bar C^{-1}\H^1(\Sigma)$ and it holds that $k\le k_0\eqdef \log_5 (\H^1(\Sigma)/\bar C r_0)$,
  and 
  \[
    r > r_0 2^{-k} \ge \bar r_0 \eqdef r_0 \cdot 2^{-k_0}.
  \]
  This shows that, if $r \le r_0$ is such that $\H^1(\Sigma\cap B_r(x)) > \bar Cr$, then $r>\bar r_0$. As a result, for every $r\le \bar r_0$ and every $x\in\Sigma$, we have $\H^1(\Sigma\cap B_r(x))\le \bar Cr$. 
\end{proof}
\begin{remark} It is interesting to observe here that the regularity constant
  $\bar C$ depends only on $d$ and $p$, while the scale $\bar r_0$ at which the Ahlfors-regularity
  holds gets smaller as $\rho_0$ gets more singular or when $\alpha$ (or $\H^1(\Sigma)$) increases (which is when $\Lambda$ decreases).
\end{remark}

\section{Conclusion}\label{section.conclusion}
In this paper we have proposed a new variational problem, which serves as a method for approximating a probability measure with a measure uniformly distributed over a one-dimensional continuum. 
In order to prove existence, we have passed through a relaxed problem and the definition of a new functional on the space of probability measures, the length functional, that generalizes the notion of length of the support of a measure. As a tool for our analysis we have also generalized \Golab's Theorem to the case of a sequence of possibly unbounded sets converging in the Kuratowski sense. We then have shown that solutions of the relaxed
problems are, in fact, solutions to the original one whenever the original measure does not give mass to 1-rectifiable sets of $\RR^d$. 
%We have also proved a few properties of the solutions of the relaxed problem in any dimension, {\em e.g.} $L^\infty$ bounds and Ahlfors regularity. 
We also have proved an elementary regularity properties of the 
optimal sets, in the form of an Ahlfors regularity estimate.

There are still many open questions left, such as:
\begin{itemize}
    \item Does the support of minimizers have loops or are they trees?
    \item What is the regularity of the optimal $\Sigma$? Can we adapt the theory in \cite{morgan1994m} and conclude they are locally $C^{1,\alpha}$ curves? 
    \item If $\nu_\Lambda$ is a solution to \eqref{problem.shape_optimization_relaxed}, what is the rate of convergence of $\nu_\Lambda \xrightharpoonup[\Lambda \to 0]{\star} \rho_0$? 
    \item The blow-up analysis in Section \ref{section.blowup} is very similar to the arguments  in \cite{santambrogio2005blow} for the blow-up of average distance minimizers. However, the argument is applied to the excess measure and not to the entire solution. Can we use similar tools to study the blow-ups of the optimal networks in our problem as well? 
    \item What are the Euler-Lagrange equations of \eqref{problem.shape_optimization_relaxed}?
    \item Could we find (efficient) numerical algorithms to solve this problem? 
\end{itemize}
Some progress has been made on a few of these questions: for instance in~\cite{MachadoThesis},
it is proven in a simplified setting (when $\rho_0$ is a finite sum
of Dirac masses) that the solution is supported on a tree; in~\cite{MachadoPhaseField},
a phase-field approach is suggested to approximate Problem~\eqref{problem.shape_optimization}, which could
lead to (still complicated) numerical methods and simulations.

%%%%%%%%%%
%%%%
%%%%%%%%%

\appendix

\section{Localized variational problem}\label{appendix.localized_variational_problem}
In this section, we prove Lemma \ref{lemma.auxiliary_problem}, which states that the optimality of $\nu$ implies that the exceeding measure $\nuexc$, or a slight modification of it, must satisfy a localized optimization problem. Before proceeding we review the notation introduced in the statement of the Lemma. Given an optimal transportation plan $\gamma$ between $\rho_0$ and the minimizer $\nu$, we recall the definition of $\gammaexc$ in \eqref{gammaexceed} and we fix a general Borel set $\mathcal{S} = \mathcal{S}_0 \times \mathcal{S}_1$ to define
\[
    \gamma_{\mathcal{S}} \eqdef \gammaexc \mres \mathcal{S}_0\times\mathcal{S}_1
\]
along with its marginals 
\[
    \rho_{\mathcal{S}} \eqdef {\pi_{0}}_\sharp\gamma_\mathcal{S} 
    %= \rhoexc \mres \mathcal{S}_0, % THIS IS NOT CORRECT
    \quad \nu_{\mathcal{S}} \eqdef {\pi_{1}}_\sharp\gamma_{\mathcal{S}},
\]

%We can interpret the choice of $\mathcal{S}$ as a form of generating admissible variations for our problem in order to extract necessary conditions for optimality. 

\begin{proof}[{\bf Proof of Lemma \ref{lemma.auxiliary_problem}:}]
	%The proof consists in generating variations of a particular form. 
	First, we fix some arbitrary $\Gamma$ such that $\Sigma \cup \Gamma \in \mathcal{A}$. We consider measures $\nu' \in \mathcal{M}_+(\Sigma\cup \Gamma)$ such that  $\nu'(\RR^d) = \nu_{\mathcal{S}}(\RR^d)$  and $\nu' \ge \alpha^{-1}\mathcal{H}^1\mres (\Gamma\setminus \Sigma)$, and we build competitors to $\nu$ of the form $\nu - \nu_{\mathcal{S}} + \nu'$. Such measures are supported over $\Sigma \cup \Gamma \in \mathcal{A}$ and 
    \begin{align*}
        \nu - \nu_{\mathcal{S}} + \nu' 
        &=
        \nuH + (\nuexc - \nu_{\mathcal{S}}) + \nu'\\ 
        &\ge 
        \alpha^{-1}\H^1\mres \Sigma + \alpha^{-1}\H^1\mres (\Gamma\setminus \Sigma)
        \ge  \alpha^{-1}\H^1\mres (\Sigma \cup \Gamma),
    \end{align*}
    so that $\L(\nu - \nu_{\mathcal{S}} + \nu') \le \alpha =\L(\nu)$. By optimality of $\nu$, we deduce that
    \begin{align*}
    W_p^p(\rho_0, \nu) \le W_p^p(\rho_0, \nu - \nu_{\mathcal{S}} + \nu').
    \end{align*}

    Given any transport plan $\gamma'$ from 
    $\rho_{\mathcal{S}}$ to $\nu'$, $\gamma-\gamma_{\mathcal{S}}+\gamma'$
    is a transport plan from $\rho_0$ to $\nu - \nu_{\mathcal{S}} + \nu'$
    and it follows, from the optimality of $\nu$ and $\gamma$:
    \[
    \int |x-y|^p d(\gamma-\gamma_{\mathcal S}) + \int |x-y|^p d\gamma_{\mathcal S} = \int |x-y|^p d\gamma
    \le 
    \int |x-y|^p d(\gamma-\gamma_{\mathcal S}) + \int |x-y|^p d\gamma',
    \]
    so that:
    \begin{equation}\label{eq:optlocalized}
    \int |x-y|^p d\gamma_{\mathcal S}\le \int |x-y|^p d\gamma'.
    \end{equation}
    Observe that in case $\nu'=\nu_{\mathcal{S}}$ (and $\Gamma=\emptyset$), 
    we find that
    $\gamma_{\mathcal{S}}$ is an optimal plan. 
    In particular the left-hand side of this equation 
    is $W_p^p(\rho_{\mathcal{S}}, \nu_{\mathcal{S}})$.

Since the same argument applies to $\gamma-\gamma_{\mathcal{S}}$, we observe
that:
\begin{equation}\label{rhoexc_decomposition_finer}
    W_p^p(\rho_0, \nu) 
    = 
    W_p^p\left(\rho_0 - \rho_{\mathcal{S}}, \nu -\nu_{\mathcal{S}}\right)
    + W_p^p\left(\rho_{\mathcal{S}},\nu_{\mathcal{S}}\right).
\end{equation}

Considering an optimal transport plan $\gamma'$ in
\eqref{eq:optlocalized}, we get in addition that 
        $W_p^p\left(\rho_{\mathcal{S}}, \nu_{\mathcal{S}}\right) \le W_p^p\left(\rho_{\mathcal{S}}, \nu'\right)$  
    for all the admissible variations $\nu'$ of the excess measure.

    As $\gamma_{\mathcal{S}}$ is an optimal transportation plan between $\rho_{\mathcal{S}}$ and $\nu_{\mathcal{S}}$, from \cite[Theorem 5.27]{santambrogio2015optimal} one can define a constant speed geodesic between such measures as 
    \begin{align*}
        \sigma_{\mathcal{S},t} \eqdef {\pi_{(1-t)}}_\sharp\gamma_{\mathcal{S}}, \text{ where } \pi_t(x,y) \eqdef (1-t)x + ty.
    \end{align*}
    
    Hence for any variation $\nu'$, admissible in the sense of the previous problem, and for any $t \in [0,1]$, it holds that
    \begin{align*}
        W_p\left(\rho_{\mathcal{S}}, \sigma_{\mathcal{S},t}\right) + W_p\left(\sigma_{\mathcal{S},t}, \nu_{\mathcal{S}}\right)
        &=  
        W_p\left(\rho_{\mathcal{S}}, \nu_{\mathcal{S}}\right)
        \le W_p\left(\rho_{\mathcal{S}}, \nu'\right)\\ 
        &\le W_p\left(\rho_{\mathcal{S}}, \sigma_{\mathcal{S},t}\right) + W_p\left(\sigma_{\mathcal{S},t}, \nu'\right).
    \end{align*}
    Where the equality comes from general properties of constant speed geodesics in metric spaces, while the inequalities come from the minimality of $\nu_{\mathcal{S}}$ and the triangle inequality, respectively. We conclude that in fact, the measure $\nu_{\mathcal{S}}$ also minimizes the Wasserstein distance to any measure $\sigma_{\mathcal{S},t}$ along the geodesic.
\end{proof}

\section{Kuratowski convergence and Golab's Theorem}\label{appendix.kuratowski_convergence}
In this appendix we give a proof of Lemma \ref{lemma.localized_Kuratowski}.
We then give a simple proof
of the local version of \Golab's, Theorem~\ref{theorem.Golab_localversion}.
\if{
\begin{lemma}
    Let $(C_n)_{n \in \mathbb{N}}$ be a sequence of closed sets in $\RR^d$, convering to $C$ in the sense of Kuratowski.
    Then, for any $x_0 \in \RR^d$, there exists a countable set $I\subset [0,+\infty)$ such
    that 
    \[
      C_n\cap\overline{B_R(x_0)} 
      \xrightarrow[n \to \infty]{d_H} 
      C\cap\overline{B_R(x_0)}, 
      \text{ for all $R\in [0,+\infty)\setminus I$.}
    \]
  \end{lemma}
}\fi
We use the notation $B_R=\{x: |x|<R\}$ and $\oB_R = \{x:|x|\le R\}$.

\begin{proof}[Proof of Lemma~\ref{lemma.localized_Kuratowski}]
  Notice that, up to a translation, it suffices to prove the result for $x_0 = 0$. We can also assume that  $C\neq\emptyset$, otherwise
  for any $R>0$, $C_n\cap B_R=\emptyset$ for $n$ large enough and the result holds.
  Defining $R_0=\inf\{R>0: C\cap \oB_R\neq\emptyset\}$, we have that if $R<R_0$,
  one has $C_n\cap \oB_{R}=\emptyset$ for $n$ large enough
  and the Hausdorff limit is empty, as expected.
 
  Now we take $R\ge R_0$ and consider a subsequence $\left(C_{n_k}\right)_{k \in \mathbb{N}}$ and a closed set $C^R$ such that
  \[
    C_{n_k}\cap \oB_R \xrightarrow[n \to \infty]{d_H} C^R.
  \]
  Since $C_{n_k}\cap\oB_R\subset C_{n_k}$, it holds that $C^R\subset C$.
  On the other hand, given $x\in C \cap B_R$, if there exists $x_n\in C_n \cap \oB_R$
  with $x_n\to x$, then $x\in C^R$.
  Therefore
  \[
    \overline{C\cap B_R}\subset C^R\subset C\cap\oB_R
  \]
  and to finish the proof it suffices to show that there is a countable set $I\subset [R_0,+\infty)$ such that
  if $R\not\in I$, $R>R_0$, then $C\cap\oB_R =
  \overline{C\cap B_R}$.
  
  Let $\xi\in \partial B_1$ and consider the function
  $R\mapsto \dist(R\xi,C\cap \oB_R)$. If $R>R'\ge R_0$ it holds that
  \[
    \dist(R\xi,C\cap \oB_R)
    \le \dist(R'\xi,C\cap \oB_{R'})+R-R'.
  \]
  Indeed, let $x_{R'}$ be the point minimizing the distance from $R'\xi$ to $C\cap \oB_{R'}$, then 
  \begin{align*}
      \dist(R\xi, C\cap \oB_R)
      &\le
      d(R\xi, x_{R'})
      \le 
      d(R\xi, R'\xi) + d(R'\xi, x_{R'}) \\  
      &= 
      \dist(R'\xi, C\cap \oB_{R'}) + R - R'.
  \end{align*}
  
  Hence the function 
  $
    \varphi_\xi : R \mapsto \dist(R\xi,C\cap \oB_R)-R,
  $
  is nonincreasing in $[R_0,+\infty)$ and in particular it has at
  most a countable number of discontinuity points. In addition, given $\xi,\xi'\in\partial B_1$, it holds that 
  \begin{align*}
      \left|
        \varphi_\xi(R) - \varphi_{\xi'}(R)
      \right|
      &= 
      \left|
        \inf_{x \in \oB_R} d(x, R\xi) - \inf_{x \in \oB_R} d(x, R\xi')
      \right|\\
      &\le
      \sup_{x \in \oB_R}
      \left|
        d(x, R\xi) - d(x, R\xi')
      \right|
      \le R|\xi - \xi'|.
  \end{align*}
  Therefore if $R$ is a point of discontinuity for $\varphi_\xi$, then
  for all $\xi'$ in a neighborhood of $\xi$, $R$ is a point
  of discontinuity for $\varphi_{\xi'}$.

  Let $(\xi_n)_{n \in \mathbb{N}}$ be a dense sequence in $\partial B_1$.
  For each $n$ we can find a countable subset $I_n\subset [R_0,+\infty)$, such that $\varphi_{\xi_n}$
  is continuous at any $R\in (R_0,+\infty)\setminus I_n$.
  Finally, we define the countable set $I$ as 
  $
    \disp I= \bigcup_{n\in \mathbb{N}} I_n. 
  $

  If $R\not\in I$, then either $R<R_0$ and
  $C\cap\oB_R=\overline{C\cap B_R}=\emptyset$, or
  $R>R_0$. In that case, for any $\xi\in\partial B_1$,
  $\varphi_\xi$ is continuous. Otherwise, there would be some $\xi_n$,
  close enough to $\xi$, such that $\varphi_{\xi_n}$ is discontinuous,
  a contradiction. In particular, whenever $x = R\xi \in C$ the continuity of $\varphi_\xi$ implies that
  \[
    \lim_{R'\uparrow R}\dist(R'\xi,C\cap \oB_{R'})=0.
  \]
  Hence take $R_n \uparrow R$, set $\varepsilon_n \eqdef \dist(R_n\xi,C\cap \oB_{R_n})$ and let $x_n \in C\cap \oB_{R_n}$ be a vector attaining this distance. As $x_n \in C \cap B_R$ and $|x - x_n| \le \varepsilon_n + R - R_n$, $x_n$ converges to $x$, and $x\in \overline{C\cap B_R}$. It follows that $(C\cap\oB_R)\setminus
  \overline{C\cap B_R}=\emptyset$, completing the proof.
\end{proof}

\begin{proof}[Proof of Theorem~\ref{theorem.Golab_localversion}]
	We will show that $\mu(\Sigma \cap B_r(y_0)) \ge \H^1(\Sigma \cap B_r(y_0))$ for $\H^1$-a.e. $y_0\in \Sigma$ and for $r>0$ small enough. This implies that $\Theta_1(\mu, y_0) \ge 1$, and the result follows by integrating. Assume that $\Sigma$ is not a singleton, otherwise there is nothing to prove. Since a compact and connected set with finite length is path-wise connected, see~\cite[Prop.~30.1 and Cor.~30.2]{david2006singular} and~\cite[Thm.~4.4]{alberti2017structure}, for any $y_0 \in \Sigma$, for $r>0$ small enough $\Sigma \cap B_r^c(y_0) \neq \emptyset$ and there is a path connecting $y_0$ to the boundary $\partial B_r(y_0)$ of length at least $r$. From the Kuratowski convergence, for $n$ large enough, each set $\Sigma_n$ has a point inside and another outside the ball $\overline{B_r(y_0)}$. 
	
	We start by fixing some $0< \delta < r$ and looking at the smaller ball $B_{r-\delta}(y_0)$. Consider the following class
	\[
		\mathcal{A}_n
		\eqdef 
		\left\{
			\gamma \text{ connected component of $\Sigma_n\cap \overline{B_r(y_0)}$ which intersects $B_{r-\delta}(y_0)$}
		\right\}.
	\]
	Each $\gamma \in \mathcal{A}_n$ must be such that $\H^1(\gamma) \ge \delta$. Indeed, as for each $n \in \mathbb{N}$ there is a point in $\Sigma_n\cap B_r(y_0)^c$ and another in $\gamma \cap \partial B_{r-\delta}(y_0)$, the connectivity implies $\gamma$ is contained in an arc joining these two points, but then it must have length at least $\delta$, as it is the smallest distance between the two balls. So define
	\[
		\tilde\Sigma_n \eqdef 
		\bigcup_{\gamma \in \mathcal{A}_n}
		\gamma,
	\]
	which is a bounded sequence of closed sets, but not necessarily connected. However this sequence has a uniformly bounded number of connected components since 
	\[
		\delta\sharp \mathcal{A}_n 
		\le 
		\sum_{\gamma \in \mathcal{A}_n} \H^1(\gamma) 
		\le \H^1(\Sigma_n \cap B_R(x_0)), \text{ hence }
		\sharp \mathcal{A}_n 
		\le \sup_{n \in \mathbb{N}} \frac{\H^1(\Sigma_n\cap B_R(y_0))}{\delta} 
		< +\infty,
	\]
	for $R>0$ large enough.
	
	As $\tilde \Sigma_n$ is a bounded sequence, by Blaschke's Theorem we can assume up to an extraction that $\tilde \Sigma_n \xrightarrow[n \to \infty]{d_H} \tilde \Sigma$. In fact, for a.e. $0 < \delta < r$, using Lemma \ref{lemma.localized_Kuratowski}, it holds that
	\begin{equation}\label{identity}
		\tilde \Sigma\cap \overline{B_{r-\delta}(y_0)} = \overline{\Sigma\cap B_{r-\delta}(y_0)},
	\end{equation}
	since by the construction, $\tilde \Sigma_n\cap \overline{B_{r-\delta}(y_0)} = \Sigma_n\cap \overline{B_{r-\delta}(y_0)}$ and choosing $\delta$ such that $\Sigma_n\cap \overline{B_{r-\delta}(y_0)} \xrightarrow[n \to \infty]{K} \overline{\Sigma\cap B_{r-\delta}(y_0)}$. 

	This way, we can apply the global version of \Golab's Theorem with a uniformly bounded number of connected components to the sequence $\tilde \Sigma_n \cap B_{r - \delta}(y_0)$ so that we write 
	\begin{align*}
		\mu\left(\overline{B_r(y_0)}\right) 
		&\ge 
		\limsup_{n \to \infty} \H^1\left(\Sigma_n \cap B_r(y_0)\right)
		\ge
		\limsup_{n \to \infty} \H^1\left(\tilde \Sigma_n\right)\\ 
		&\ge 
		\liminf_{n \to \infty} \H^1\left(\tilde \Sigma_n\cap B_{r-\delta}\right)\\
		&\ge 
		\H^1\left(\tilde \Sigma\cap B_{r-\delta}(y_0)\right) 
		= \H^1\left(\overline{\Sigma\cap B_{r-\delta}(y_0)}\right)\\
		&\ge 
		\H^1\left(\Sigma\cap B_{r-\delta}(y_0)\right),
	\end{align*}
	where the first inequality is due to the local weak-$\star$ convergence of the measures and the forth is given by \Golab's Theorem. But as this estimate is true for any $\delta > 0$, it must hold that $\mu\left(\overline{B_r(y_0)}\right)  \ge \H^1\left(\Sigma\cap B_{r}(y_0)\right)$ for any $y_0 \in \Sigma$ and $r > 0$. To extend this to open balls as well we use the following estimates
	\begin{align*}
		\mu(B_r) 
		=
		\lim_{n \to \infty} \mu\left(\overline{B_{r - 1/n}}\right)
		\ge  
		\lim_{n \to \infty} \mathcal{H}^1\left(\Sigma\cap B_{r - 1/n}\right)
		= \mathcal{H}^1\left(\Sigma\cap B_r\right). 
	\end{align*}
\end{proof}

%%%%%%%%%%%%%%%%%%%%%%%%%%%%%%%%%%%%%%%%%%%%%%%%%%%%%%

\bibliographystyle{plain}
\bibliography{references.bib}

\begin{thebibliography}{10}

\bibitem{alberti2017structure}
Giovanni Alberti and Martino Ottolini.
\newblock On the structure of continua with finite length and
  {G}o{\l}{\k{a}}b’s semicontinuity theorem.
\newblock {\em Nonlinear Analysis: Theory, Methods \& Applications},
  153:35--55, 2017.

\bibitem{ambrosio2021lectures}
Luigi Ambrosio, Elia Bru{\'e}, and Daniele Semola.
\newblock Lectures on optimal transport, 2021.

\bibitem{ambrosio2000functions}
Luigi Ambrosio, Nicola Fusco, and Diego Pallara.
\newblock {\em Functions of bounded variation and free discontinuity problems}.
\newblock Courier Corporation, 2000.

\bibitem{ambrosio2008gradient}
Luigi Ambrosio, Nicola Gigli, and Giuseppe Savar{\'e}.
\newblock {\em Gradient flows: in metric spaces and in the space of probability
  measures}.
\newblock Springer Science \& Business Media, 2008.

\bibitem{ambrosio2004topics}
Luigi Ambrosio and Paolo Tilli.
\newblock {\em Topics on analysis in metric spaces}, volume~25.
\newblock Oxford University Press on Demand, 2004.

\bibitem{braides_gamma-convergence_2002}
Andrea Braides.
\newblock {\em Gamma-convergence for beginners}.
\newblock Number~22 in Oxford lecture series in mathematics and its
  applications. Oxford University Press, New York, 2002.

\bibitem{buttazzo2003optimal}
Giuseppe Buttazzo and Eugene Stepanov.
\newblock Optimal transportation networks as free dirichlet regions for the
  monge-kantorovich problem.
\newblock {\em Annali della Scuola Normale Superiore di Pisa-Classe di
  Scienze}, 2(4):631--678, 2003.

\bibitem{chauffert2017projection}
Nicolas Chauffert, Philippe Ciuciu, Jonas Kahn, and Pierre Weiss.
\newblock A projection method on measures sets.
\newblock {\em Constructive Approximation}, 45(1):83--111, 2017.

\bibitem{dal_maso_introduction_1993}
Gianni Dal~Maso.
\newblock {\em An {Introduction} to {$\Gamma$}-convergence}.
\newblock Number v. 8 in Progress in nonlinear differential equations and their
  applications. Birkhäuser, Boston, MA, 1993.

\bibitem{david2006singular}
Guy David.
\newblock {\em Singular sets of minimizers for the Mumford-Shah functional},
  volume 233.
\newblock Springer Science \& Business Media, 2006.

\bibitem{delellis2006lecture}
Camillo De~Lellis.
\newblock Lecture notes on rectifiable sets, densities, and tangent measures.
\newblock {\em Preprint}, 23, 2006.

\bibitem{delattre2020principal}
Sylvain Delattre and Aur{\'e}lie Fischer.
\newblock On principal curves with a length constraint.
\newblock In {\em Annales de l'Institut Henri Poincar{\'e}, Probabilit{\'e}s et
  Statistiques}, volume~56, pages 2108--2140. Institut Henri Poincar{\'e},
  2020.

\bibitem{ehler2021curve}
Martin Ehler, Manuel Gr{\"a}f, Sebastian Neumayer, and Gabriele Steidl.
\newblock Curve based approximation of measures on manifolds by discrepancy
  minimization.
\newblock {\em Foundations of Computational Mathematics}, 21(6):1595--1642,
  2021.

\bibitem{federer2014geometric}
Herbert Federer.
\newblock {\em Geometric measure theory}.
\newblock Springer, 2014.

\bibitem{Golab}
S.~Go{\l}\k{a}b.
\newblock Sur quelques points de la th\'eorie de la longueur.
\newblock {\em Annales de la Soci\'et\'e Polonaise de Math\'ematiques},
  VII:227--241, 1928.
\newblock Rocznik polskiego tow.matematycznego.

\bibitem{hastie1989principal}
Trevor Hastie and Werner Stuetzle.
\newblock Principal curves.
\newblock {\em Journal of the American Statistical Association},
  84(406):502--516, 1989.

\bibitem{kegl2000learning}
Bal{\'a}zs K{\'e}gl, Adam Krzyzak, Tam{\'a}s Linder, and Kenneth Zeger.
\newblock Learning and design of principal curves.
\newblock {\em IEEE transactions on pattern analysis and machine intelligence},
  22(3):281--297, 2000.

\bibitem{lebrat2019optimal}
L{\'e}o Lebrat, Fr{\'e}d{\'e}ric de~Gournay, Jonas Kahn, and Pierre Weiss.
\newblock Optimal transport approximation of 2-dimensional measures.
\newblock {\em SIAM Journal on Imaging Sciences}, 12(2):762--787, 2019.

\bibitem{lemenant2010presentation}
Antoine Lemenant.
\newblock A presentation of the average distance minimizing problem.
\newblock {\em Zap. Nauchn. Sem. S.-Peterburg. Otdel. Mat. Inst. Steklov.
  (POMI)}, 2011.

\bibitem{lu2016average}
Xin~Yang Lu and Dejan Slep{\v{c}}ev.
\newblock Average-distance problem for parameterized curves.
\newblock {\em ESAIM: Control, Optimisation and Calculus of Variations},
  22(2):404--416, 2016.

\bibitem{MachadoThesis}
Jo\~{a}o~Miguel Machado.
\newblock {\em Optimization in spaces of measures: Optimal Transport, Geometric
  Structures and Game Theory}.
\newblock PhD thesis, \'Ecole Doctorale SDOSE - PSL, 2024.

\bibitem{MachadoPhaseField}
Jo\~{a}o~Miguel Machado.
\newblock {Phase-field approximation for 1-dimensional shape optimization
  problems}.
\newblock Preprint HAL-04620380, June 2024.

\bibitem{maggi2012sets}
Francesco Maggi.
\newblock {\em Sets of finite perimeter and geometric variational problems: an
  introduction to Geometric Measure Theory}.
\newblock Number 135. Cambridge University Press, 2012.

\bibitem{morel2012variational}
Jean-Michel Morel and Sergio Solimini.
\newblock {\em Variational methods in image segmentation: with seven image
  processing experiments}, volume~14.
\newblock Springer Science \& Business Media, 2012.

\bibitem{morgan1994m}
Frank Morgan.
\newblock (m, $\varepsilon$, $\delta$)-minimal curve regularity.
\newblock {\em Proceedings of the American Mathematical Society}, pages
  677--686, 1994.

\bibitem{neumayer2021optimal}
Sebastian Neumayer and Gabriele Steidl.
\newblock From optimal transport to discrepancy.
\newblock {\em Handbook of Mathematical Models and Algorithms in Computer
  Vision and Imaging: Mathematical Imaging and Vision}, pages 1--36, 2021.

\bibitem{paolini2004qualitative}
Emanuele Paolini and Eugene Stepanov.
\newblock Qualitative properties of maximum distance minimizers and average
  distance minimizers in {$\mathbb{R}^n$}.
\newblock {\em Journal of Mathematical Sciences}, 122(3):3290--3309, 2004.

\bibitem{paolini2013existence}
Emanuele Paolini and Eugene Stepanov.
\newblock Existence and regularity results for the steiner problem.
\newblock {\em Calculus of Variations and Partial Differential Equations},
  46(3):837--860, 2013.

\bibitem{rockafellar2009variational}
R~Tyrrell Rockafellar and Roger J-B Wets.
\newblock {\em Variational analysis}, volume 317.
\newblock Springer Science \& Business Media, 2009.

\bibitem{santambrogio2015optimal}
Filippo Santambrogio.
\newblock Optimal transport for applied mathematicians.
\newblock {\em Birk{\"a}user, NY}, 55(58-63):94, 2015.

\bibitem{santambrogio2005blow}
Filippo Santambrogio and Paolo Tilli.
\newblock Blow-up of optimal sets in the irrigation problem.
\newblock {\em The Journal of Geometric Analysis}, 15:343--362, 2005.

\bibitem{villani2009optimal}
C{\'e}dric Villani.
\newblock {\em Optimal transport: old and new}, volume 338.
\newblock Springer, 2009.

\end{thebibliography}

%%%%%%
% Compile bibliography with: 
% pdflatex main.tex; bibtex main; pdflatex.tex

\end{document}